\documentclass{amsart}

\usepackage{amssymb}
\usepackage{latexsym}
\usepackage{amsmath}

\def\wt{\widetilde}
\def\wh{\widehat}
\def\ov{\overline}
\def \im{{\rm Im}}
 \def\up{\upharpoonright}
\def\cH{\mathcal H}

\def\cK{\mathcal K} \def\cL{\mathcal L}
\def\cM{\mathcal M} \def\cN{\mathcal N}  
 \def\cT{\mathcal T} \def\cI{\mathcal I}

\def \gH{\mathfrak H}   \def \gN{\mathfrak N}

\def \bC{\mathbb C}    \def\bR{\mathbb R}
\def\bH{\mathbb H} 
\def \l{\lambda}
\def \a{\alpha} \def \b{\beta}    
 \def \t{\theta} \def\g {\gamma}
  
\def \f{\varphi}  \def \G{\Gamma} \def\D {\Delta}

\def \C{\widetilde {\mathcal C}}
\def \CA{\C(\cH_0,\cH_1)}

\def \cd {\cdot}

\def\HH {\cH_0\oplus\cH_1}
\def\AC {AC(\cI; \bH)}   \def\LI {L_\Delta^2(\cI)}
\def\lI {\cL_\Delta^2(\cI)}
\def\tma{\cT_{max}} \def\tmi{\cT_{min}} \def\Tma{T_{max}} \def\Tmi{T_{min}}

\def \dom {{\rm dom}\,}  \def \ran {{\rm ran}\,}  \def \ker{{\rm ker\,}}
 \def \mul {{\rm mul}\,} \def \sign {{\rm sign}\,}

\def \exa { {Ext}_A}
  \def\tm{\times}

\def  \RH {\wt R_\a (\cH_0,\cH_1)}
\def  \RP {\wt R_{+1} (\wt\cH_b,\cH_b)}
\def  \RM {\wt R_{-1} (\wt\cH_b,\cH_b)}

 \def \RZ {\wt R_\a^0 (\cH_0,\cH_1)}
\def  \Rh {\wt R (\cH)} 
\def \pair {\tau=\{\tau_+,\tau_-\}}
\def \CR {\bC\setminus\bR}

\newcommand {\lo}[1] {\cL_\D^2[#1,\bH ]}
\newcommand {\Ca}[1] {\emph {Case #1}}

\def\bt{\{\cH,\G_0,\G_1\}}
\def\bta{\{\cH_0\oplus \cH_1,\Gamma _0,\Gamma _1\}}

\newtheorem{theorem}{Theorem}[section]
\newtheorem{proposition}[theorem]{Proposition}
\newtheorem{corollary}[theorem]{Corollary}
\newtheorem{lemma}[theorem]{Lemma}
\newtheorem{assertion}[theorem]{Assertion}
\theoremstyle{definition}

\theoremstyle{definition}
\newtheorem {definition} [theorem]{Definition}
\theoremstyle{remark}
\newtheorem{remark}[theorem]{Remark}
\numberwithin{equation}{section}
\begin{document}

\title[On Titchmarsh-Weyl functions ]
{On Titchmarsh-Weyl functions  of first-order symmetric systems with arbitrary
deficiency indices}
\author{Sergio Albeverio}
\author{Mark Malamud}
\author {Vadim Mogilevskii}



\begin{abstract}
We study general (not necessarily Hamiltonian) first-order symmetric  systems
$J y'(t)-B(t)y(t)=\D(t) f(t)$ on an interval $[a,b\rangle $ with the regular
endpoint $a$. The deficiency indices $n_\pm$ of the corresponding  minimal
relation $\Tmi$ may be arbitrary (possibly unequal).  Our approach is based on
the concept of a decomposing boundary triplet, which enables one to parametrize
various classes of extensions of $\Tmi$ (self-adjoint, $m$-dissipative, etc.)
in terms of boundary conditions imposed on regular and singular values of a
function $y\in \dom \tma$ at the endpoints $a$ and $b$ respectively. In
particular, we describe self-adjoint and $\l$-depending Nevanlinna boundary
conditions which are analogs  of separated ones for Hamiltonian systems.  With
a boundary value problem involving such conditions we associate the
$m$-function $m(\cd)$, which is an analog of the Titchmarsh-Weyl coefficient
for the Hamiltonian system. In the simplest   case of minimal (unequal)
deficiency indices $n_\pm$ the $m$-function $m(\cd)$ coincides with the
rectangular Titchmarsh-Weyl coefficient introduced by Hinton and Schneider. We
parametrize all $m$-functions in terms of the Nevanlinna boundary parameter at
the endpoint $b$ by means of the formula similar to the known Krein formula for
resolvents. Application of these results to differential operators of an odd
order enables us to complete the results by Everitt and Krishna Kumar on the
Titchmarsh-Weyl theory of such operators.
\end{abstract}
\maketitle
\section{Introduction}
Assume that $H$ and $\wh H$ are finite dimensional Hilbert spaces with $\dim
H=\nu_+$ and $\dim \wh H=\wh\nu$ and let
\begin{gather}\label{1.1}
H_0=H\oplus\wh H, \qquad \bH=H_0\oplus H=H\oplus\wh H \oplus H.
\end{gather}

The main object of the paper is a first-order symmetric  system of differential
equations defined on an interval $\cI=[a,b\rangle, -\infty<a <b\leq\infty,$
with the regular endpoint $a$ and singular, generally speaking, endpoint $b$.
Such a system is of the form \cite{Atk,GK}
\begin {equation}\label{1.2}
J y'(t)-B(t)y(t)=\D(t) f(t), \quad t\in\cI,
\end{equation}
where $B(t)=B^*(t)$ and $\D(t)\geq 0$ are the $[\bH]$-valued functions on $\cI$
and
\begin {equation} \label{1.3}
J=\begin{pmatrix} 0 & 0&-I_H \cr 0& i I_{\wh H}&0\cr I_H&
0&0\end{pmatrix}:H\oplus\wh H\oplus H \to H\oplus\wh H\oplus H.
\end{equation}
We suppose that the system \eqref{1.2} is definite, that is for each $\l\in\bC$
the equalities
\begin {equation}\label{1.5}
J y'(t)-B(t)y(t)=\l \D(t) y(t)
\end{equation}
and $\D(t)y(t)=0$ a.e. on $\cI$ yield $y(t)=0, \; t\in\cI$.

The system \eqref{1.2} is called Hamiltonian if $\wh H=\{0\}$, in which case
\begin {equation}\label{1.5a}
J=\begin{pmatrix}  0&-I_H \cr  I_H& 0\end{pmatrix}:H\oplus H \to H\oplus H.
\end{equation}

Assume that $\lI$ is the semi-Hilbert space of $\bH$-valued functions $f(t)$ on
$\cI$ with $||f||_\D^2:=\int\limits_{\cI}(\D (t)f(t),f(t))_\bH \,dt<\infty$,
$\LI$ is the corresponding Hilbert space of equivalence classes, $\pi$ is a
quotient map from $\lI$ onto $\LI$ and $\wt\pi=\pi\oplus\pi$. Denote also by
$\lo{\cK}$ the set of all operator functions $Y(t)(\in [\cK,\bH])$ on $\cI$
such that $Y(t)h\in \lI$ for each $h\in\cK$ (here $\cK$ is a finite dimensional
Hilbert space).

As is known the extension theory of symmetric linear  relations is the natural
approach to boundary value problems involving symmetric systems (see
\cite{Orc,LT82,DLS88,DLS93,HSW00,Kac03,BHSW10,LesMal03} and references
therein). According to \cite{Orc} the system \eqref{1.2} generates linear
relations $\tmi$ and $\tma$ in $\lI$ and minimal and maximal relations $\Tmi=
\wt\pi\tmi$ and $\Tma= \wt\pi\tma$ in $\LI$. It turns out that  $\Tmi$ is a
closed symmetric relation with not necessarily equal deficiency indices $n_\pm$
and $\Tma=\Tmi^*$. Moreover, the equality
\begin {equation}\label{1.6}
[y,z]_b=\lim _{t\uparrow b} (J y(t), z(t)), \quad y,z\in\dom\tma,
\end{equation}
defines a skew-Hermitian bilinear form on  the domain of $\tma$ with finite
indices of inertia $\nu_{b+}$ and $\nu_{b-}$.

A description of various classes of extensions of $\Tmi$ (self-adjoint,
$m$-dissipative, etc.) in terms of boundary conditions is an important problem
in the spectral theory of symmetric systems. In particular, a boundary value
problem  for the system \eqref{1.2} with self-adjoint separated boundary
conditions generates the Fourier transform with the spectral function of the
minimal dimension. Assume that the system \eqref{1.2} is Hamiltonian,
$n_+=n_-=:n$ and let $y(t)=\{y_0(t),y_1(t)\}(\in H\oplus H)$ be  the
representation of a function  $y\in\dom\tma$. Then according to \cite{HinSch93}
the general form of  self-adjoint separated boundary conditions is
\begin {equation}\label{1.7}
\cos B_1 \, y_0(a)+\sin B_1 \, y_1(a)=0, \quad [y,\chi_j]_b=0, \quad j=1\div
\nu_b,\;\; \; y\in\dom\tma,
\end{equation}
where $B_1=B_1^*\in [H], \; \nu_b=n-\dim H$ and $\chi_1,\;\chi_2, \; \dots,\;
\chi_{\nu_b} $ are linearly independent modulo $\dom\tmi$ functions  from
$\dom\tma$ such that $\chi_j(0)=0$ and $[\chi_j,\chi_k]_b=0,\; j\neq k$. An
element $y_b:=\{[y,\chi_j]_b\}_1^{\nu_b}\in \bC^{\nu_b}$ is called a singular
boundary value   of a function $y\in\dom\tma$. Observe that for differential
operators the notion of a singular boundary value as well as formula
\eqref{1.7} go back to the paper by Calkin \cite{Cal39} (see also
\cite[Ch.13.2] {DunSch}).

Boundary conditions \eqref{1.7} generate a self-adjoint extension $\wt A$ of
$\Tmi$ given by $\wt A=\wt\pi\{\{y,f\}\in\tma:\, y \;\; \text{satisfies
\eqref{1.7}}\}$. The resolvent of $\wt A$ is defined  by  $(\wt A-\l)^{-1}\wt f
= \pi y_f$, where $y_f$ is the $\cL_\D^2$-solution of the boundary problem
involving the system
\begin {equation}\label{1.7a}
J y'(t)-B(t)y(t)=\l\D (t)y+\D(t) f(t), \quad f\in\wt f, \quad \l\in \CR,
\end{equation}
and the boundary conditions \eqref{1.7}. Moreover, according to \cite{HinSch93}
the Titchmarsh - Weyl coefficient $M_{TW}(\l)(\in [H])$ of the boundary problem
\eqref{1.7a}, \eqref{1.7} is defined by the relations
\begin {equation}\label{1.7b}
v(t,\l):=\f(t,\l)M_{TW}(\l)+\psi(t,\l)\in \lo{H} \;\;\;\;\text{and}\;\;\;\;
[v(\cd,\l)h,\chi_j]_b=0, \;\;\;\ h\in H.
\end{equation}
for all $j=1\div\nu_b$. Here $\f(\cd,\l)$ and $\psi(\cd,\l)$ are the
$[H,\bH]$-valued operator solutions of Eq. \eqref{1.5} with the initial data
$\f(a,\l)= (\sin B_1\,:\, -\cos B_1)^\top$ and $\psi (a,\l)= (-\cos B_1\,:\,
\sin B_1)^\top$. Note also the paper \cite{Kra89}, in which  the Titchmarsh -
Weyl coefficient is defined by means of a limiting process from a compact
interval $[a,\b]\subset \cI$. It turns out that $M_{TW}(\cd)$ is a Nevanlinna
operator function, i.e.,  $M_{TW}(\cd)$ is holomorphic on $\CR$ and $\im\l\cd
\im M_{TW}(\l)\geq 0, \; M_{TW}^*(\l)=M_{TW}(\ov\l), \; \l\in\CR$. Moreover,
the spectral function  of $M_{TW}(\cd)$ is a spectral function of the
corresponding Fourier transform with the minimal dimension.

Another approach to description of boundary conditions is based on the concept
of a decomposing boundary triplet for $\Tma$ (see \cite{Mog11} for symmetric
systems and \cite{Mog09.1,Mog09.2,Mog11mz} for differential operators). To
explain this concept note that there exist finite-dimensional Hilbert spaces
$\cH_b$ and $\wh\cH_b$ and a surjective linear map
\begin{gather}\label{1.8}
\G_b=(\G_{0b}:\,  \wh\G_b:\,  \G_{1b})^\top:\dom\tma\to
\cH_b\oplus\wh\cH_b\oplus \cH_b
\end{gather}
such that the bilinear form \eqref{1.6} admits the representation
\begin {equation}\label{1.9}
[y,z]_b=i\cdot\sign (\nu_{b+}-\nu_{b-}) (\wh\G_b y, \wh\G_b
z)-(\G_{1b}y,\G_{0b}z)+(\G_{0b}y,\G_{1b}z).
\end{equation}
Moreover, let $X_a\in [\bH]$ be the operator such that $X_a^* JX_a=J$,  and let
\begin {equation*}
\G_a=\left(\G_{0a}\,:\, \wh\G_a\,:\, \G_{1a}\right)^\top :\AC\to H\oplus \wh
H\oplus H.
\end{equation*}
be the block representation of the linear map $\G_a y=X_a y(a), \; y\in \AC$
(here $\AC$ is the set of all absolutely continuous $\bH$-valued functions on
$\cI$).  By using $\cH_b, \; \wh \cH_b$ and $\G_a, \; \G_b$ one constructs the
Hilbert space $\cH_0$, the subspace $\cH_1$ in $\cH_0$ and the linear maps
$\G_j':\dom\tma\to \cH_j, \; j\in\{0,1\},$ such that the classical Lagrange's
identity takes the form
\begin {equation}\label{1.11}
(f,z)_\D - (y,g)_\D= (\G_1'y, \G_0'z)- (\G_0'y, \G_1'z)+i\,\sign (n_+
-n_-)(P_2\G_0'y, P_2\G_0'z)
\end{equation}
(in \eqref{1.11} $\{y,f\}, \;\{z,g\}\in\tma $ and $P_2$ is the orthoprojector
in $\cH_0$ onto $\cH_2:=\cH_0\ominus\cH_1$). Finally, a decomposing boundary
triplet for $\Tma$ is defined as a collection $\Pi=\bta$, in which
$\G_j:\Tma\to\cH_j, \; j\in\{0,1\},$ are the linear maps given by
\begin{gather}\label{1.12}
\G_0\{\wt y, \wt f\}=\G_0'y, \qquad \G_1\{\wt y, \wt f\}=\G_1'y, \qquad \{\wt
y, \wt f\}\in\Tma.
\end{gather}
In the case of equal deficiency indices $n_+=n_-$ one has
\begin {equation*}
\cH=H_0\oplus\cH_b(:=\cH_0=\cH_1)
\end{equation*}
and  the decomposing boundary triplet takes the form $\Pi=\bt$, where
\begin {gather}
\G_0\{\wt y, \wt f\}= \{- \G_{1a}y +i(\wh\G_a-\wh\G_b)y,\,
\G_{0b}y\} (\in H_0\oplus\cH_b),\label{1.13}\\
\G_1\{\wt y, \wt f\}= \{ \G_{0a}y + \tfrac 1 2(\wh\G_a+\wh\G_b)y,\,
-\G_{1b}y\}(\in H_0\oplus\cH_b), \quad  \{\wt y, \wt f\}\in\Tma .\label{1.14}
\end{gather}
Moreover, for the Hamiltonian system with $n_+=n_-$ one has $\cH=H\oplus\cH_b$
and
\begin {equation}\label{1.15}
\G_0\{\wt y, \wt f\}= \{- \G_{1a}y , \G_{0b}y\} (\in H\oplus\cH_b), \quad
\G_1\{\wt y, \wt f\}= \{ \G_{0a}y, -\G_{1b}y\}(\in H\oplus\cH_b).
\end{equation}
It turns out that $\G_b y $ can be represented as a singular boundary value
$y_b$ of a function $y\in\dom\tma$ (for more details see Remark \ref{rem3.2a}).
Therefore the operators \eqref{1.13} and \eqref{1.14} are defined , in fact, by
means of boundary values of a function $y$ at the endpoints $a$ (regular value)
and $b$ (singular value). At the same time emphasize that a concrete form of
the map $\G_b$ satisfying \eqref{1.9} does not matter, which is suitable  for a
compact representation of boundary conditions. To illustrate this assertion
note that according to \cite{Mog11} self-adjoint separated boundary conditions
exists only for a Hamiltonian system \eqref{1.2} with $n_+=n_-$, in which case
the general form of such conditions is
\begin {gather}
\cos B_1 \, y_0(a)+\sin B_1 \, y_1(a)=0,\label{1.16}\\
\cos B_2\G_{0b} y \, +\sin B_2  \,\G_{1b} y=0,\quad y\in\dom\tma,\label{1.17}
\end{gather}
with self-adjoint operators  $B_1\in [H]$ and $B_2 \in [\cH_b]$. Formulas
\eqref{1.16} and \eqref{1.17} seem to be more convenient than \eqref{1.7},
because they enable one to parametrize regular self-adjoint boundary conditions
\eqref{1.16} (at the point $a$) and singular ones \eqref{1.17} (at the point
$b$) by means of self-adjoint boundary parameters $B_1$ and $B_2$ respectively.

In the present paper we investigate boundary value problems for general (not
necessarily Hamiltonian) symmetric systems \eqref{1.2} with the aid of
decomposing boundary triplets. We do not impose any restrictions on the
deficiency indices $n_\pm$ of $\Tmi$.  To cover the case $n_+\neq n_-$ we
consider the following problems:

-- to find and describe $\l$-depending Nevanlinna (in particular, self-adjoint)
boundary conditions which are analogs of self-adjoint separated boundary
conditions for Hamiltonian systems;

-- to find the operator functions which are analogs of the Titchmarsh-Weyl
coefficient for Hamiltonian systems and describe these functions in terms of
boundary conditions.

We suppose that solution of these problems will give rise to generalized
Fourier transforms for the system \eqref{1.2} with the spectral functions of
the minimally possible dimension. Our investigations are based on a fact that a
decomposing boundary triplet $\Pi=\bta$ is a boundary triplet for $\Tma$ in the
sense of \cite{Mog06.2}; moreover, in the case $n_+=n_-$ a decomposing triplet
$\Pi=\bt$ is a boundary triplet (boundary value space) for $\Tma$ in the sense
of \cite{GorGor,Mal92}. This makes it possible to apply to the systems
\eqref{1.2} the general theory of boundary triplets for abstract symmetric
relations in Hilbert spaces (see \cite{GorGor,DM91,DM00,Mal92,Mog06.2} and
references therein).

Assume for simplicity that  $n_+=n_-$ and let $\Pi=\bt$ be a decomposing
boundary triplet \eqref{1.13}, \eqref{1.14} for $\Tma$. By using the results in
\cite{DM00,Mal92} we show that
\begin{gather}\label {1.18}
T:=\{\{\wt y, \wt f\}\in\Tma: \, \G_{1a}y=0, \;\wh\G_a y=\wh\G_b y,\; \G_{0b}y
=\G_{1b}y=0 \}
\end{gather}
is a symmetric extension of $\Tmi$ and each generalized resolvent $R(\l)$ of
$T$ is defined by $R(\l)\wt f=\pi (y_f(\cd,\l)),  \;\l\in\CR,$ where $f\in\lI,
\; \pi f =\wt f $ and $y_f(\cd,\l)$ is the $\cL_\D^2$-solution of the following
boundary value problem:
\begin{gather}
J y'-B(t)y=\l \D(t)y+\D(t)f(t), \quad t\in\cI,\label{1.19}\\
\G_{1a}y=0, \quad \wh \G_a y= \wh\G_b y,\label{1.20}\\
C_0(\l)\G_{0b}y+C_1(\l)\G_{1b}y=0, \quad \l\in\CR. \label{1.21}
\end{gather}
Here $C_0(\l)(\in [\cH_b])$ and $C_1(\l)(\in [\cH_b])$ are components of a
Nevanlinna operator pair $\tau(\l)=\{(C_0(\l),C_1(\l))\}$, so that \eqref{1.21}
defines a Nevanlinna boundary condition at the singular endpoint $b$. A pair
$\tau=\tau(\l)$ plays a role of a boundary parameter, since $R(\l)$ runs over
the set of generalized resolvents of $T$ when $\tau (\l)$ runs over the set
$\wt R(\cH_b)$ of all Nevanlinna operator pairs. To emphasize this fact we
write $R(\l)=R_\tau(\l)$. Observe also that a particular case of a boundary
parameter $\tau\in\wt R(\cH_b)$ is $\tau (\l)=\{(I,K(\l))\}$, where $K(\l)$ is
a Nevanlinna operator function.

The boundary problem \eqref{1.19}-\eqref{1.21} defines a canonical resolvent
$R_\tau(\l)$ if and only if $\tau$ is a self-adjoint operator pair
$\tau=\{(\cos B,\sin B)\}$ with some $B=B^*\in [\cH_b]$. In this case
\begin {equation}\label{1.21a}
R_\tau(\l)=(\wt A_\tau - \l)^{-1}, \; \l\in\CR,
\end{equation}
where  $\wt A_\tau$ is a self-adjoint extension of $\Tmi$ defined by the
following mixed boundary conditions (c.f. \eqref{1.16} and \eqref{1.17}):
\begin {equation*}
\G_{1a}y=0, \quad \wh \G_a y= \wh\G_b y,\quad \cos B\cd\G_{0b}y+\sin
B\cd\G_{1b}y=0.
\end{equation*}

For each $\l\in\CR$ denote by $\wh\gN_\l(\subset \Tma)$ the subspace of all
$\{\wt y,\wt f\}\in\Tma$ such that $\wt f =\l \wt y$. According to
\cite{DM91,Mal92} one associates with the decomposing boundary triplet
$\Pi=\bt$ for $\Tma$ the $\g$-field $\g(\l)(\in [\cH,\LI])$ and the abstract
Weyl function $M(\l)(\in [\cH])$ defined by
\begin {equation}\label{1.23}
\g(\l)=\pi_1(\G_0 \up\wh\gN_\l)^{-1}, \quad M(\l)h=\G_1 \{\g(\l)h, \l \g(\l)h
\}, \;\; h\in\cH, \;\;\l\in\CR.
\end{equation}
It turns out that the $\g$-field satisfies the equality
\begin {equation*}
(\g(\l)h)(t)=\pi (Z(t,\l)h), \quad h\in\cH, \quad \l\in\CR,
\end{equation*}
with some operator $\cL_\D^2$-solution $Z(\cd,\l)\in \lo{\cH}$ of Eq.
\eqref{1.5}. This fact enables us to show, that for each Nevanlinna boundary
parameter $\tau (\l)=\{(C_0(\l),C_1(\l))\}$ there exists a unique operator
$\cL_\D^2$-solution $v_\tau(\cd,\l)\in\lo{H_0}\; (\l\in\CR)$ of Eq. \eqref{1.5}
satisfying the boundary conditions
\begin{gather}
\G_{1a}(v_\tau(t,\l)h_0)=-P_H h_0,\nonumber\\
i(\wh \G_a - \wh\G_b )(v_\tau(t,\l)h_0)=P_{\wh H} h_0 ,\label{1.24}\\
C_0(\l)\G_{0b}(v_\tau(t,\l)h_0)+C_1(\l)\G_{1b}(v_\tau(t,\l)h_0)=0,  \quad
h_0\in H_0, \quad \l\in\CR\label{1.25}
\end{gather}
(here $P_H$ and $P_{\wh H}$ are the orthoprojectors in $H_0$ onto $H$ and $\wh
H$ respectively). By using the solution $v_\tau(\cd,\l)$ we introduce the
concept of the $m$-function $m_\tau(\cd):\CR\to [H_0]$ corresponding to the
boundary parameter $\tau$ or, equivalently, to the boundary value problem
\eqref{1.19}-\eqref{1.21}. This function is defined by the following statement:

--- for each $\tau(\l) = \{(C_0(\l),C_1(\l))\}\in \wt R(\cH_b)$ there exists a unique
operator function $m_\tau(\l)(\in [H_0])$ such that the operator solution
\begin {equation}\label{1.28}
v_\tau(t,\l):=\f(t,\l)m_\tau(\l) +\psi (t,\l), \quad \l\in\CR,
\end{equation}
of Eq. \eqref{1.5} belongs to $\lo{H_0}$ and satisfies the boundary conditions
\eqref{1.24} and \eqref{1.25}.

Here $\f(\cd,\l)$ and $\psi (\cd,\l)$ are the $[H_0,\bH]$-valued solutions of
Eq. \eqref{1.5} with the initial data
\begin {equation}\label{1.29}
X_a\f(a,\l)=\begin{pmatrix} I_{H_0}\cr 0\end{pmatrix}(\in [H_0, H_0\oplus H]),
\quad X_a\psi(a,\l)=\begin{pmatrix} -\tfrac i 2 P_{\wh H}\cr
-P_H\end{pmatrix}(\in [H_0, H_0\oplus H]).
\end{equation}
The $m$-function $m_\tau(\cd)$ is called canonical if $\tau=\{(\cos B,\sin
B)\}$ is a selfa-adjoint operator pair or, equivalently, if $m_\tau(\cd)$
corresponds to the canonical resolvent \eqref{1.21a}. In this case the boundary
condition \eqref{1.25} can be written as
\begin {equation}\label{1.29a}
\cos B\cd\G_{0b}(v_\tau(t,\l)h_0)+\sin B\cd\G_{1b}(v_\tau(t,\l)h_0)=0,\quad
h_0\in H_0, \quad \l\in\CR.
\end{equation}
It turns out that under the special choice of the maps $\G_{0b}$ and $\G_{1b}$
the condition \eqref{1.29a} takes the form of the second relation in
\eqref{1.7b}. This and \eqref{1.28} imply that in the case of the Hamiltonian
system \eqref{1.2} the canonical $m$-function $m_\tau(\cd)$ coincides with the
Titchmarsh-Weyl coefficient $M_{TW}(\cd)$ in the sense of \cite{HinSch93} (for
more details see Remark \ref{rem6.9}).

We show in the paper that all $m$-functions can be parametrized immediately in
terms of the Nevanlinna boundary parameter $\tau$ by means of the formula
similar to the known Krein formula for resolvents. More precisely the following
theorem holds
\begin{theorem}\label{th1.1}
Let $\Pi=\bt$ be a decomposing boundary triplet  for $\Tma$ and let
\begin {equation}\label{1.30}
M(\l)=\begin{pmatrix} m_0(\l) & M_2(\l) \cr M_3(\l) & M_4(\l)
\end{pmatrix}:H_0\oplus\cH_b\to H_0\oplus\cH_b, \quad \l\in\CR,
\end{equation}
be the block representation of the  Weyl function \eqref{1.23}. Then for every
Nevanlinna boundary parameter $\tau (\l)=\{(C_0(\l), C_1(\l))\}$  the
corresponding $m$-function $m_\tau(\cd)$ is of the form
\begin {equation}\label{1.31}
m_\tau(\l)=m_0(\l)+M_2(\l)(C_0(\l)-C_1(\l)M_4(\l))^{-1}C_1(\l)M_3(\l),
\quad\l\in\CR.
\end{equation}
\end{theorem}
Note that a description of all canonical $m$-functions of a differential
operator in the case of maximal deficiency indices of the minimal operator
 can be found in \cite{Gor66,Ful77,Hol85}; similar result for Hamiltonian
systems was obtained in \cite{HinSha82}. In these papers each canonical
$m$-function $m_\tau(\cd)$ is represented as a certain linear fractional
transformation of a self-adjoint boundary parameter $\tau$. Observe also that
for a differential operator of an even order with arbitrary (possibly unequal)
deficiency indices  a description of $m$-functions in the form \eqref{1.31} was
obtained in \cite{Mog07}.

It turns out that $m_\tau(\cd) $ is a Nevanlinna operator function satisfying
the inequality
\begin {equation}\label{1.32}
(\im \,\l)^{-1}\cd \im\, m_\tau(\l)\geq \int_\cI v_\tau^*(t,\l)\D(t)
v_\tau(t,\l)\, dt, \quad\l\in\CR,
\end{equation}
Moreover, the canonical $m$-function $m_\tau(\cd) $ satisfies the identity
\begin {equation}\label{1.33}
m_\tau(\mu)-m_\tau^*(\l)= (\mu-\ov\l)\int_\cI v_\tau^*(t,\l)\D(t)
v_\tau(t,\mu)\, dt, \quad \mu,\l\in\CR,
\end{equation}
which implies that  for the canonical $m$-function  the inequality \eqref{1.32}
turns into the equality. The identity \eqref{1.33} follows from the fact that
$m_\tau(\cd)$ is the abstract Weyl function of a boundary triplet for some
symmetric extension of $\Tmi$. Note that for the Titchmarsh-Weyl coefficient
$M_{TW}(\cd)$ of the Hamiltonian system the identity \eqref{1.33} was proved in
\cite{HinSch93}.

In the case of minimal equal deficiency indices $n_+=n_-(=\nu_-)$ the extension
$T$ in \eqref{1.18} is self-adjoint and the boundary condition \eqref{1.21}
vanishes. Therefore in this case there exists a unique (canonical) $m$-function
$m(\cd)$ of the problem \eqref{1.19}, \eqref{1.20}, which coincides with the
abstract Weyl function $M(\l)$ (see \eqref{1.23}).

Actually we consider symmetric systems with arbitrary (possibly unequal)
deficiency indices $n_\pm$. To this end we use the decomposing boundary triplet
$\Pi=\bta$ with possibly unequal Hilbert spaces $\cH_0$  and $\cH_1$ (see
\eqref{1.12}), which enables us to obtain the results similar to those
specified above for the case $n_+=n_-$. In particular, we define the
$m$-function $m_\tau(\l)(\in [H_0])$ and describe all the $m$-functions by
means of formulas similar to \eqref{1.31}. It turns that $m_\tau(\cd)$ is a
Nevanlinna function, which in the case $n_+ < n_-$  has the triangular form
\begin {equation}\label{1.34}
m_\tau(\l)=\begin{pmatrix} m_{1,\tau}(\l) & 0 \cr  m_{+,\tau}(\l) & \tfrac i 2
I\end{pmatrix}, \quad \l\in\bC_+.
\end{equation}
Emphasize that for the system \eqref{1.2} with $n_+\neq n_-$ there are no
longer canonical $m$-functions.

The simplest situation is in the case of minimal deficiency indices
$n_\pm=\nu_\pm$ (for not Hamiltonian systems \eqref{1.2} this implies that $n_+
< n_- $). In this case  there exists a unique $m$-function $m(\cd)$, which has
the triangular form
\begin {equation}\label{1.35}
m(\l)=\begin{pmatrix} M(\l) & 0 \cr  N_+(\l) & \tfrac i 2 I_{\wh
H}\end{pmatrix}:H\oplus\wh H \to H\oplus\wh H , \quad \l\in\bC_+.
\end{equation}
Here the entries  $M(\l)$ and $N_+(\l)$ are taken from the block representation
\begin {equation}\label{1.36}
M_+(\l)=(M(\l) \,:\, N_+(\l))^\top :H \to H\oplus\wh H, \quad \l\in\bC_+,
\end{equation}
of the abstract Weyl function $M_+(\cd)$ corresponding to the decomposing
boundary triplet $\Pi$ (see Definition \ref{def2.10}). Note in this connection
that the systems \eqref{1.2} with minimal deficiency indices $n_\pm$ were
studied in the paper by Hinton and Schneider \cite{HinSch06}, where the concept
of the ''rectangular'' Titchmarsh-Weyl coefficient $M_{TW}(\l)(\in
[H,H\oplus\wh H]), \; \l\in\bC_+, $ was introduced. This coefficient is defined
by the relation
\begin {equation}\label{1.37}
\f(t,\l)M_{TW}(\l)+\chi (t,\l)\in \lo{H}, \quad \l\in\bC_+,
\end{equation}
where $\f(t,\l)(\in [H_0,\bH])$ and $\chi(t,\l)(\in [H,\bH])$ are the operator
solutions of Eq. \eqref{1.5} with the initial data
\begin {equation}\label{1.38}
X_a\f(a,\l)=\begin{pmatrix} I_{H_0}\cr 0\end{pmatrix}(\in [H_0, H_0\oplus H]),
\quad X_a\chi(a,\l)=\begin{pmatrix} 0\cr -I_H\end{pmatrix}(\in [H, H_0\oplus
H])
\end{equation}
(c.f. \eqref{1.29}). It is not difficult to prove that the abstract Weyl
function \eqref{1.36} coincides with $M_{TW}(\l)$ (see Remark \ref{rem6.3}).

In the final part of the paper we consider the operators generated by a
differential expression $l[y]$ of an odd order $r=2n+1$ defined on an interval
$\cI=[a, b\rangle $ (see \eqref{6.39}). Such differential operators have been
investigated in the papers by Everitt and Krishna  Kumar
\cite{Eve72,EveKum76.1,EveKum76.2,Kum82}, where the limiting process from the
compact intervals $[a,\b]\subset \cI$ was used for construction of
$(n+1)$-component operator $L^2$-solutions $v(t,\l)$ of the equation $l[y]=\l
y, \; \l\in\CR$. With each solution $v(t,\l)$  the authors associate curtain
boundary conditions and the Titchmarsh-Weyl matrix $M_{TW}(\l)=
(m_{rs}(\l))_{r,s=1}^{k+1}$. These results are not completed; in particular,
they do not enable to define self-adjoint boundary conditions without some
hardly verifiable assumptions even in the case of equal minimally possible
deficiency indices $n_+(L_0)=n_-(L_0)=n+1$ of the minimal operator $L_0$.

Our approach is based on the known fact \cite{KogRof75} that the equation
$l[y]=\l y$ is equivalent to some  symmetric not Hamiltonian system
\eqref{1.5}. This enables us to extend the results obtained for symmetric
systems to differential operators of an odd order with arbitrary deficiency
indices $n_\pm(L_0)$. In particular, we define the $m$-function $m_\tau(\cd)$
of such an operator and describe all $m$-functions immediately in terms of a
Nevanlinna boundary parameter $\tau$.

Note in conclusion that the Green's functions of generalized resolvents
$R_\tau(\l) $ and the generalized Fourier transform for symmetric systems  will
be considered in the forthcoming paper.
\section{Preliminaries}
\subsection{Notations}
The following notations will be used throughout the paper: $\gH$, $\cH$ denote
Hilbert spaces; $[\cH_1,\cH_2]$  is the set of all bounded linear operators
defined on the Hilbert space $\cH_1$ with values in the Hilbert space $\cH_2$;
$[\cH]:=[\cH,\cH]$; $A\up \cL$ is the restriction of an operator $A$ onto the
linear manifold $\cL$; $P_\cL$ is the orthogonal projector in $\gH$ onto the
subspace $\cL\subset\gH$; $\bC_+\,(\bC_-)$ is the upper (lower) half-plane  of
the complex plane.

Recall that a closed linear   relation   from $\cH_0$ to $\cH_1$ is a closed
linear subspace in $\cH_0\oplus\cH_1$. The set of all closed linear relations
from $\cH_0$ to $\cH_1$ (in $\cH$) will be denoted by $\C (\cH_0,\cH_1)$
($\C(\cH)$). A closed linear operator $T$ from $\cH_0$ to $\cH_1$  is
identified  with its graph $\text {gr}\, T\in\CA$.

For a linear relation $T\in\C (\cH_0,\cH_1)$  we denote by $\dom T,\,\ran T,
\,\ker T$ and $\mul T$  the domain, range, kernel and the multivalued part of
$T$ respectively. Recall also that the inverse and adjoint linear relations of
$T$ are the relations $T^{-1}\in\C (\cH_1,\cH_0)$ and $T^*\in\C (\cH_1,\cH_0)$
defined by
\begin{gather}
T^{-1}=\{\{h_1,h_0\}\in\cH_1\oplus\cH_0:\{h_0,h_1\}\in T\}\nonumber\\
T^* = \{\{k_1,k_0\}\in \cH_1\oplus\cH_0:\, (k_0,h_0)-(k_1,h_1)=0, \;
\{h_0,h_1\}\in T\}\label{2.0}.
\end{gather}

  In the
case $T\in\CA$ we write $0\in \rho (T)$ if $\ker T=\{0\}$\ and\ $\ran T=\cH_1$,
or equivalently if $T^{-1}\in [\cH_1,\cH_0]$; $0\in \wh\rho (T)$\ \ if\ \ $\ker
T=\{0\}$\ and\   $\ran T$ is a closed subspace in $\cH_1$. For a linear
relation $T\in \C(\cH)$ we denote by $\rho (T):=\{\l \in \bC:\ 0\in \rho
(T-\l)\}$ and $\wh\rho (T)=\{\l \in \bC:\ 0\in \wh\rho (T-\l)\}$ the resolvent
set and the set of regular type points  of $T$ respectively.

Recall also the following definition.
\begin{definition}\label{def2.0}
A holomorphic operator function $\Phi (\cd):\bC\setminus\bR\to [\cH]$ is called
a Nevanlinna function  if $\im\, \l\cd \im \Phi (\l)\geq 0 $ and $\Phi ^*(\l)=
\Phi (\ov \l), \; \l\in\bC\setminus\bR$.
\end{definition}
\subsection{Holomorphic operator pairs }
Let  $\Lambda$ be an open set in $\bC$, let $\cK,\cH_0,\cH_1$  be Hilbert
spaces and let $C_j(\cd):\Lambda \to [\cH_j,\cK], \; j\in \{0,1\}$ be a pair of
holomorphic operator functions (in short a holomorphic pair). Two such pairs
$C_j(\cd):\Lambda\to [\cH_j,\cK]$ and $C_j'(\cd):\Lambda\to [\cH_j,\cK']$ are
said to be equivalent if there exists a holomorphic isomorphism $\f(\cd):
\Lambda\to [\cK, \cK']$ such that $C_j'(\l)=\f (\l)C_j(\l), \l\in\Lambda,\;j\in
\{0,1\}$. Clearly, the set of all holomorphic pairs splits into disjoint
equivalence classes; moreover, the equality
\begin {equation}\label{2.1}
\tau (\l)=\{(C_0(\l), C_1(\l));\cK\}:=\{\{h_0,h_1\}\in \cH_0\oplus\cH_1:
C_0(\l)h_0+C_1(\l)h_1=0\}
\end{equation}
allows us to identify such a class with the $\CA$-valued function $\tau(\l), \;
\l\in\Lambda $.

In what follows, unless otherwise stated, $\cH_0$ is a Hilbert space, $\cH_1$
is a subspace in $\cH_0$,  $\cH_2:=\cH_0\ominus\cH_1$ and $P_j$ is the
orthoprojector in $\cH_0$ onto $\cH_j,\; j\in\{1,2\}$.

Let $\a\in \{-1,+1\}$. With each linear relation $\t\in\CA$ we associate the
$\tm$-adjoint linear relation $\t_\a^\times\in\CA$ given by
\begin {equation*}
\t_\a^\times=\{\{k_0,k_1\}\in \cH_0\oplus\cH_1:(k_1,h_0)-(k_0,h_1)+i\a (P_2
k_0,P_2 h_0)=0\;\;\text{for all}\;\; \{h_0,h_1\}\in\t\}.
\end{equation*}
It follows from \eqref{2.0} that  in the case $\cH_0=\cH_1=:\cH$ one has
$\t_\a^\times=\t^*$.

Next assume that  $\cK_+$ and $\cK_-$ are auxiliary Hilbert spaces and
\begin {equation}\label{2.2}
\tau_+(\l)=\{(C_0(\l),C_1(\l));\cK_+\}, \;\;\l\in\bC_+; \;\;\;\;
\tau_-(\l)=\{(D_0(\l),D_1(\l));\cK_-\}, \;\;\l\in\bC_-
\end{equation}
are  equivalence classes of the holomorphic pairs
\begin {gather}
(C_0(\l):C_1(\l)):\HH\to\cK_+,\;\;\;\;\l\in\bC_+ \label{2.3}\\
(D_0(\l):D_1(\l)):\HH\to\cK_-,\;\;\;\;\l\in\bC_-.\label{2.3a}
\end{gather}
Assume also that
\begin {equation*}
C_0(\l)=(C_{01}(\l):C_{02}(\l)):\cH_1\oplus\cH_2\to\cK_+ ;\;\;\;
D_0(\l)=(D_{01}(\l):D_{02}(\l)):\cH_1\oplus\cH_2\to\cK_-
\end{equation*}
are the block representations of $C_0(\l)$ and $D_0(\l)$.
\begin{definition}\label{def2.1}
Let as before $\a\in \{-1,+1\}$. A collection $\pair$ of two holomorphic pairs
\eqref{2.2} (more precisely, of the equivalence classes of the corresponding
pairs) belongs to the class $\RH$ if it satisfies the following relations:
\begin{gather}
2\,\im(C_{1}(\l)C_{01}^*(\l))+\a C_{02}(\l)C_{02}^*(\l)\geq 0,\;\;\l\in\bC_+
\label{2.6}\\
\quad 2\,\im(D_1(\l)D_{01}^*(\l))+\a D_{02}(\l)D_{02}^*(\l)\leq
0,\;\;\l\in\bC_- \label{2.6a}\\
C_1(\l)D_{01}^*(\ov\l)-C_{01}(\l)D_{1}^*(\ov\l)+i\a C_{02}(\l)D_{02}^*
(\ov\l)=0, \;\;\; \l\in\bC_+\label{2.7}\\
\text {if}\; \a=+1, \;\text{then} \;\; 0\in \rho (C_0(\l)-iC_1(\l)P_1) \;\;
\text{and} \;\; 0\in\rho (D_{01}(\l)+iD_{1}(\l))\label{2.8}\\
\text {if}\; \a=-1, \;\text {then} \;\; 0\in\rho (C_{01}(\l)-iC_{1}(\l)) \;\;
\text{and} \;\; 0\in \rho (D_{0}(\l)+iD_{1}(\l)P_1).\label{2.8a}
\end{gather}

A collection $\pair\in\RH$ belongs to the class $\RZ$ if for some (and hence
for any) $\l\in\bC_+$ one has
\begin{gather*}
2\,\im(C_{1}(\l)C_{01}^*(\l))+\a C_{02}(\l)C_{02}^*(\l)= 0,\\
 0\in\rho
(C_{01}(\l)+iC_{1}(\l)) \;\;\text{if}\;\; \a=+1\;\; \text{and}\;\; 0\in\rho
(C_{0}(\l)+iC_{1}(\l)P_1) \;\;\text{if}\;\; \a=-1.
\end{gather*}
\end{definition}
The following proposition is immediate from Definition \ref{def2.1} and the
results of \cite{Mog06.1}.
\begin{proposition}\label{pr2.2}
1) If $\pair\in\RH$, then $(-\tau_\pm (\ov\l))_\a^\tm=-\tau_\mp(\l),
\;\l\in\bC_\mp,$ and the following equality holds
\begin {equation}\label{2.8b}
\tau_\mp(\l)=\{\{-h_1-i\a P_2 h_0, -P_1h_0\}:\{h_1,h_0\}\in (\tau_\pm
(\ov\l))^*\}.
\end{equation}

2) Each collection $\pair\in\RH$ given by \eqref{2.2} satisfies the relations
\begin{gather}
\text{if}\;\;\a=+1, \;\;\text {then} \;\;\dim \cK_+=\dim\cH_0\;\; \text{and}
\;\;\dim \cK_-=\dim\cH_1;\label{2.9}\\
\text{if}\;\;\a=-1,\; \;\text {then} \;\;\dim \cK_+=\dim\cH_1\;\; \text{and}
\;\;\dim \cK_-=\dim\cH_0\label{2.9a}
\end{gather}

3) The set $\RZ$ is not empty if and only if $\dim\cH_0= \dim\cH_1$. This
implies that in the case $\dim\cH_0<\infty$ the set $\RZ$ is not empty if and
only if $\cH_0= \cH_1=:\cH$.

4) Each collection $\pair\in\RZ$ can be represented as a constant
\begin {equation}\label{2.10}
\tau_\pm(\l)\equiv \{(C_0,C_1);\cK\}=\t(\in\CA),\quad \l\in\bC_\pm,
\end{equation}
where $C_j\in [\cH_j,\cK],\; j\in\{0,1\}$ and   $(-\t)_\a^\tm =-\t$.
\end{proposition}
Moreover, one can easily prove the following proposition.
\begin{proposition}\label{pr2.3}
If $\dim\cH_0<\infty$, then a collection $\pair$ of two holomorphic pairs
\eqref{2.2} belongs to the class $\RH$ if and only if it satisfies
\eqref{2.6}--\eqref{2.7}, \eqref{2.9}, \eqref{2.9a} and the following relations
\begin {equation}\label{2.11}
\ran (C_0(\l): C_1(\l))=\cK_+,\;\;\l\in\bC_+; \qquad \ran (D_0(\l): D_1(\l))=
\cK_-,\;\;\l\in\bC_-.
\end{equation}
\end{proposition}
\begin{remark}\label{rem2.4}
1) It follows from Proposition \ref{pr2.2}, 2) that for each collection
$\pair\in\RH$ one can put in the representation \eqref{2.2}
$\cK_+=\cH_0,\;\cK_-=\cH_1$ in the case $\a=+1$ and $\cK_+=\cH_1,\;\cK_-=\cH_0$
in the case $\a=-1$.

2) If $\cH_1=\cH_0=:\cH$, then the class $\Rh:=\wt R_\a (\cH,\cH)\; (\a\in
\{-1,+1\})$ coincides with the well-known class of Nevanlinna functions $\tau
(\cd)$ with values in $\C (\cH)$ (see, for instance, \cite{DM00}). In this case
the collection \eqref{2.2} turns into the Nevanlinna pair
\begin {equation}\label{2.19}
\tau(\l)=\{(C_0(\l),C_1(\l));\cH\}, \quad \l\in\CR,
\end{equation}
with $C_0(\l), \; C_1(\l)\in [\cH]$. In view of \eqref{2.6}--\eqref{2.8a} such
a pair is characterized by the relations (cf. \cite[Definition 2.2]{DM00})
\begin{gather}
\im \l\cd \im (C_1(\l)C_0^*(\l))\geq 0,\quad C_1(\l)C_0^*(\ov\l)-
C_0(\l)C_1^*(\ov\l)=0,  \;\; \l\in\CR,\label{2.20}\\
0\in\rho (C_0(\l)-iC_1(\l)), \;\; \l\in\bC_+; \qquad 0\in\rho (C_0(\l)+
iC_1(\l)), \;\; \l\in\bC_-.
\end{gather}
Moreover, the function $\tau(\cd)$ belongs to the class $\wt R^0(\cH):=\wt
R_\a^0(\cH,\cH)$ if and only if it admits the representation in the form of the
constant (cf. \eqref{2.10})
\begin {equation}\label{2.22}
\tau(\l)\equiv \{(C_0,C_1);\cH\}=\t(\in\C (\cH)),\quad \l\in\CR
\end{equation}
with the operators $C_j\in [\cH]$ such that $\im (C_1C_0^*)=0 $ and $0\in\rho
(C_0\pm i C_1)$ (this means that $\t=\t^*$). Observe also that according to
\cite{Rof69} each $\tau\in\wt R^0(\cH)$ admits the normalized representation
\eqref{2.22} with
\begin {equation}\label{2.23}
C_0=\cos B, \qquad C_1=\sin B, \qquad B=B^*\in [\cH].
\end{equation}

Assume now that $n:=\dim \cH<\infty$, $e=\{e_j\}_1^n$ is an orthonormal basis
in $\cH$,  $\tau (\l)=\{(C_0 (\l),C_1(\l)); \cH\}$ is a pair of holomorphic
operator-functions $C_l(\cd):\CR\to [\cH]$ and $C_l(\l)= (c_{kj,l} (\l)
)_{k,j=1}^n $ is the matrix representations of the operator $C_l(\l),\;
l\in\{0,1\},$ in the basis $e$. Then by Proposition \ref{pr2.3} $\tau$ belongs
to the class $\wt R(\cH)$ if and only if the matrices $C_0(\l)$ and $C_1(\l)$
satisfy \eqref{2.20} and the following equality:
\begin {equation*}
\text {rank} \,(C_0(\l):C_1(\l))=n, \quad \l\in\CR.
\end{equation*}
Moreover, the operator pair $\t=\{(C_0,C_1);\cH\}$ belongs to the class $\wt
R^0(\cH)$ if and only if  $\im (C_1C_0^*)=0$ and $\text{rank}\, (C_0:C_1)=n$
(here $C_l=(c_{kj,l} )_{k,j=1}^n$ is the matrix representation of the operator
$C_l, \; l\in \{0,1\},$ in the basis $e$). Note that such a ''matrix''
definition of the classes $\wt R(\cH)$ and $\wt R^0(\cH)$ in the case $\dim
\cH<\infty$ can be found, e.g. in \cite{DLS93,Kov83}
\end{remark}
\subsection{Boundary triplets and Weyl functions}
Let $A$ be a closed  symmetric linear relation in the Hilbert space $\gH$, let
$\gN_\l(A)=\ker (A^*-\l)\; (\l\in\wh\rho (A))$ be a defect subspace of $A$, let
$\wh\gN_\l(A)=\{\{f,\l f\}:\, f\in \gN_\l(A)\}$ and let $n_\pm (A):=\dim
\gN_\l(A)\leq\infty, \; \l\in\bC_\pm$ be deficiency indices of $A$. Denote by
$\exa$ the set of all proper extensions of $A$, i.e., the set of all relations
$\wt A\in \C (\gH)$ such that $A\subset\wt A\subset A^*$.

 Next assume that $\cH_0$ is a Hilbert space,  $\cH_1$ is a subspace
in $\cH_0$ and   $\cH_2:=\cH_0\ominus\cH_1$, so that $\cH_0=\cH_1\oplus\cH_2$.
Denote by $P_j$ the orthoprojector  in $\cH_0$ onto $\cH_j,\; j\in\{1,2\} $.
\begin{definition}\label{def2.5}
Let $\a\in\{-1,+1\}$. A collection $\Pi_\a=\bta$, where $\G_j: A^*\to \cH_j, \;
j\in\{0,1\}$ are linear mappings, is called a boundary triplet for $A^*$, if
the mapping $\G :\wh f\to \{\G_0 \wh f, \G_1 \wh f\}, \wh f\in A^*,$ from $A^*$
into $\cH_0\oplus\cH_1$ is surjective and the following Green's identity
\begin {equation}\label{2.28}
(f',g)-(f,g')=(\G_1  \wh f,\G_0 \wh g)_{\cH_0}- (\G_0 \wh f,\G_1 \wh
g)_{\cH_0}+i\a (P_2\G_0 \wh f,P_2\G_0 \wh g)_{\cH_2}
\end{equation}
 holds for all $\wh
f=\{f,f'\}, \; \wh g=\{g,g'\}\in A^*$.

In the sequel we will also use the notation $\Pi_+$ (resp. $\Pi_-$) instead of
$\Pi_{+1}$ (resp. $\Pi_{-1}$).
\end{definition}
\begin{proposition}\label{pr2.6}
Let $\Pi_\a=\bta$ be a boundary triplet for   $A^*$. Then
\begin {gather}
\dim \cH_1=n_-(A)\leq n_+(A)=\dim \cH_0, \;\;\text{if}\;\; \a=+1;\label{2.29}\\
\dim \cH_1=n_+(A)\leq n_-(A)=\dim \cH_0, \;\;\text{if}\;\; \a=-1.\label{2.30}
\end{gather}
Conversely for any symmetric relation  $A$ with $n_-(A)\leq n_+(A)$ (resp.
$n_+(A)\leq n_-(A)$) there exists a boundary triplet $\Pi_+$ (resp. $\Pi_-$)
for $A^*$.
\end{proposition}
\begin{proposition}\label{pr2.7}
Let $\Pi_\a=\bta$ be a boundary triplet for  $A^*$. Then:

1) $\ker \G_0\cap\ker\G_1=A$ and $\G_j$ is a bounded operator from $A^*$ into
$\cH_j, \;  j\in\{0,1\}$;

 2) The equality
\begin {equation}\label{2.31}
A_0:=\ker \G_0=\{\wh f\in A^*:\G_0 \wh f=0\}
 \end{equation}
defines the maximal symmetric extension $A_0\in\exa$ such that $\bC_+\subset
\rho (A_0) $ in the case $\a=+1$ and  $\bC_-\subset \rho (A_0) $ in the case
$\a=-1$.
\end{proposition}
In the following two propositions we denote by $\pi_1$  the orthoprojector  in
$\gH\oplus\gH$ onto $\gH\oplus \{0\}$.
\begin{proposition}\label{pr2.8}
Let $n_-(A)\leq n_+(A)$ and let $\Pi_+=\bta$ be a boundary triplet for
$A^*$.Then:

1) the operators $\G_0\up \wh \gN_\l (A), \;\l\in\bC_+,$ and $P_1\G_0\up \wh
\gN_z (A),\; z\in\bC_-,$ isomorphically map $\wh\gN_\l (A)$ onto $\cH_0$ and
$\wh\gN_z(A)$ onto $\cH_1$ respectively. Therefore the equalities
\begin{gather}
\g_{+} (\l)=\pi_1(\G_0\up\wh \gN_\l (A))^{-1}, \;\;\l\in\Bbb C_+;\quad \g_{-}
(z)=\pi_1(P_1\G_0\up\wh\gN_z (A))^{-1}, \;\; z\in\Bbb C_-,
\label{2.32}\\
M_{+}(\l)h_0=\G_1\{\g_+(\l)h_0, \l\g_+(\l)h_0\}, \quad h_0\in\cH_0, \quad
\l\in\bC_+\label{2.32a}\\
M_{-}(z)h_1=(\G_1+iP_2\G_0)\{\g_-(z)h_1, z\g_-(z)h_1\}, \quad h_1\in\cH_1,
\quad z\in\bC_- \label{2.32b}
\end{gather}
correctly define the operator functions $\g_{+}(\cdot):\Bbb C_+\to[\cH_0,\gH],
\; \; \g_{-}(\cdot):\Bbb C_-\to[\cH_1,\gH]$ and $M_{+}(\cdot):\bC_+\to
[\cH_0,\cH_1], \;\; M_{-}(\cdot):\bC_-\to [\cH_1,\cH_0]$, which are holomorphic
on their domains. Moreover, the equality $M_+^*(\ov\l)=M_-(\l), \;\l\in\bC_-,$
is valid.

2) assume that
\begin{gather}
M_+(\l)=(M(\l):N_+(\l)):\cH_1\oplus\cH_2\to \cH_1, \quad
\l\in\bC_+\label{2.33}\\
M_-(z)=(M(z):N_-(z))^\top:\cH_1\to\cH_1\oplus\cH_2 , \quad
z\in\bC_-\label{2.34}
\end{gather}
are the block representations of $M_+(\l)$ and $M_-(z)$ respectively and let
\begin{gather}
\cM(\l)=\begin{pmatrix}M(\l) & N_+(\l) \cr 0 & \tfrac i 2 I_{\cH_2}
\end{pmatrix}:\cH_1\oplus\cH_2\to \cH_1\oplus\cH_2, \quad
\l\in\bC_+\label{2.35}\\
\cM(\l)=\begin{pmatrix}M(\l) & 0 \cr  N_-(\l)  & -\tfrac i 2 I_{\cH_2}
\end{pmatrix}:\cH_1\oplus\cH_2\to \cH_1\oplus\cH_2, \quad
\l\in\bC_-.\label{2.36}
\end{gather}
Then $\cM(\cd)$ is a Nevanlinna operator function satisfying the identity
\begin {equation}\label{2.37}
\cM(\mu)-\cM^*(\l)=(\mu-\ov\l)\g_+^*(\l)\g_+(\mu), \quad \mu,\l\in \bC_+ .
\end{equation}
\end{proposition}
Similar statements for the triplet $\Pi_-$ are specified in the following
proposition.
\begin{proposition}\label{pr2.9}
Let $n_+(A)\leq n_-(A)$ and let $\Pi_-=\bta$ be a boundary triplet for
$A^*$.Then:

1) the equalities
\begin{gather}
\g_{+} (\l)=\pi_1(P_1\G_0\up\wh \gN_\l (A))^{-1}, \;\;\l\in\Bbb C_+;\quad
\g_{-} (z)=\pi_1(\G_0\up\wh\gN_z (A))^{-1}, \;\; z\in\Bbb C_-,
\label{2.38}\\
M_{+}(\l)h_1=(\G_1-iP_2\G_0)\{\g_+(\l)h_1, \l\g_+(\l)h_1\}, \quad h_1\in\cH_1,
\quad\l\in\bC_+\label{2.39}\\
M_{-}(z)h_0=\G_1\{\g_-(z)h_0, z\g_-(z)h_0\}, \quad h_0\in\cH_0, \quad z\in\bC_-
\label{2.40}
\end{gather}
correctly define the holomorphic operator functions $\g_{+}(\cdot):\Bbb
C_+\to[\cH_1,\gH], \; \; \g_{-}(\cdot):\Bbb C_-\to[\cH_0,\gH]$ and
$M_{+}(\cdot):\bC_+\to [\cH_1,\cH_0], \;\; M_{-}(\cdot):\bC_-\to
[\cH_0,\cH_1]$.

2) assume that
\begin{gather}
M_+(\l)=(M(\l):N_+(\l))^\top:\cH_1\to\cH_1\oplus\cH_2 , \quad
\l\in\bC_+\label{2.41}\\
M_-(z)=(M(z):N_-(z)):\cH_1\oplus\cH_2\to \cH_1, \quad z\in\bC_-\label{2.42}
\end{gather}
are the block representations of $M_+(\l)$ and $M_-(z)$ respectively and let
\begin{gather}
\cM(\l)=\begin{pmatrix}M(\l) & 0 \cr N_+(\l) & \tfrac i 2 I_{\cH_2}
\end{pmatrix}:\cH_1\oplus\cH_2\to \cH_1\oplus\cH_2, \quad
\l\in\bC_+\label{2.43}\\
\cM(\l)=\begin{pmatrix}M(\l) & N_-(\l) \cr  0  & -\tfrac i 2 I_{\cH_2}
\end{pmatrix}:\cH_1\oplus\cH_2\to \cH_1\oplus\cH_2, \quad
\l\in\bC_-.\label{2.44}
\end{gather}
Then $\cM(\cd)$ is a Nevanlinna operator function satisfying the identity
\begin {equation}\label{2.44a}
\cM(\mu)-\cM^*(\l)=(\mu-\ov\l)\g_-^*(\l)\g_-(\mu), \quad \mu,\l\in \bC_- .
\end{equation}
\end{proposition}
\begin{definition}\label{def2.10}
The operator functions $\g_\pm(\cd)$ and $M_\pm(\cd)$ defined in Propositions
\ref{pr2.8} and \ref{pr2.9} are called the $\g$-fields and the Weyl functions,
respectively, corresponding to the boundary triplet $\Pi_\a$.
\end{definition}
\begin{proposition}\label{pr2.10a}
Let $\Pi_\a=\bta$ be a boundary triplet for $A^*$ and let $\g_\pm(\cd)$ and
$M_\pm(\cd)$ be the corresponding $\g$-fields and Weyl functions respectively.
Moreover, let the spaces $\cH_0$ and $\cH_1$ be decomposed as
\begin {equation*}
\cH_1=\wh\cH\oplus\dot\cH_1, \qquad \cH_0=\wh\cH\oplus\dot\cH_0
\end{equation*}
(so  that $\dot\cH_0=\dot\cH_1\oplus\cH_2$) and let
\begin {equation*}
\G_0=( \wh\G_0 : \dot\G_0)^\top : A^* \to \wh\cH\oplus\dot\cH_0, \qquad \G_1=(
\wh\G_1 : \dot\G_1 )^\top: A^* \to \wh\cH\oplus\dot\cH_1
\end{equation*}
be the block representations of the operators $\G_0$ and $\G_1$. Then:

1)The equality
\begin {equation*}
\wt A=\{\wh f\in A^*: \wh\G_0 \wh f=\dot\G_0\wh f= \dot\G_1\wh f=0\}
\end{equation*}
defines a closed symmetric extension $\wt A\in\exa$ and the adjoint relation
$\wt A^*$ of $\wt A$ is
\begin {equation*}
\wt A^*=\{\wh f\in A^*: \wh\G_0 \wh f=0\}.
\end{equation*}

If in addition $n_\pm (A)<\infty$, then the deficiency indices of $\wt A$ are
$n_\pm (\wt A)=n_\pm (A)-\dim \wh \cH$.

 2) The collection $\dot\Pi_\a=\{\dot\cH_0\oplus\dot\cH_1, \dot \G_0\up \wt
A^*, \dot \G_1\up \wt A^*\}$  is a boundary triplet for $\wt A^*$.

3) The $\g$-fields $\dot\g_\pm(\cd)$ and the Weyl functions $\dot M_\pm (\cd)$
corresponding to $\dot\Pi_\a$ are given by
\begin{gather*}
\dot\g_+(\l)=\g_+(\l)\up \dot\cH_0, \qquad \dot M_+(\l)=P_{\dot\cH_1}M_+(\l)\up
\dot\cH_0, \quad \l\in\bC_+\\
\dot\g_-(\l)=\g_-(\l)\up \dot\cH_1, \qquad \dot M_-(\l)=P_{\dot\cH_0}M_-(\l)\up
\dot\cH_1, \quad \l\in\bC_-
\end{gather*}
in the case $\a=+1$ and by the same formulas with $\dot\cH_1 \;(\dot\cH_0)$ in
place of $\dot\cH_0$ (resp. $\dot\cH_1$) in the case $\a=-1$.
\end{proposition}
We omit the proof of Proposition \ref{pr2.10a}, since it is similar to that of
Proposition 4.1 in \cite{DM00} (see also remark \ref{rem2.13} below).

Recall further the following definition.
\begin{definition}\label{def2.11}
An operator function $R(\cd):\CR\to [\gH]$ is called a generalized resolvent of
a symmetric linear relation $A\in\C (\gH)$ if there exist a Hilbert space $\wt
\gH\supset\gH$ and a self-adjoint linear relation $\wt A\in \C (\wt\gH)$ such
that $A\subset \wt A$ and $R(\l) =P_\gH (\wt A- \l)^{-1}\up \gH, \;\; \l \in
\CR$.

$R(\cd)$ is a canonical resolvent if and only if $\wt\gH=\gH$. In this case
$R(\l) =(\wt A- \l)^{-1},\; \l \in \CR$.
\end{definition}
\begin{theorem}\label{th2.12}
Let $\Pi_\a=\bta$ be a boundary triplet for $A^*$. If $\pair\in\RH$ is a
collection of holomorphic pairs \eqref{2.2}, then for every $g\in\gH$ and
$\l\in\CR$ the abstract boundary value problem
\begin{gather}
\{f,\l f+g\}\in A^*\label{2.45}\\
C_0(\l)\G_0\{f,\l f+g\}-C_1(\l)\G_1\{f,\l f+g\}=0, \quad \l\in\bC_+
\label{2.46}\\
D_0(\l)\G_0\{f,\l f+g\}-D_1(\l)\G_1\{f,\l f+g\}=0, \quad \l\in\bC_-\label{2.47}
\end{gather}
has a unique solution $f=f(g,\l)$ and the equality $R(\l)g:=f(g,\l)$ defines a
generalized resolvent $R(\l)=R_\tau (\l)$ of the relation $A$. Conversely, for
each generalized resolvent $R(\l)$ of $A$ there exists a unique $\tau\in\RH$
such that $R(\l)=R_\tau (\l)$. Moreover, $R_\tau(\l)$ is a canonical resolvent
if and only if $\tau\in\RZ$.


\end{theorem}
In the following corollary we reformulate the statement of Theorem \ref{th2.12}
for   parameters $\tau$ of a special form.
\begin{corollary}\label{cor2.12a}
Assume that $\cH'$ and $\wt\cH_1$ are Hilbert spaces, $\cH_1$ is a subspace in
$\wt\cH_1$ and $\Pi_-=\{(\cH'\oplus\wt\cH_1)\oplus \cH_1,\G_0,\G_1\}$ is a
boundary triplet for $A^*$. If $\pair\in\wt R_{-1}(\wt \cH_1,\cH_1)$ is a
collection \eqref{2.2}, then the direct statement of Theorem \ref{th2.12} holds
with the following boundary conditions in place of  \eqref{2.46} and
\eqref{2.47}:
\begin{gather*}
C_0(\l)P_{\wt\cH_1}\G_0\{f,\l f+g\}-C_1(\l)\G_1\{f,\l f+g\}=0, \quad \l\in\bC_+
\\
 D_0(\l)P_{\wt\cH_1}\G_0\{f,\l f+g\}-D_1(\l)\G_1\{f,\l f+g\}=0,
\quad P_{\cH'}\G_0\{f,\l f+g\}=0,\quad  \l\in\bC_-
\end{gather*}
\end{corollary}
\begin{proof}
Let $\dot\tau_+(\l)=\{(\dot C_0(\l), \dot C_1(\l)), \cK_+\},\;\l\in\bC_+,$ and
$\dot\tau_-(\l)=\{(\dot D_0(\l), \dot D_1(\l)),\cH'\oplus
\cK_-\},\;\l\in\bC_-,$ be holomorphic operator pairs with
\begin{gather*}
\dot C_0(\l)=C_0(\l)P_{\wt\cH_1}\,(\in [\cH'\oplus\wt\cH_1, \cK_+] ),\quad \dot
C_1(\l)=C_1(\l) \,(\in [\cH_1,\cK_+])\\
\dot D_0(\l)=P_{\cH'}+D_0(\l)P_{\wt\cH_1}\,(\in [\cH'\oplus\wt\cH_1,
\cH'\oplus\cK_-] ),\quad \dot D_1(\l)=D_1(\l)) \,(\in [\cH_1,\cH'\oplus\cK_-]).
\end{gather*}
Then the direct calculations show that the operator functions $\dot C_j(\cd)$
and $\dot D_j(\cd), \; j\in\{0,1\},$ satisfy the relations
\eqref{2.6}--\eqref{2.7}, \eqref{2.8a} and hence a collection
$\dot\tau=\{\dot\tau_+, \dot\tau_-\}$ belongs to $\wt
R_{-1}(\cH'\oplus\wt\cH_1, \cH_1)$. Applying now Theorem \ref{th2.12} to
$\dot\tau$ we arrive at the desired statement.
\end{proof}
\begin{remark}\label{rem2.13}
1) For $\a=+1$ definition of the boundary triplet $\Pi_\a=\Pi_+$ and the
corresponding Weyl functions $M_\pm(\cd)$ are given in the paper
\cite{Mog06.2}. Moreover, the proof of Propositions \ref{pr2.6}-\ref{pr2.8} and
Theorem \ref{th2.12} for the triplets $\Pi_+$  is adduced in this paper as well
(for the triplets $\Pi_-$ the proof is similar).

2) If $\cH_0=\cH_1:=\cH$, then   the triplet $\Pi_\a$ turns into the boundary
triplet (boundary value space) $\Pi=\{\cH,\G_0,\G_1\}$ for $A^*$ in the sense
of \cite{GorGor,Mal92}.In this case $n_+(A)=n_-(A)=\dim \cH$, $\,A_0(=\ker \G_0
)$ is a self-adjoint extension of $A$ and according
 to \cite{DM91,Mal92,DM95} the relations
\begin {equation}\label{2.48}
\g(\l)=\pi_1(\G_0\up\wh\gN_\l(A))^{-1}, \qquad \G_1\up\wh \gN_\l(A)=M(\l)\G_0
\up\wh\gN_\l(A), \qquad \l\in\rho (A_0)
\end{equation}
define the $\g$-field $\g(\cd):\rho (A_0)\to [\cH,\gH]$ and the Weyl function
$M(\cd):\rho (A_0)\to [\cH]$ corresponding to the triplet $\Pi$. It follows
from \eqref{2.48} that $\g(\cd)$ and $M(\cd)$ are associated
 with the operator functions  $\g_\pm(\cd)$ and $M_\pm(\cd)$ from Definition
 \ref{def2.10}  via $\g(\l)=\g_{\pm}(\l)$ and $M(\l)=M_{\pm}(\l),
\;\l\in\bC_\pm $. Moreover, for such a triplet the identity \eqref{2.37} takes
the form
\begin {equation}\label{2.49}
M(\mu)-M^*(\l)=(\mu-\ov\l)\g^*(\l)\g(\mu), \quad \mu,\l\in\CR.
\end{equation}
 Observe also that  for the triplet $\Pi=\bt$ all the results
in this subsection were obtained in \cite{DM91,Mal92,DM95,DM00}.

In what follows a boundary triplet $\Pi=\bt$ in the sense  of
\cite{GorGor,Mal92} will be sometimes called an ordinary boundary triplet for
$A^*$.
\end{remark}
\section{Decomposing boundary triplets for symmetric systems}
\subsection{Notations}
Let $\cI=[ a,b\rangle\; (-\infty < a< b\leq\infty)$ be an interval of the real
line (in the case  $b<\infty$ the endpoint $b$ may or may not belong to $\cI$),
let $\bH$ be a finite-dimensional Hilbert space, let $\AC$ be the set of all
functions $f(\cd):\cI\to \bH$ which are absolutely continuous on each segment
$[a,\b]\subset \cI$ and let $AC (\cI):=AC(\cI;\bC)$. Denote also by
$\cL_{loc}^1(\cI; [\bH])$ the set of all Borel operator functions $F(\cd)$
defined almost everywhere on $\cI$ with values in $[\bH]$ and such that
$\int\limits_{[a,\b]}||F(t)||\,dt<\infty$ for each $\b\in \cI$.

Next assume that $\D(\cd)\in \cL_{loc}^1(\cI;[\bH])$ is an operator function
such that $\D(t)\geq 0 $ a.e. on $\cI$ and let $\lI$ be the linear space of all
Borel functions $f(\cd)$ defined almost everywhere on $\cI$ with values in
$\bH$ and such that $\int\limits_{\cI}(\D (t)f(t),f(t))_\bH \,dt<\infty$.
Moreover, for a given finite-dimensional Hilbert space $\cK$ denote by
$\lo{\cK}$ the set of all Borel operator-functions $F(\cd): \cI\to [\cK,\bH]$
such that there exists the integral $\int\limits_{\cI} F^*(t)\D(t)F(t)\,dt$. It
is clear that the latter condition is equivalent to $F(t)h\in \lI$ for each
$h\in\cK$.

As is known \cite{Kac50, DunSch} $\lI$ is a semi-Hilbert space with the
semi-definite inner product $(\cd,\cd)_\D$ and semi-norm $||\cd||_\D$ given by
\begin {equation}\label{3.0}
(f,g)_\D=\int_{\cI}(\D (t)f(t),g(t))_\bH \,dt, \quad
||f||_\D=((f,f)_\D)^{\frac1 2 }, \qquad f,g\in \lI.
\end{equation}
The semi-Hilbert space $\lI$ gives rise to the Hilbert space $\LI=\lI /
\{f\in\lI: ||f||_\D=0\}$, i.e., $\LI$ is the Hilbert space of all equivalence
classes.  The inner product and norm in $\LI$ are defined by
\begin {equation*}
(\wt f, \wt g)=(f,g)_\D, \quad ||\wt f||=(\wt f, \wt f)^{\frac 1 2}=||f||_\D,
\qquad \wt f, \wt g\in\LI,
\end{equation*}
where $f\in\wt f \; (g\in\wt g)$ is any representative of the class $\wt f$
(resp. $\wt g$).

In the sequel we systematically use the quotient map $\pi$ from $\lI$ onto
$\LI$ given by $\pi f=\wt f(\ni f), \; f\in \lI$. Moreover, we let
$\wt\pi=\pi\oplus\pi: (\lI)^2 \to (\LI)^2$, so that $\wt \pi\{f,g\}=\{\wt f,
\wt g\}, \;\; f,g \in \lI$.
\subsection{Symmetric systems}
In this subsection we provide some known results on symmetric systems of
differential equations.

Let as above $\cI=[ a,b\rangle \;(-\infty < a <b\leq\infty )$ be an interval
and let $\bH$ be a Hilbert space with $n:=\dim \bH<\infty$. Moreover, let
$B(\cd), \D (\cd)\in \cL_{loc}^1(\cI; [\bH])$ be operator functions such that
$B(t)=B^*(t)$ and $\D(t)\geq 0$ a.e. on $\cI$ and let $J\in [\bH]$ be a
signature operator ( this means that $J^*=J^{-1}=-J$).

A first-order symmetric  system on an interval $\cI$ (with the regular endpoint
$a$) is a system of differential equations of the form
\begin {equation}\label{3.1}
J y'(t)-B(t)y(t)=\D(t) f(t), \quad t\in\cI,
\end{equation}
where $f(\cd)\in \lI$. Together with \eqref{3.1} we consider also the
homogeneous  system
\begin {equation}\label{3.2}
J y'(t)-B(t)y(t)=\l \D(t) y(t), \quad t\in\cI, \quad \l\in\bC.
\end{equation}
A function $y\in\AC$ is a solution of \eqref{3.1} (resp. \eqref{3.2}) if the
equality \eqref{3.1} (resp. \eqref{3.2} holds a.e. on $\cI$. Moreover, a
function $Y(\cd,\l):\cI\to [\cK,\bH]$ is an operator solution of the equation
\eqref{3.2} if $y(t)=Y(t,\l)h$ is a (vector) solution of this equation for each
$h\in\cK$ (here $\cK$ is a Hilbert space with $\dim\cK<\infty$).

Everywhere below we suppose that the system \eqref{3.1} is definite in the
sense of the following definition.
 \begin{definition}\label{def3.1}$\,$\cite{GK,Orc,KogRof75}
The symmetric system \eqref{3.1} is called definite if for each $\l\in\bC$ and
each solution $y$ of \eqref{3.2} the equality $\D(t)y(t)=0$ (a.e. on $\cI$)
implies $y(t)=0, \; t\in\cI$.
\end{definition}
As is known \cite{Orc} the symmetric system \eqref{3.1} induces the
\emph{maximal relations} $\tma$ in $\lI$ and $\Tma$ in $\LI$, which are defined
by
\begin {equation}\label{3.4}
\tma=\{\{y,f\}\in(\lI)^2 :y\in\AC \;\;\text{and}\;\; J y'(t)-B(t)y(t)=\D(t)
f(t)\;\;\text{a.e. on}\;\; \cI \}
\end{equation}
and $\Tma=\wt\pi \tma$. Moreover the Lagrange's identity
\begin {equation}\label{3.6}
(f,z)_\D-(y,g)_\D=[y,z]_b - (J y(a),z(a)),\quad \{y,f\}, \; \{z,g\} \in\tma.
\end{equation}
holds with
\begin {equation}\label{3.7}
[y,z]_b:=\lim_{t \uparrow b}(J y(t),z(t)), \quad y,z \in\dom\tma.
\end{equation}
Formula \eqref{3.7} defines the boundary bilinear form $[\cd,\cd]_b $ on $\dom
\tma$, which plays an essential role in our considerations. By using this form
we define the \emph{minimal relations} $\tmi$ in $\lI$ and $\Tmi$ in $\LI$ via
\begin {equation}\label{3.10}
\tmi=\{\{y,f\}\in\tma: y(a)=0 \;\; \text{and}\;\; [y,z]_b=0\;\;\text{for
each}\;\; z\in \dom \tma \}.
\end{equation}
and $\Tmi= \wt\pi \tmi$. According to \cite{Orc} $\Tmi$ is a closed symmetric
linear relation in $\LI$ and $\Tmi^*=\Tma$.

For each $\l\in\bC$ denote by $\cN_\l$ the linear space of all solutions of the
homogeneous system \eqref{3.2} belonging to $\lI$. Definition \eqref{3.4} of
$\tma$ implies that
\begin{equation*}
\cN_\l=\ker (\tma-\l)=\{y\in\lI:\; \{y,\l y\}\in\tma\}, \quad\l\in\bC.
\end{equation*}
and hence $\cN_\l\subset \dom\tma$.

Assume that
\begin{gather*}
n_\pm:=n_\pm (\Tmi )=\dim \gN_\l (\Tmi), \quad \l\in\bC_\pm,
\end{gather*}
are deficiency indices of  $\Tmi$. It is easily seen that $ \pi\cN_\l=\gN_\l
(\Tmi)$ and $\ker(\pi\up\cN_\l)=\{0\},\;\; \l\in\bC$.  This implies that $\dim
\cN_\l=n_\pm, \; \l\in\bC_\pm$.

Let $J\in [\bH]$ be the signature operator  in \eqref{3.1} and let
\begin{gather*}
\nu_+=\dim\ker (i J-I) \;\;\; \text{and} \;\;\; \nu_-=\dim\ker (i J+I).
\end{gather*}
In what follows we suppose that
\begin {equation}\label{3.12}
\wh\nu:=\nu_- - \nu_+\geq 0 .
\end{equation}
In this case one can assume without loss of generality that the following
statements hold:

(i) the Hilbert space $\bH$ is of the form
\begin{gather}\label{3.16}
\bH=H\oplus\wh H \oplus H,
\end{gather}
where $H$ and $\wh H$ are finite dimensional Hilbert spaces with
\begin {equation}\label{3.16a}
 \dim H=\nu_+, \qquad  \dim \wh H=\wh\nu;
\end{equation}

 (ii) the operator $J$ is of the form \eqref{1.3}.

Introducing the Hilbert space
\begin {equation} \label{3.17a}
H_0=H\oplus\wh H
\end{equation}
one can represent the equality \eqref{3.16} as
\begin {equation} \label{3.17b}
\bH=(H\oplus\wh H) \oplus H=H_0\oplus H.
\end{equation}

Let $\nu_{b+} $ and $\nu_{b-}$ be  indices of inertia of the skew-Hermitian
bilinear form \eqref{3.7}. Then $\nu_{b\pm}<\infty$ and the following equality
holds \cite{BHSW10,Mog11}
\begin {equation} \label{3.17c}
n_+=\nu_+ +\nu_{b+}, \qquad n_-=\nu_- +\nu_{b-}.
\end{equation}
Therefore $\Tmi$ has  equal deficiency indices $n_+=n_-$ if and only if
\begin {equation} \label{3.17d}
\wh\nu=\nu_{b+}-\nu_{b-}.
\end{equation}
Observe also that  according to \cite [Lemma 5.1]{Mog11} there exist Hilbert
spaces $\cH_b$ and $\wh\cH_b$ and a surjective linear map
\begin{gather}\label{3.18}
\G_b=(\G_{0b}:\,  \wh\G_b:\,  \G_{1b})^\top:\dom\tma\to
\cH_b\oplus\wh\cH_b\oplus \cH_b
\end{gather}
such that for all $y,z \in \dom\tma$ the following equality is valid
\begin{gather}\label{3.19}
[y,z]_b=i\cdot\sign (\nu_{b+}-\nu_{b-}) (\wh\G_b y, \wh\G_b
z)-(\G_{1b}y,\G_{0b}z)+(\G_{0b}y,\G_{1b}z).
\end{gather}
Moreover, for such a map $\G_b$ one has $\ker \G_b=\ker [\cd,\cd]_b$ and
\begin {equation}\label{3.20}
\dim\cH_b=\min\{\nu_{b+},\nu_{b-}\}, \qquad \dim \wh\cH_b=|\nu_{b+}- \nu_{b-}|.
\end{equation}

Recall that the system \eqref{3.1} is called regular if $\cI=[a,b]$ is a
compact interval and both the integrals $\int_\cI ||B(t)||\, dt$ and $\int_\cI
||\D(t)||\, dt$ are finite. For a regular system one can put $\cH_b=H, \; \wh
\cH_b=\wh H$ and $\G_b y= X_b y(b), \; y\in \dom\tma,$ where $X_b\in [\bH]$ and
$X_b^*J X_b=J $.

Next assume that $X_a\in [\bH]$ is the operator such that $X_a^* JX_a=J$ and
let $\G_a: \AC\to \bH$ be the linear map given by
\begin {equation}\label{3.21}
\G_ay =X_a y(a), \quad y\in \AC.
\end{equation}
In accordance with the decomposition \eqref{3.16} $\G_a$ admits the block
representation
\begin {equation}\label{3.22}
\G_a=\left(\G_{0a}\,:\, \wh\G_a\,:\, \G_{1a}\right)^\top :\AC\to H\oplus \wh
H\oplus H.
\end{equation}
The particular case of the operator $X_a$ is (cf. \cite{HinSch06})
\begin {equation}\label{3.23}
X_a=\begin{pmatrix} X_{00} & 0 & X_{01} \cr 0& I & 0 \cr  X_{10} & 0 & X_{11}
\end{pmatrix}:H\oplus\wh H \oplus H\to H\oplus\wh H \oplus H,
\end{equation}
where the entries $X_{jk}$ satisfy
\begin {equation*}
\im (X_{00}X_{01}^*)=0, \quad \im (X_{10}X_{11}^*)=0, \quad
-X_{10}X_{01}^*+X_{11}X_{00}^*=I_H.
\end{equation*}
If $X_a$ is given by \eqref{3.23} and the function $y\in\AC$ is decomposed as
\begin {equation*}
y(t)=\{y_0(t),\,\wh y(t), \, y_1(t) \}(\in H\oplus\wh H \oplus H), \quad
t\in\cI,
\end{equation*}
then in the representation \eqref{3.22} one has
\begin {equation}\label{3.24}
\G_{0a}y=X_{00}y_0(a)+X_{01}y_1(a), \quad \wh\G_{a}y=\wh y(a), \quad
\G_{1a}y=X_{10}y_0(a)+X_{11}y_1(a).
\end{equation}

Let $\l\in\bC$. By using the operator $X_a$ we associate with each operator
solution $Y(\cd,\l):\cI\to [\cK,\bH]$ of the equation \eqref{3.2} the operator
$Y_a (\l)\in [\cK,\bH]$ given by
\begin {equation}\label{3.26}
Y_a(\l)=X_a Y(a,\l)
\end{equation}
(recall that here $\cK$ is a finite-dimensional  Hilbert space).
\begin{lemma}\label{lem3.2}
1) If $Y(\cd,\l)\in \lo{\cK}$ is an operator solution  of Eq.  \eqref{3.2},
then the relation
\begin {equation}\label{3.30}
\cK\ni h\to (Y(\l) h)(t)=Y(t,\l)h \in\cN_\l.
\end{equation}
defines the linear map $Y(\l):\cK\to \cN_\l$ and,conversely, for each such a
map $Y(\l)$ there exists a unique operator solution $Y(\cd,\l)\in \lo{\cK}$ of
Eq. \eqref{3.2} such that \eqref{3.30} holds.

2) Let $Y(\cd,\l)\in \lo{\cK}$ be an operator solution of Eq. \eqref{3.2} and
let $F(\l)=\pi Y(\l)(\in [\cK, \LI])$. Then for each $\wt f\in \LI$
\begin {equation}\label{3.31}
F^*(\l)\wt f=\int_\cI Y^*(t,\l)\D(t)f(t)\, dt, \quad f\in\wt f.
\end{equation}
\end{lemma}
The first statement of this lemma is obvious, while the second one can be
proved in the same way as formula (3.70) in \cite{Mog09.1}.

Clearly, for each solution $Y(\cd,\l) \in \lo{\cK}$ of Eq.  \eqref{3.2} the
operator \eqref{3.26} admits the representation
\begin {equation}\label{3.32}
Y_{a}(\l)=\G_{a} Y(\l),
\end{equation}
where $Y(\l)$ is defined in Lemma \ref{lem3.2}.
\begin{remark}\label{rem3.2a}
According to \cite[Remark 5.2]{Mog11} one can construct the map $\G_b$ by using
the following assertion:

--- there exist systems  of functions
$\{\psi_j\}_1^{\nu_b},\;\{\f_j\}_1^{\hat\nu_b} $ and $\{\t_j\}_1^{\nu_b}$ in
$\dom\tma$ with $\nu_b=\text{min} \{\nu_{b+}, \nu_{b-}\}$ and
$\hat\nu_b=|\nu_{b+}- \nu_{b-}|$ such that  the operators
\begin {equation}\label{3.33}
\G_{0b}y =\{[y,\psi_j ]_b\}_1^{\nu_b}, \quad   \hat\G_b y=\{[y,\f_j
]_b\}_1^{\hat\nu_b}, \quad \G_{1b}y =\{[y,\t_j ]_b\}_1^{\nu_b}, \quad
y\in\dom\tma
\end{equation}
form the surjective linear map $\G_b=(\G_{0b}:\,  \hat\G_b:\,
\G_{1b})^\top:\dom\tma\to \bC^{\nu_b}\oplus\bC^{\hat \nu_b}\oplus \bC^{\nu_b}$
satisfying the equality \eqref{3.19}.

This assertion  shows that $\G_b y$ is, in fact, a singular boundary value of a
function $y\in\dom\tma$ (c.f. \cite[Ch. 13.2]{DunSch}).
\end{remark}
\subsection{Decomposing boundary triplets}\label{sub2.3}
As is known (see for instance \cite{LesMal03}) the maximal relation $\Tma$
induced by the definite symmetric system \eqref{3.1} possesses the following
property: for each $\{\wt y, \wt f \}\in \Tma $ there exist a unique function
$y\in \AC \cap \lI $ such that $y\in \wt y$ and $\{y,f\}\in \tma$ for each
$f\in\wt f$. Below, without any additional comments, we associate such a
function $ y\in \AC \cap \lI$ with each pair $\{\wt y, \wt f\}\in\Tma$.

Let as before $\G_b$ and $\G_a$ be the operators \eqref{3.18} and \eqref{3.22}
respectively and let $H_0$ be the Hilbert \eqref{3.17a}.  Consider the
following three alternative cases:

 \underline {\emph {Case 1}}:
$\;\nu_{b+}-\nu_{b-}\geq \nu_- -\nu_+ \geq 0$.

It follows from \eqref{3.20} that in this case
\begin {gather}\label{3.34}
\dim \cH_b=\nu_{b_-}, \qquad \dim\wh\cH_b=\nu_{b_+}-\nu_{b_-}
\end{gather}
and \eqref{3.16a} gives $\dim\wh\cH_b\geq  \dim \wh H $. Therefore without loss
of generality we can assume that $\wh H\subset \wh \cH_b$ and hence
\begin {gather}\label{3.34.0}
\wh \cH_b= \cH_2'\oplus\wh H
\end{gather}
with $\cH_2'=\wh \cH_b\ominus\wh H$. Let $\wt\cH_b=\cH_b\oplus\cH_2'$ (so that
$\cH_b\subset\wt\cH_b$) and let
\begin {gather}\label{3.34a}
\wt\G_{0b}=\G_{0b}+P_{\cH_2'}\wh\G_b:\dom\tma\to \wt\cH_b
\end{gather}
In {\emph {Case 1}} we put
\begin {gather}
\cH_0=H_0 \oplus\wt\cH_b, \qquad \cH_1=H_0\oplus \cH_b,\label{3.34b}\\
\begin{array}{l}\label{3.35}
\G_0'=\begin{pmatrix} - \G_{1a} +i(\wh\G_a-P_{\wh H}\wh\G_b)\cr \wt \G_{0b}
\end{pmatrix}:\dom\tma \to H_0\oplus\wt\cH_b,\\
\G_1'=\begin{pmatrix} \G_{0a} + \tfrac 1 2(\wh\G_a+P_{\wh H}\wh\G_b)\cr
-\G_{1b}
\end{pmatrix}:\dom\tma \to H_0\oplus\cH_b,
\end{array}
\end{gather}
If in addition $n_+=n_-$, then in view of \eqref{3.17d} and \eqref{3.20}  $\wh
\cH_b=\wh H$ and $\cH_2'=\{0\}$. Therefore
\begin {equation}\label{3.35.0}
\wt\cH_b =\cH_b, \qquad \wt\G_{0b}=\G_{0b}
\end{equation}
and the equalities \eqref{3.34b} and \eqref{3.35} take the form
\begin {gather}
\cH=H_0 \oplus\cH_b(:=\cH_0=\cH_1), \label{3.35.1}\\
\begin{array}{l}\label{3.35.2}
\G_0'= (- \G_{1a} +i(\wh\G_a-\wh\G_b)\;:\;  \G_{0b})^\top:\dom\tma
\to H_0\oplus\cH_b,\\
\G_1'=(\G_{0a} + \tfrac 1 2(\wh\G_a+\wh\G_b)\;:\; -\G_{1b})^\top :\dom\tma \to
H_0\oplus\cH_b.
\end{array}
\end{gather}

\underline {\emph {Case 2}}: $\;\nu_--\nu_+> \nu_{b+}-\nu_{b-}> 0,$

\noindent so that  the equalities \eqref{3.34} holds. It follows from
\eqref{3.16a} that in this case $\dim \wh H > \dim\wh\cH_b$. Therefore one may
assume that $\wh\cH_b\subset \wh H$ and hence $\wh H=\wh\cH_b\oplus\cH_2'$ with
$\cH_2'=\wh H\ominus\wh\cH_b $. This implies that the Hilbert space
\eqref{3.17a} admits the representation
\begin {equation}\label{3.36}
H_0=H\oplus\underbrace {\wh\cH_b\oplus\cH_2'}_{\wh H}=H_0'\oplus \cH_2',
\end{equation}
where
\begin {equation}\label{3.37}
H_0'=H\oplus \wh\cH_b.
\end{equation}

In \emph{Case 2} we let
\begin{gather}
\cH_0=H_0'\oplus\cH_2'\oplus\cH_b, \quad \cH_1=H_0'\oplus\cH_b,\label{3.37a}\\
\begin{array}{l}\label{3.38}
\G_0'=\begin{pmatrix} - \G_{1a} +i(P_{\wh\cH_b}\wh\G_a-\wh\G_b)\cr i
P_{\cH_2'}\wh\G_a\cr \G_{0b}\end{pmatrix}:\dom\tma \to H_0'\oplus\cH_2'
\oplus\cH_b,\\
\G_1'=\begin{pmatrix} \G_{0a} + \tfrac 1 2(P_{\wh\cH_b}\wh\G_a+\wh\G_b)\cr
-\G_{1b}
\end{pmatrix}:\dom\tma \to H_0'\oplus\cH_b.
\end{array}
\end{gather}

\underline {\emph {Case 3}}: $\;\wh\nu\geq 0 \geq \nu_{b+}-\nu_{b-}$ and
$\wh\nu\neq \nu_{b+}-\nu_{b-}(\neq 0)$,

so that in view of \eqref{3.20}
\begin {equation}\label{3.39a}
\dim\cH_b=\nu_{b+}, \qquad \dim\wh\cH_b=\nu_{b-}-\nu_{b+}.
\end{equation}
Let $\wt\cH_b:=\cH_b\oplus \wh\cH_b$ (so that $\cH_b\subset\wt\cH_b $) and let
$\wt\G_{0b}:\dom\tma\to\wt\cH_b$ be the linear map given by
\begin {equation}\label{3.39b}
\wt\G_{0b}=\G_{0b}+ \wh \G_b.
\end{equation}
In \emph{Case 3} we put
\begin {gather}
\cH_0=H\oplus \wh H \oplus \wt \cH_b=H_0\oplus \wt \cH_b, \qquad \cH_1=H \oplus \cH_b,\label{3.39c}\\
\G_0'=\begin{pmatrix} - \G_{1a} \cr i \wh\G_a\cr
\wt\G_{0b}\end{pmatrix}:\dom\tma \to H\oplus\wh H\oplus \wt\cH_b,\;\;\;
\G_1'=\begin{pmatrix} \G_{0a} \cr -\G_{1b}
\end{pmatrix}:\dom\tma \to H\oplus\cH_b.\label{3.40}
\end{gather}

Note that for evrey system \eqref{3.1} one (and only one) of \emph{Cases 1--3}
holds. In each of these cases $\cH_1$ is a subspace in $\cH_0$ and $\G_j'$ is a
linear map from $\dom\tma$ to $\cH_j, \; j\in \{0.1\}$. Moreover, the subspace
$\cH_2=\cH_0\ominus \cH_1$ coincides with $\cH_2'$ in \emph{Cases 1--2} and
$\cH_2=\wh H\oplus\wh \cH_b$ in \emph{Case 3}. Observe also that according to
\cite[Proposition 5.5]{Mog11} the deficiency indices of $\Tmi$ are
$n_\pm=\nu_\pm+\nu_{b\pm}$. Therefore $n_-\leq n_+$ in  \emph{Case 1} and
$n_+<n_-$ in  \emph{Cases 2} and \emph{3}. Moreover, formulas \eqref{3.16a},
\eqref{3.34} and \eqref{3.39a} imply that in all  \emph{Cases 1--3}
\begin {equation}\label{3.41a}
\dim\cH_0 + \dim\cH_1=\nu_+ +\nu_- +\nu_{b+}+ \nu_{b-}=n_+ + n_-.
\end{equation}
\begin{proposition}\label{pr3.3}
Let $\cH_j$ be Hilbert spaces and $\G_j':\dom\tma\to\cH_j, \; j\in\{0,1\},$ be
linear mappings constructed for the alternative {Cases 1--3} before the
proposition and let $\G_j:\Tma\to\cH_j, \; j\in \{0,1\},$ be the operators
given by
\begin{gather}\label{3.42}
\G_0\{\wt y, \wt f\}=\G_0'y, \qquad \G_1\{\wt y, \wt f\}=\G_1'y, \qquad \{\wt
y, \wt f\}\in\Tma.
\end{gather}
Then the collection $\Pi_\a=\bta$ with $\a=+1$ in  {Case 1} and $\a=-1$ in
{Cases 2} and {3} is a boundary triplet for $\Tma$.

 If in addition $n_+=n_-$
(so that Case 1 holds), then $\Pi_+$ turns into an ordinary boundary triplet
$\Pi=\bt$ for $\Tma$, where $\cH$ is the Hilbert space \eqref{3.35.1} and
$\G_j:\Tma\to\cH, \; j\in\{0,1\},$ are the operators given by \eqref{3.42} and
\eqref{3.35.2}.
\end{proposition}
\begin{proof}
The immediate calculations with taking \eqref{3.19} into account show that in
each of the \emph{Cases 1--3} the operators $\G_0'$ and $\G_1'$ satisfy the
relation
\begin{gather*}
[y,z]_b- (J y(a), z(a))=(\G_1'y, \G_0'z)-(\G_0'y, \G_1'z)+ i\a (P_2\G_0'y,
P_2\G_0'z), \quad  y,z\in\dom\tma.
\end{gather*}
This  and the Lagrange's identity \eqref{3.6} give the identity \eqref{2.28}
for the operators $\G_0$ and $\G_1$ defined by \eqref{3.42}. To prove
surjectivity of the mapping $\G=(\G_0:\G_1)^\top$ note that $\ker\G_0'\cap
\ker\G_1'=\ker\G_a\cap \ker\G_b=\dom\tmi$. Hence
$\ker\G(=\ker\G_0\cap\ker\G_1)=\Tmi$ and by  using \eqref{3.41a} one obtains
\begin {equation*}
\dim (\dom\G /\ker \G)=\dim (\Tma / \Tmi)=n_+ +n_-= \dim (\HH).
\end{equation*}
This implies that $\ran\G=\HH$ and, consequently, $\Pi_\a$ is a boundary
triplet for $\Tma$.

The latter statement of the proposition follows from reasonings before formula
\eqref{3.35.1}.
\end{proof}
\begin{definition}\label{def3.4}
The boundary triplet $\Pi_\a=\bta$  constructed in Proposition \ref{pr3.3} will
be called a decomposing  boundary triplet for $\Tma$.
\end{definition}
\begin{remark}\label{rem3.5}
 In the paper \cite{Mog11} decomposing boundary triplets $\Pi_+$  were
constructed for the maximal relations $\Tma $ satisfying the condition $n_-\leq
n_+$. In   \emph{Case 1} such a triplet coincides with the triplet $\Pi_+$
introduced in Proposition \ref{pr3.3}.
\end{remark}
Combining Propositions  \ref{pr3.3} and \ref{pr2.10a} we arrive at the
following three propositions.
\begin{proposition}\label{pr3.6}
Let in Case 1 $\Pi_+=\bta$ be a decomposing boundary triplet \eqref{3.35},
\eqref{3.42} for $\Tma$. Then:

1) The equalities
\begin{gather}
T=\{\{\wt y, \wt f\}\in\Tma: \, \G_{1a}y=0, \;\wh\G_a y=P_{\wh H}\wh\G_b y,\;
\wt\G_{0b}y =\G_{1b}y=0 \}\label {3.45}\\
T^*=\{\{\wt y, \wt f\}\in\Tma: \, \G_{1a}y=0, \;\wh\G_a y=P_{\wh H}\wh\G_b y
\}\label {3.45a}
\end{gather}
define a symmetric extension $T$ of $\Tmi$ and its adjoint $T^*$. Moreover, the
deficiency indices of $T$ are $ n_+(T)=\nu_{b+}-\wh\nu$ and $ n_-(T)=\nu_{b-}$.

2) The collection $\dot\Pi_+=\{\wt\cH_b\oplus \cH_b, \dot\G_0,\dot\G_1\}$ with
the operators
\begin {equation}\label {3.46}
\dot\G_0 \{\wt y,\wt f\}=\wt\G_{0b}y, \qquad \dot\G_1 \{\wt y,\wt
f\}=-\G_{1b}y, \qquad \{\wt y,\wt f\}\in T^*,
\end{equation}
is a boundary triplet for $T^*$ and the (maximal symmetric)  relation
$A_0(=\ker\dot\G_0)$ is of the form
\begin {equation}\label {3.47}
A_0=\{\{\wt y, \wt f\}\in\Tma: \, \G_{1a}y=0, \;\wh\G_a y=P_{\wh H}\wh\G_b y,\;
\wt\G_{0b}y =0 \}.
\end{equation}

If in addition $n_+=n_-$ and $\Pi=\bt$ is an ordinary decomposing boundary
triplet \eqref{3.35.2}, \eqref{3.42} for $\Tma$, then the equality \eqref{3.45}
 take the form
\begin{gather}
T=\{\{\wt y, \wt f\}\in\Tma: \, \G_{1a}y=0, \;\wh\G_a y=\wh\G_b y,\; \G_{0b}y
=\G_{1b}y=0 \}\label {3.48}
\end{gather}
and $ n_+(T)=n_-(T)=\nu_{b-}$. Moreover, in this case $A_0=A_0^*$ and
\begin {equation}\label {3.49}
A_0=\{\{\wt y, \wt f\}\in\Tma: \, \G_{1a}y=0, \;\wh\G_a y=\wh\G_b y,\; \G_{0b}y
=0 \}.
\end{equation}

\end{proposition}
\begin{proposition}\label{pr3.7}
Let Case 2 holds and let  $\Pi_-=\bta$ be a decomposing boundary triplet
\eqref{3.38}, \eqref{3.42} for $\Tma$. Then:

1)Statement 1) of Proposition \ref{pr3.6} holds with
\begin{gather}
T=\{\{\wt y, \wt f\}\in\Tma: \, \G_{1a}y=0, \;\wh\G_a y=\wh\G_b y,\;
\G_{0b}y =\G_{1b}y=0 \}\label {3.50}\\
T^*=\{\{\wt y, \wt f\}\in\Tma: \, \G_{1a}y=0, \;P_{\wh \cH_b}\wh\G_a y=\wh\G_b
y \}.\label {3.51}
\end{gather}
Moreover, the deficiency indices of $T$ are $n_+(T)=\nu_{b-}$ and
$n_-(T)=\wh\nu +2\nu_{b-}-\nu_{b+}$.

2) The collection $\dot\Pi_-=\{(\cH_2'\oplus \cH_b)\oplus \cH_b, \dot\G_0,
\dot\G_1\}$ with the operators
\begin {equation}\label {3.52}
\dot\G_0 \{\wt y,\wt f\}=\{iP_{\cH_2'}\wh \G_a y,\G_{0b}y\}(\in
\cH_2'\oplus\cH_b), \quad \dot\G_1 \{\wt y,\wt f\}=-\G_{1b}y(\in \cH_b), \quad
\{\wt y,\wt f\}\in T^*,
\end{equation}
is a boundary triplet for $T^*$ and   $A_0(=\ker\dot\G_0)$ is of the form
\begin {equation}\label {3.53}
A_0=\{\{\wt y, \wt f\}\in\Tma: \, \G_{1a}y=0, \;\wh\G_a y=\wh\G_b y,\;
 \G_{0b}y =0 \}.
\end{equation}
\end{proposition}

\begin{proposition}\label{pr3.8}
Let in Case 3   $\Pi_-=\bta$ be a decomposing boundary triplet \eqref{3.40},
\eqref{3.42} for $\Tma$. Then:

1)  Statement 1) of Proposition \ref{pr3.6} holds with
\begin{gather}
T=\{\{\wt y, \wt f\}\in\Tma: \, \G_{1a}y=0, \;\wh\G_a y=0,\;
\wt\G_{0b}y =\G_{1b}y=0 \}\label {3.56}\\
T^*=\{\{\wt y, \wt f\}\in\Tma: \, \G_{1a}y=0 \}.\label {3.57}
\end{gather}
Moreover, the deficiency indices of $T$ are $n_+(T)=\nu_{b+}$ and
$n_-(T)=\wh\nu +\nu_{b-}$.

2) The collection $\dot\Pi_-=\{(\wh H\oplus \wt\cH_b)\oplus \cH_b, \dot\G_0,
\dot\G_1\}$ with the operators
\begin {equation}\label {3.57a}
\dot\G_0 \{\wt y,\wt f\}=\{i\wh \G_a y,\wt\G_{0b}y\}(\in \wh H\oplus\wt\cH_b),
\quad \dot\G_1 \{\wt y,\wt f\}=-\G_{1b}y(\in \cH_b), \quad \{\wt y,\wt f\}\in
T^*,
\end{equation}
is a boundary triplet for $T^*$ and  $A_0(=\ker\dot\G_0)$ is of the form
\begin {equation}\label {3.58}
A_0=\{\{\wt y, \wt f\}\in\Tma: \, \G_{1a}y=0, \;\wh\G_a y=0,\; \; \wt\G_{0b}y
=0 \}.
\end{equation}
\end{proposition}
\section{$\cL_\D^2$-solutions of boundary value problems}
\subsection{Case 1}
Assume that in \emph{Case 1} $\Pi_+=\bta$ is a decomposing boundary triplet
\eqref{3.35}, \eqref{3.42} for $\Tma$ and $\pair\in\RP$ is a collection of
holomorphic pairs \eqref{2.2}. For a given $f\in\lI$ consider the following
boundary value problem:
\begin{gather}
J y'-B(t)y=\l \D(t)y+\D(t)f(t), \quad t\in\cI,\label{4.0.1}\\
\G_{1a}y=0, \quad \wh \G_a y=P_{\wh H} \wh\G_b y,\quad C_0(\l)\wt\G_{0b}y
+C_1(\l)\G_{1b}y=0, \quad \l\in\bC_+, \label{4.0.2}\\
\G_{1a}y=0, \quad \wh \G_a y=P_{\wh H} \wh\G_b y,\quad D_0(\l)\wt\G_{0b}y
+D_1(\l)\G_{1b}y=0, \quad \l\in\bC_-. \label{4.0.3}
\end{gather}
A function $y(\cd,\cd):\cI\tm (\CR)\to\bH$ is called a solution of this problem
if for each $\l\in\CR$ the function $y(\cd,\l)$ belongs to $\AC\cap\lI$ and
satisfies the equation \eqref{4.0.1} a.e. on $\cI$ (so that $y\in\dom\tma$) and
the boundary conditions \eqref{4.0.2}, \eqref{4.0.3}.

Application of Theorem \ref{th2.12} to the boundary triplet \eqref{3.46} yields
the following theorem.

\begin{theorem}\label{th4.0.1}
Let in Case 1 $T$ be a symmetric relation in $\LI$ defined by \eqref{3.45}. If
$\pair\in\RP$ is a collection \eqref{2.2}, then for every $f\in\lI$ the
boundary problem \eqref{4.0.1} - \eqref{4.0.3} has a unique solution
$y(t,\l)=y_f(t,\l) $ and the equality
\begin {equation}\label {4.0.4}
R(\l)\wt f = \pi(y_f(\cd,\l)), \quad \wt f\in \LI, \quad f\in\wt f, \quad
\l\in\CR,
\end{equation}
defines a generalized resolvent $R(\l)=:R_\tau(\l)$ of $T$. Conversely, for
each generalized resolvent $R(\l)$ of $T$ there exists a unique $\tau\in\RP$
such that $R(\l)=R_\tau(\l)$.
\end{theorem}
If $n_+=n_-$, then \eqref{3.35.0} is valid. This and Theorem \ref{th4.0.1}
yield the following corollary.
\begin{corollary}\label{cor4.0.2}
Let $n_+=n_-$, let $\Pi=\bt$ be a decomposing boundary triplet \eqref{3.35.2},
\eqref{3.42} for $\Tma$ and let $T$ be a symmetric relation \eqref{3.48}. Then
the statements of Theorem \ref{th4.0.1} hold with the Nevanlinna operator pairs
$\tau\in\wt R(\cH_b)$ in the form \eqref{2.19} and the following boundary
conditions in place of \eqref{4.0.2} and  \eqref{4.0.3}:
\begin {equation}\label {4.0.5}
\G_{1a}y=0, \quad \wh \G_a y= \wh\G_b y,\quad C_0(\l)\G_{0b}y
+C_1(\l)\G_{1b}y=0, \quad \l\in\CR.
\end{equation}
In this case $R_\tau(\l)$ is a canonical resolvent of $T$ if and only if
$\tau\in\wt R^0(\cH_b)$.
\end{corollary}
\begin{remark}\label{rem4.0.3}
Let in Theorem \ref{th4.0.1} $\tau_0=\{\tau_{+},\tau_{-} \}\in\wt
R_{+1}(\wt\cH_b,\cH_b)$ be defined by \eqref{2.2} with
\begin {equation}\label{4.0.6}
C_0(\l) \equiv I_{\wt\cH_b}, \quad C_1(\l)\equiv 0 \;\;\;\text{and}\;\; \;
D_0(\l) \equiv P_{\cH_b}(\in [\wt\cH_b,\cH_b]), \quad D_1(\l)\equiv 0
\end{equation}
and let $R_0(\l)=R_{\tau_0}(\l)$ be the corresponding generalized resolvent of
$T$. Then
\begin {equation*}
R_0(\l)=(A_0-\l)^{-1},\;\l\in\bC_+ \;\;\text{and}\;\;
R_0(\l)=(A_0^*-\l)^{-1},\;\l\in\bC_-,
\end{equation*}
where $A_0$ is given by \eqref{3.47}.

Similarly, let in Corollary \ref{cor4.0.2} $\tau_0=\{(I_{\cH_b},0);\cH_b\}\in
\wt R^0(\cH_b)$. Then $R_0(\l):=R_{\tau_0}(\l)=(A_0-\l)^{-1}$, where $A_0$ is
the selfadjoint extension \eqref{3.49}.
\end{remark}
\begin{proposition}\label{pr4.1}
Let in Case 1 $\Pi_+=\bta$ be a decomposing boundary triplet \eqref{3.35},
\eqref{3.42} for $\Tma$, let $\g_\pm(\cd)$  be the corresponding $\g$-fields
 and let
\begin{gather}
M_+(\l)=\begin{pmatrix} m_0(\l )& M_{2+}(\l) \cr M_{3+}(\l) & M_{4+}(\l)
\end{pmatrix}: H_0\oplus\wt\cH_b\to H_0\oplus\cH_b,
\;\;\;\l\in\bC_+\label{4.1}\\
M_-(\l)=\begin{pmatrix} m_0(\l )& M_{2-}(\l) \cr M_{3-}(\l) & M_{4-}(\l)
\end{pmatrix}: H_0\oplus\cH_b\to H_0\oplus \wt\cH_b,
\;\;\;\l\in\bC_-\label{4.2}
\end{gather}
be the block representations of the corresponding Weyl functiions. Then:

1) For every $\l\in\CR$ there exists an operator solution
$v_0(\cd,\l)\in\lo{H_0}$ of Eq. \eqref{3.2} such that
\begin{gather}
\G_{1a} v_0(\l)=-P_H, \quad \l\in\CR,\label{4.3}\\
(\G_{0a}+\wh\G_a)v_0(\l)=m_0(\l)-\tfrac i 2 P_{\wh H}, \quad
\l\in\CR,\label{4.3a}\\
i(\wh\G_a-P_{\wh  H}\wh\G_b)v_0(\l)=P_{\wh  H}, \quad\l\in\CR,\label {4.4}\\
\wt\G_{0b} v_0(\l)=0, \qquad \G_{1b}v_0(\l)=-M_{3+}(\l),\quad  \l\in\bC_+,
\label {4.5}\\
\wt\G_{0b} v_0(\l)=-i P_{\cH_2'}M_{3-}(\l), \qquad
\G_{1b}v_0(\l)=-P_{\cH_b}M_{3-}(\l),\quad \l\in\bC_-.\label {4.6}
\end{gather}

2) For every $\l\in\bC_+ \; (\l\in\bC_-)$ there exists a solution
$u_+(\cd,\l)\in\lo{\wt\cH_b}$ (resp. $u_-(\cd,\l)\in\lo{\cH_b}$) such that
\begin{gather}
\G_{1a}u_\pm(\l)=0, \quad \l\in\bC_\pm,\label{4.7}\\
(\G_{0a}+\wh\G_a)u_\pm(\l)=M_{2\pm}(\l),\quad \l\in\bC_\pm, \label{4.7a}\\
i(\wh\G_a-P_{\wh  H}\wh\G_b)u_\pm(\l)=0,\quad \l\in\bC_\pm,\label{4.8}\\
\wt \G_{0b}u_+(\l)=I_{\wt\cH_b}, \qquad \G_{1b}u_+(\l)=-M_{4+}(\l), \quad
\l\in\bC_+,\label{4.9}\\
\wt \G_{0b}u_-(\l)=I_{\cH_b}-i P_{\cH_2'}M_{4-}(\l), \qquad \G_{1b}
u_-(\l)=-P_{\cH_b}M_{4-}(\l), \quad \l\in\bC_-. \label{4.10}
\end{gather}
In formulas \eqref{4.3}-- \eqref{4.10} $v_0(\l)$ and $u_\pm(\l)$ are linear
maps from Lemma \ref{lem3.2} corresponding to the solutions $v_0(\cd,\l)$ and
$u_{\pm}(\cd,\l)$ respectively.

3) The solutions $v_0(\cd,\l)$ and $u_\pm(\cd,\l)$ are connected with
$\g$-fields $\g_\pm(\cd)$ by
\begin {gather}
\g_\pm(\l)\up H_0=\pi v_0(\l),\;\;\;\;\l\in\bC_\pm;\label{4.10a}\\
\quad \g_+(\l)\up \wt\cH_b=\pi u_+(\l), \;\;\;\l\in\bC_+;\qquad \g_-(\l)\up
\cH_b=\pi u_-(\l), \;\;\;\l\in\bC_-.\label{4.10b}
\end{gather}
\end{proposition}
\begin{proof}
Let $\g_\pm(\cd)$ be the $\g$-fields \eqref{2.32} of the triplet $\Pi_+$. Since
the quotient mapping $\pi$ isomorphically maps $\cN_\l$ onto $\gN_\l (\Tmi)$,
it follows that for every $\l\in\bC_+\;(\l\in\bC_-)$ there exists an
isomorphism $Z_+(\l):\cH_0\to \cN_\l$ (resp. $Z_-(\l):\cH_1\to \cN_\l$) such
that
\begin {equation}\label{4.11}
\g_+(\l)=\pi Z_+(\l), \;\;\;\l\in\bC_+; \qquad \g_-(\l)=\pi Z_-(\l),
\;\;\;\l\in\bC_-.
\end{equation}
Combining of \eqref{4.11} with \eqref{2.32} - \eqref{2.32b}  and the obvious
equality $\G_j\{\pi y, \l \pi y\}=\G_j' y, \; y\in\cN_\l, \; j \in \{0,1\},$
gives
\begin{gather}
\G_0' Z_+(\l)=I_{\cH_0}, \qquad \G_1' Z_+(\l)=M_+(\l), \quad
\l\in\bC_+,\label{4.12}\\
P_{\cH_1}\G_0' Z_-(\l)=I_{\cH_1}, \qquad (\G_1'+i P_{\cH_2'}\G_0')
Z_-(\l)=M_-(\l), \quad \l\in\bC_-,\label{4.13}
\end{gather}
which in view of \eqref{3.35} can be written as
\begin{gather}
\begin{pmatrix}-\G_{1a}+i (\wh\G_a- P_{\wh H}\wh\G_b) \cr \wt\G_{0b}
\end{pmatrix} Z_+(\l)=\begin{pmatrix} I_{H_0} & 0 \cr 0 & I_{\wt\cH_b}
\end{pmatrix}, \;\; \l\in\bC_+ \label{4.14}\\
\begin{pmatrix}\G_{0a}+\tfrac 1 2  (\wh\G_a+ P_{\wh H}\wh\G_b) \cr -\G_{1b}
\end{pmatrix} Z_+(\l)=\begin{pmatrix} m_0(\l )& M_{2+}(\l) \cr M_{3+}(\l) & M_{4+}
(\l)
\end{pmatrix}, \;\; \l\in\bC_+ \label{4.15}\\
\begin{pmatrix}-\G_{1a}+i (\wh\G_a- P_{\wh H}\wh\G_b) \cr \G_{0b}
\end{pmatrix} Z_-(\l)=\begin{pmatrix} I_{H_0} & 0 \cr 0 & I_{\cH_b}
\end{pmatrix}, \;\; \l\in\bC_- \label{4.16}\\
\begin{pmatrix}\G_{0a}+\tfrac 1 2  (\wh\G_a+ P_{\hat H}\wh\G_b) \cr
-\G_{1b}+iP_{\cH_2'}\wh\G_b
\end{pmatrix} Z_-(\l)=\begin{pmatrix} m_0(\l )& M_{2-}(\l) \cr M_{3-}(\l) & M_{4-}(\l)
\end{pmatrix}, \;\; \l\in\bC_-. \label{4.17}
\end{gather}
It follows from \eqref{4.14}--\eqref{4.17} that
\begin{gather}
\G_{1a}Z_\pm(\l)=(-P_H\,:\, 0), \;\;\; \tfrac 1 2 (\wh\G_a- P_{\hat H}\wh\G_b)
Z_\pm(\l)=(-\tfrac i 2 P_{\wh H}\;:\; 0), \;\;\;\l\in\bC_\pm \label{4.18}\\
\G_{0a}Z_\pm(\l)=(P_H m_0(\l)\,:\, P_H M_{2\pm}(\l)), \;\;\;\l\in\bC_\pm
\label{4.19} \\
\tfrac 1 2 (\wh\G_a+ P_{\wh H}\wh\G_b) Z_\pm(\l)=(P_{\wh H} m_0(\l)\,:\, P_{\wh
H} M_{2\pm}(\l)),\; \;\;\l\in\bC_\pm. \label{4.19a}
\end{gather}
Summing up the second equality in \eqref{4.18} with  \eqref{4.19} and
\eqref{4.19a} one obtains
\begin {equation}\label{4.20}
(\G_{0a}+\hat\G_a)Z_{\pm}(\l)=(m_0(\l)-\tfrac i 2 P_{\wh H}:M_{2\pm}(\l)),
\;\;\;\l\in\bC_\pm.
\end{equation}
Moreover, \eqref{4.14}-\eqref{4.17} yield
\begin{gather}
\wt \G_{0b} Z_+(\l)= (0\,:\, I_{\wt \cH_b}), \qquad \G_{1b}
Z_+(\l)=(-M_{3+}(\l)
\,:\, -M_{4+}(\l), \quad \l\in\bC_+,\label{4.22}\\
\G_{1b} Z_-(\l)=(-P_{\cH_b}M_{3-}(\l) \,:\, - P_{\cH_b}M_{4-}(\l), \quad
\l\in\bC_-,\label{4.23}\\
\G_{0b} Z_-(\l)=(0\,:\, I_{\cH_b}), \quad P_{\cH_2'}\wh\G_b Z_-(\l)=(-i
P_{\cH_2'} M_{3-}(\l)\,:\,-i P_{\cH_2'}M_{4-}(\l) ), \quad \l\in\bC_- \nonumber
\end{gather}
and in view of \eqref{3.34a} one has
\begin {equation}\label{4.25}
\wt \G_{0b} Z_-(\l)=(-i P_{\cH_2'} M_{3-}(\l)\,:\, I_{\cH_b}-i
P_{\cH_2'}M_{4-}(\l)),\quad\l\in\bC_-.
\end{equation}
Assume now that the block representations of $Z_\pm(\l)$ are
\begin {gather}
Z_+(\l)=(v_0(\l):u_+(\l)):H_0\oplus\wt\cH_b\to
\cN_\l,\;\;\l\in\bC_+\label{4.26}\\
 Z_-(\l)=(v_0(\l):u_-(\l)):H_0\oplus \cH_b\to
\cN_\l,\;\;\l\in\bC_-\label{4.26a}
\end{gather}
and let $v_0(\cd,\l)\in \lo{H_0}, \; u_+(\cd,\l)\in \lo{\wt\cH_b} $ and
$u_-(\cd,\l)\in \lo{\cH_b} $ be the operator solutions of Eq. \eqref{3.2}
corresponding to $v_0(\l), \; u_+(\l)$ and $u_-(\l)$ respectively (see Lemma
\ref{lem3.2}). Then  the representations \eqref{4.26} and \eqref{4.26a}
together with \eqref{4.18}, \eqref{4.20} and \eqref{4.22} - \eqref{4.25} yield
the relations \eqref{4.3}-\eqref{4.10} for $v_0(\cd,\l)$ and $u_\pm (\cd,\l)$.

Finally, \eqref{4.10a} and \eqref{4.10b} follow from \eqref{4.11} and
\eqref{4.26}, \eqref{4.26a}.
\end{proof}
\begin{theorem}\label{th4.2}
Let the assumptions of Proposition \ref{pr4.1} be satisfied and let $\pair\in
\RP$ be a collection of operator pairs \eqref{2.2}.
 Then:

 1) For each $\l\in\CR$ there exists a unique operator solution
$v_\tau(\cd,\l)\in\lo{H_0}$ of Eq.  \eqref{3.2} satisfying the boundary
conditions
\begin {gather}
\G_{1a}v_\tau(\l)=-P_H, \quad \l\in\CR,\label {4.27}\\
i(\wh\G_a-P_{\wh H}\wh\G_b)v_\tau(\l)=P_{\wh H},\quad \l\in\CR,\label {4.27a}\\
C_0(\l)\wt\G_{0b}v_\tau(\l)+C_1(\l)\G_{1b}v_\tau(\l)=0,
\;\;\l\in\bC_+,\label {4.28} \\
D_0(\l)\wt\G_{0b}v_\tau(\l)+D_1(\l)\G_{1b}v_\tau(\l)=0, \;\;\l\in\bC_-\label
{4.28a}
\end{gather}
(here $P_H$ and $P_{\wh H}$ are the orthoprojectors  in  $H_0$ onto $H$ and
$\wh H$ respectively).

2) $v_\tau(\cd,\l)$ is connected with the  solutions $v_0(\cd,\l)$ and
$u_\pm(\cd,\l)$  from  Proposition \ref{pr4.1} by
\begin {gather}
v_\tau(t,\l)=v_0(t,\l)-u_+(t,\l)(\tau_+(\l)+M_{4+}(\l))^{-1}M_{3+}(\l), \quad
\l\in\bC_+ \label{4.29}\\
v_\tau(t,\l)=v_0(t,\l)-u_-(t,\l)(\tau_+^*(\ov\l)+M_{4-}(\l))^{-1}M_{3-}(\l),
\quad \l\in\bC_- \label{4.30}.
\end{gather}

If in addition $n_+=n_-$ and $\Pi=\bt$ is a decomposing boundary triplet
\eqref{3.35.2}, \eqref{3.42} for $\Tma$, then $\tau\in \wt R(\cH_b)$ is given
by \eqref{2.19} and the boundary conditions \eqref{4.27}-\eqref{4.28a} take the
form
\begin {equation*}
\G_{1a}v_\tau(\l)=-P_H, \;\; i(\wh\G_a-\wh\G_b)v_\tau(\l)=P_{\wh H},\;\;
C_0(\l)\G_{0b}v_\tau(\l)+C_1(\l)\G_{1b}v_\tau(\l)=0, \;\;\l\in\CR.
\end{equation*}
\end{theorem}
\begin{proof}
Since in view of Proposition \ref{pr3.6}, 2) $M_{4\pm}(\cd)$ are the Weyl
functions of the boundary triplet $\dot\Pi_+$, it follows from \cite{Mog06.2}
that $0\in\rho (\tau_+(\l)+M_{4+}(\l)), \;\l\in\bC_+,$ and $0\in\rho
(\tau_+^*(\ov\l)+M_{4-}(\l)), \;\l\in\bC_-$. Therefore for each $\l\in\CR$ the
equalities \eqref{4.29} and \eqref{4.30} correctly define the solution
$v_\tau(\cd,\l)\in\lo{H_0}$ of Eq. \eqref{3.2}. Let us show that this solution
satisfies \eqref{4.27}--\eqref{4.28a}.

 Combining \eqref{4.29} and \eqref{4.30}
with  \eqref{4.3}, \eqref{4.4} and \eqref{4.7}, \eqref{4.8} one gets the
equalities \eqref{4.27} and \eqref{4.27a}. To prove \eqref{4.28} and
\eqref{4.28a} we let $T_+(\l)=(\tau_+(\l)+M_{4+}(\l))^{-1},\;\l\in\bC_+,$ and
$T_-(\l)=(\tau_+^*(\ov\l)+M_{4-}(\l))^{-1},\;\;\l\in\bC_-$. Then
\begin {equation}\label{4.31.1}
\tau_+(\l)=\{\{T_+(\l)h, (I-M_{4+}(\l)T_+(\l))h\}:h\in\cH_b\}
\end{equation}
and $\tau_+^*(\ov\l)=\{\{T_-(\l)h, h-M_{4-}(\l)T_-(\l))h\}:h\in\wt\cH_b\}$,
which in view of \eqref{2.8b} yields
\begin {gather}
\tau_-(\l)=\{\{(-T_-(\l)-iP_{\cH_2'}+iP_{\cH_2'}M_{4-}(\l)T_-(\l))h ,\qquad
\qquad\qquad \qquad
\qquad\qquad\label{4.31.2}\\
\qquad\qquad \qquad\qquad\qquad \qquad\qquad \qquad\qquad(-P_{\cH_b} +
P_{\cH_b}M_{4-}(\l)T_-(\l))h\}:h\in\wt\cH_b\}.\nonumber
\end{gather}
Moreover, the relations \eqref{4.29} and \eqref{4.30} with taking \eqref{4.5},
\eqref{4.6}, \eqref{4.9} and \eqref{4.10} into account give
\begin {gather*}
\wt\G_{0b}v_\tau{\l}=-T_+(\l)M_{3+}(\l), \quad
\G_{1b}v_\tau{\l}=-(I-M_{4+}(\l)T_+(\l))M_{3+}(\l), \quad \l\in\bC_+,\\
\wt\G_{0b}v_\tau{\l}=(-iP_{\cH_2'}-T_-(\l)+iP_{\cH_2'}M_{4-}(\l)T_-(\l))
M_{3-}(\l),\;\;\;\l\in\bC_-,\\
\G_{1b}v_\tau{\l}=(-P_{\cH_b} + P_{\cH_b}M_{4-}(\l)T_-(\l))M_{3-}(\l),
\;\;\;\l\in\bC_-
\end{gather*}
Hence by \eqref{4.31.1} and \eqref{4.31.2} one has
\begin {equation}\label{4.31.5}
\{\wt\G_{0b} v_\tau(\l)h_0, \G_{1b}v_\tau(\l)h_0\}\in\tau_\pm(\l), \;\;\;
h_0\in H_0, \;\;\l\in\bC_\pm,
\end{equation}
 which in view of the equalities \eqref{2.2} yields \eqref{4.28} and
\eqref{4.28a}.

Next assume that $v_1(\cd,\l)\in\lo{H_0}$ and $v_2(\cd,\l)\in\lo{H_0}$ are the
operator solutions of Eq. \eqref{3.2} satisfying \eqref{4.27}--\eqref{4.28a}
and let $v(t,\l)=v_1(t,\l)-v_2(t,\l)$. Then for each $h_0\in H_0$ the function
$y=v(t,\l)h_0$ is a solution of the homogenous boundary problem
\eqref{4.0.1}--\eqref{4.0.3} (with $f=0$). Since by Theorem \ref{th4.0.1} such
a problem has a unique solution $y=0$, it follows that $v(t,\l)=0$. This proves
the uniqueness of $v_\tau(\cd,\l)$.
\end{proof}
\subsection{Case 2}
Applying Corollary \ref{cor2.12a} to the boundary triplet \eqref{3.52} we
obtain the following theorem.
\begin{theorem}\label{th4.2a}
Let in Case 2 $\Pi_-=\bta$ be a decomposing boundary triplet \eqref{3.38},
\eqref{3.42} for $\Tma$, let $T$ be a symmetric relation in $\LI$ defined by
\eqref{3.50} and let $\tau\in\wt R(\cH_b)$ be a Nevanlinna operator pair
\eqref{2.19}. Then  for every $f\in\lI$ the boundary value problem
\begin{gather}
J y'-B(t)y=\l \D(t)y+\D(t)f(t), \quad t\in\cI,\label{4.32.1}\\
\G_{1a}y=0, \quad P_{\wh\cH_b}\wh \G_a y= \wh\G_b y,\quad C_0(\l) \G_{0b}y
+C_1(\l)\G_{1b}y=0, \quad \l\in\bC_+, \label{4.32.2}\\
\G_{1a}y=0, \quad \wh \G_a y= \wh\G_b y,\quad C_0(\l)\G_{0b}y
+C_1(\l)\G_{1b}y=0, \quad \l\in\bC_-. \label{4.32.3}
\end{gather}
has a unique solution $y(t,\l)=y_f(t,\l) $ (in the same sense as the problem
\eqref{4.0.1}--\eqref{4.0.3}) and the equality \eqref{4.0.4} gives a
generalized resolvent $R(\l)=:R_\tau(\l)$ of $T$.
\end{theorem}
\begin{remark}\label{rem4.2b}
Let in Theorem \ref{th4.2a} $\tau_0=\{(I_{\cH_b},0);\cH_b\}\in \wt R^0 (\cH_b)$
and let $A_0$ be the symmetric extension \eqref{3.53}. Then
$R_0(\l):=R_{\tau_0}(\l)$ is if the form
\begin {equation}\label{4.33.0}
R_0(\l)=(A_0^*-\l)^{-1},\;\l\in\bC_+ \;\;\text{and}\;\;
R_0(\l)=(A_0-\l)^{-1},\;\l\in\bC_-,
\end{equation}
\end{remark}
\begin{proposition}\label{pr4.3}
Assume that in Case 2  $\Pi_-=\bta$ is a decomposing boundary triplet
\eqref{3.38}, \eqref{3.42} for $\Tma$,  $\g_\pm(\cd)$  are the $\g$-fields of
$\Pi_-$ and
\begin{gather}
M_+(\l)=\begin{pmatrix} M_1(\l )& M_2(\l) \cr N_{1+}(\l) &  N_{2+}(\l) \cr
M_{3}(\l) & M_{4}(\l)\end{pmatrix}: H_0'\oplus\cH_b\to
H_0'\oplus\cH_2'\oplus\cH_b,
\;\;\;\l\in\bC_+\label{4.33.1}\\
M_-(\l)=\begin{pmatrix} M_1(\l )& N_{1-}(\l) & M_{2}(\l) \cr M_{3}(\l) &
N_{2-}(\l) & M_{4}(\l)
\end{pmatrix}: H_0'\oplus \cH_2'\oplus\cH_b\to H_0'\oplus \cH_b,
\;\;\;\l\in\bC_-\label{4.33.2}
\end{gather}
are the block representations of the corresponding Weyl functions. Moreover,
let $m_0(\cd):\CR\to [H_0]$ be the operator function given by
\begin{gather}
m_0(\l)=\begin{pmatrix}M_1(\l) & 0 \cr N_{1+}(\l) & \tfrac i 2
I_{\cH_2'}\end{pmatrix}:\underbrace{H_0'\oplus\cH_2'}_{H_0}\to
\underbrace{H_0'\oplus\cH_2'}_{H_0}, \quad \l\in\bC_+\label{4.34}\\
m_0(\l)=\begin{pmatrix}M_1(\l) &  N_{1-}(\l) \cr  0   & -\tfrac i 2
I_{\cH_2'}\end{pmatrix}:\underbrace{H_0'\oplus\cH_2'}_{H_0}\to
\underbrace{H_0'\oplus\cH_2'}_{H_0}, \quad \l\in\bC_-\label{4.35}.
\end{gather}
 Then: 1) for every $\l\in\CR$ there exists an operator solution
$v_0(\cd,\l)\in\lo{H_0}$ of Eq. \eqref{3.2} such that \eqref{4.3} and
\eqref{4.3a} hold and
\begin{gather}
i(P_{\wh\cH_b}\wh\G_a -\wh\G_b)v_0(\l)=P_{\wh\cH_b}, \quad\l\in\bC_+; \qquad
i(\wh\G_a -\wh\G_b)v_0(\l)=P_{\wh H}, \quad\l\in\bC_-,\label{4.35a}\\
\G_{0b} v_0(\l)=0,\;\;\;\l\in\CR,\label {4.36}\\
\G_{1b}v_0(\l)=-M_{3}(\l)P_{H_0'}
,\;\;\l\in\bC_+;\qquad \G_{1b}v_0(\l)=(-M_{3}(\l)\,:\,
-N_{2-}(\l)),\;\;\l\in\bC_-\label {4.37a}
\end{gather}
(in \eqref{4.37a} $P_{H_0'}$ is the orthoprojector in $H_0$ onto $H_0'$).
Moreover,  for every $\l\in\CR$ there exists a solution $u(\cd,\l)\in
\lo{\cH_b}$ of Eq. \eqref{3.2} such that
\begin{gather}
\G_{1a} u(\l)=0, \quad \l\in\CR, \label{4.38} \\
(\G_{0a}+\wh\G_{\a})u(\l)=M_{2}(\l)+N_{2+}(\l),\;\; \l\in\bC_+; \quad
(\G_{0a}+\wh\G_{\a})u(\l)=M_{2}(\l), \;\; \l\in\bC_-\label{4.39}\\
i(P_{\wh\cH_b}\wh\G_a -\wh\G_b)u(\l)=0, \;\;\;\l\in\bC_+; \qquad i(\wh\G_a
-\wh\G_b)u(\l)=0, \quad\l\in\bC_-,\label{4.39a}\\
\G_{0b} u (\l)=I_{\cH_b}, \qquad \G_{1b} u (\l)=-M_4(\l), \;\;\;
\;\;\l\in\CR.\label{4.39b}
\end{gather}
2) The following equalities hold
\begin{gather}
\g_-(\l)\up H_0= \pi v_0 (\l), \quad\l\in\bC_-; \qquad \g_\pm(\l)\up\cH_b=\pi u
(\l), \quad\l\in\bC_\pm.\label{4.39c}
\end{gather}
\end{proposition}
\begin{proof}
1) By using the reasonings from Proposition \ref{pr4.1} to the $\g$-fields
\eqref{2.38} of the triplet $\Pi_-$ one can prove that for every $\l\in\bC_+\;
(\l\in\bC_-)$ there  exists an isomorphism $Z_+(\l):\cH_1\to \cN_\l$ (resp.
$Z_-(\l):\cH_0\to \cN_\l$) such that \eqref{4.11} holds and the relations
\begin{gather}
P_{\cH_1}\G_0' Z_+(\l)=I_{\cH_1}, \qquad (\G_1'-i P_{\cH_2'}\G_0')Z_+(\l)=
M_+(\l), \quad \l\in\bC_+,\label{4.40}\\
\G_0' Z_-(\l)=I_{\cH_0}, \qquad \G_1' Z_-(\l)=M_-(\l), \quad
\l\in\bC_-.\label{4.41}
\end{gather}
are valid. In view of \eqref{3.38} the equalities \eqref{4.40}  can be
represented as
\begin{gather}
\begin{pmatrix}-\G_{1a}+i (P_{\wh \cH_b}\wh\G_a- \wh\G_b) \cr \G_{0b}
\end{pmatrix} Z_+(\l)=\begin{pmatrix} I_{H_0'} & 0 \cr 0 & I_{\cH_b}
\end{pmatrix}, \;\; \l\in\bC_+ \label{4.42}\\
\begin{pmatrix}\G_{0a}+\tfrac 1 2  (P_{\wh \cH_b}\wh\G_a+ \wh\G_b)\cr
P_{\cH_2'}\wh\G_a \cr -\G_{1b}\end{pmatrix} Z_+(\l)=\begin{pmatrix} M_1(\l )&
M_{2}(\l)\cr N_{1+}(\l) & N_{2+}(\l) \cr M_{3}(\l) & M_{4}(\l)
\end{pmatrix}, \;\; \l\in\bC_+ \label{4.43}
\end{gather}
Therefore for all  $\l\in\bC_+$ one has
\begin{gather}
\G_{1a}Z_+(\l)=(-P_H\,:\,0), \qquad \tfrac 1 2 (P_{\wh\cH_b}\wh\G_a -\wh\G_b)
Z_+(\l)= (-\tfrac i 2 P_{\wh\cH_b}\,:\,0),\label{4.44}\\
\G_{0a}Z_+(\l)= (P_H M_1(\l)\,:\, P_H M_2(\l)),\label{4.45}\\
\tfrac 1 2 (P_{\wh\cH_b}\wh\G_a +\wh\G_b) Z_+(\l)=(P_{\wh\cH_b}M_1(\l):
P_{\wh\cH_b}M_2(\l)),\label{4.46}\\
 \quad P_{\cH_2'}\wh\G_a Z_+(\l)= (N_{1+}(\l)
:N_{2+}(\l)).\label{4.46a}
\end{gather}
Moreover, summing up the second equality in \eqref{4.44} with the equalities
\eqref{4.45}, \eqref{4.46} and \eqref{4.46a} one gets
\begin {equation}\label{4.47}
(\G_{0a}+\wh\G_a)Z_+(\l)=(M_1(\l)+N_{1+}(\l) - \tfrac i 2 P_{\wh\cH_b}\,:
\,M_2(\l)+N_{2+}(\l) ):H_0'\oplus\cH_b\to H_0.
\end{equation}
Next, by using \eqref{3.38} we may rewrite the equalities \eqref{4.41} as
\begin{gather}
\begin{pmatrix}-\G_{1a}+i (P_{\wh \cH_b}\wh\G_a- \wh\G_b)\cr i P_{\cH_2'}
\wh\G_a \cr \G_{0b}\end{pmatrix} Z_-(\l)=\begin{pmatrix} I_{H_0'} & 0 & 0 \cr 0
& I_{\cH_2'} & 0 \cr 0 & 0 & I_{\cH_b}\end{pmatrix}, \;\; \l\in\bC_-
\label{4.48}\\
\begin{pmatrix}\G_{0a}+\tfrac 1 2  (P_{\wh \cH_b}\wh\G_a+ \wh\G_b) \cr
-\G_{1b}\end{pmatrix} Z_-(\l)=\begin{pmatrix} M_1(\l )& N_{1-}(\l) & M_{2}(\l)
\cr M_{3}(\l) & N_{2-}(\l) & M_{4}(\l)\end{pmatrix}, \;\; \l\in\bC_-.
\label{4.49}
\end{gather}
In view of \eqref{3.37a} and \eqref{3.36} one has
\begin {equation}\label{4.50}
\cH_0=H_0\oplus \cH_b.
\end{equation}
Let $T_-(\l)\in [H_0,H_0']$ be the operator given by the block representation
\begin {equation}\label{4.51}
T_-(\l)=(M_1(\l)\,:\, N_{1-}(\l)):H_0'\oplus\cH_2'\to H_0', \quad \l\in\bC_-.
\end{equation}
It follows from \eqref{4.48} and \eqref{4.49} that in the decomposition
\eqref{4.50} of $\cH_0$ the following equalities hold for all $\l\in\bC_-$:
\begin{gather}
\G_{1a}Z_-(\l)=(-P_H\,:\,0), \qquad \tfrac 1 2 (P_{\wh\cH_b}\wh\G_a -\wh\G_b)
Z_-(\l)= (-\tfrac i 2 P_{\wh\cH_b}\,:\,0),\label{4.52.1}\\
P_{\cH_2'}\wh\G_a Z_-(\l)= (-i P_{\cH_2'}\,:\, 0), \qquad
\G_{0a}Z_-(\l)= (P_H T_-(\l)\,:\, P_H M_2(\l)),\label{4.52.2}\\
\tfrac 1 2 (P_{\wh\cH_b}\wh\G_a +\wh\G_b) Z_-(\l)=(P_{\wh\cH_b}T_-(\l):
P_{\wh\cH_b}M_2(\l)),\label{4.52.3}
\end{gather}
Multiplying the second equality in \eqref{4.52.1} by $2$ and summing up with
the first equality in \eqref{4.52.2} we obtain
\begin {equation}\label{4.53}
(\wh\G_a-\wh \G_b)Z_-(\l)= -i (P_{\wh H}\,:\, 0), \quad \l\in\bC_-.
\end{equation}
Moreover, summing up the second equality in \eqref{4.52.1} with the equalities
\eqref{4.52.2} and \eqref{4.52.3} one gets
\begin {equation}\label{4.54}
(\G_{0a}+\wh\G_a)Z_-(\l)= (T_-(\l) - \tfrac i 2 P_{\wh\cH_b} - i
P_{\cH_2'}\;:\; M_2(\l) ),\quad \l\in\bC_-.
\end{equation}
 Assume now that
\begin {gather}
Z_+(\l)=(r(\l)\,: \,u(\l)):H_0'\oplus\cH_b\to\cN_\l,\;\;\;\l\in\bC_+
\label{4.54.1}\\
 Z_-(\l)=(v_0(\l):u(\l)):H_0\oplus \cH_b\to
\cN_\l,\;\;\;\l\in\bC_-\label{4.54.2}
\end{gather}
are the block representations of $Z_\pm (\l)$ (in \eqref{4.54.2} we make use of
the decomposition \eqref{4.50} of $\cH_0$). Moreover,  let
\begin {equation}\label{4.54.3}
v_0(\l):= (r(\l)\;:\; 0):H_0'\oplus\cH_2'\to \cN_\l,\;\;\;\l\in\bC_+.
\end{equation}
Then the equalities \eqref{4.54.2} and \eqref{4.54.3} define the linear mapping
$v_0(\l):H_0\to\cN_\l$, while \eqref{4.54.1} and \eqref{4.54.2} give the
mapping $u(\l):\cH_b\to \cN_\l$. Next we show that the operator solutions
$v_0(\cd,\l)\in \lo{H_0}$ and  $u(\cd,\l)\in \lo{\cH_b} $   of Eq. \eqref{3.2}
corresponding to $v_0(\l)$ and $u(\l)$ in accordance with Lemma \ref{lem3.2}
have the desired properties.

Combining \eqref{4.54.1}--\eqref{4.54.3} with  \eqref{4.44}, the first equality
in  \eqref{4.52.1} and \eqref{4.53} we obtain \eqref{4.3} and the relations
\eqref{4.35a}, \eqref{4.38} and \eqref{4.39a}. Moreover, in view of
\eqref{4.34}, \eqref{4.35} and \eqref{4.51} one has
\begin {gather*}
m_0(\l)-\tfrac i 2 P_{\wh H}=(M_1(\l)+N_{1+}(\l))P_{H_0'}+\tfrac i 2
P_{\cH_2'}-\tfrac i 2 (P_{\wh\cH_b}+P_{\cH_2'})=\\
(M_1(\l)+N_{1+}(\l)- \tfrac i 2 P_{\wh\cH_b})P_{H_0'}, \quad \l\in\bC_+,\\
m_0(\l)-\tfrac i 2 P_{\wh H}=T_-(\l)-\tfrac i 2 P_{\cH_2'}-\tfrac i 2
(P_{\wh\cH_b}+P_{\cH_2'})= T_-(\l)-\tfrac i 2 P_{\wh\cH_b}-i P_{\cH_2'}, \quad
\l\in\bC_-.
\end{gather*}
Combining these equalities with \eqref{4.47}, \eqref{4.54} and taking the block
representations \eqref{4.54.1}--\eqref{4.54.3} into account one gets the
equalities \eqref{4.3a} and \eqref{4.39}. Finally, the relations \eqref{4.36},
\eqref{4.37a} and \eqref{4.39b} are implied by the block representations
\eqref{4.54.1}-\eqref{4.54.3} and the equalities \eqref{4.42}, \eqref{4.43} and
\eqref{4.48}, \eqref{4.49}.

2) The equalities \eqref{4.39c} are immediate from \eqref{4.11} and the block
representations  \eqref{4.54.1}, \eqref{4.54.2}.
\end{proof}
\begin{theorem}\label{th4.4}
Let the assumptions of Proposition \ref{pr4.3} be fulfilled and let
$\tau=\tau(\l)\in \wt R(\cH_b)$ be a Nevanlinna  operator pair \eqref{2.19}.
Then:

1) The statement 1) of Theorem \ref{th4.2} holds with the boundary condition
\eqref{4.27} and the following boundary conditions instead of
\eqref{4.27a}--\eqref{4.28a}:
\begin {gather}
i(P_{\wh\cH_b}\wh\G_a- \wh\G_b)v_\tau(\l)=P_{\wh \cH_b},\quad \l\in\bC_+;
\qquad i(\wh\G_a- \wh\G_b)v_\tau(\l)=P_{\wh H},\quad\; \l\in\bC_-\label {4.56}\\
C_0(\l)\G_{0b}v_\tau(\l)+C_1(\l)\G_{1b}v_\tau(\l)=0, \;\;\;\;\l\in\CR \label
{4.57}
\end{gather}
(here $P_{\wh\cH_b}\; (P_{\wh H})$ is the orthoprojector in $H_0$ onto
$\wh\cH_b$ (resp. $\wh H$) in accordance with the decomposition \eqref{3.36});

 2) the operator function  $v_\tau(\cd,\l)$ can be represented as
\begin {gather}
v_\tau(t,\l)=v_0(t,\l)-u(t,\l)(\tau(\l)+M_{4}(\l))^{-1}M_3(\l)P_{H_0'}, \quad
\l\in\bC_+,\label{4.58}\\
v_\tau(t,\l)=v_0(t,\l)-u(t,\l)(\tau(\l)+M_{4}(\l))^{-1}S_-(\l), \quad
\l\in\bC_-,\label{4.58.1}
\end{gather}
where $v_0(t,\l)$ and $u(t,\l)$ are defined in Proposition \ref{pr4.3} and
\begin {equation}\label{4.58a}
S_-(\l):=(M_3(\l)\,:\,N_{2-}(\l)):H_0'\oplus \cH_2'\to \cH_b,\;\;\;\;
\l\in\bC_-.
\end{equation}
\end{theorem}
\begin{proof}
Let us show that the solution $v_\tau(\cd,\l)$ defined in \eqref{4.58}
satisfies the boundary conditions  \eqref{4.27},  \eqref{4.56} and
\eqref{4.57}. Combining \eqref{4.58} with \eqref{4.3}, \eqref{4.35a} and
\eqref{4.38}, \eqref{4.39a} one obtains the relations \eqref{4.27} and
\eqref{4.56}. Similarly, the relations \eqref{4.36}, \eqref{4.37a} and
\eqref{4.39b} give
\begin {gather*}
\G_{0b}v_\tau(\l)=-(\tau(\l)+M_{4}(\l))^{-1}M_3(\l)P_{\cH_0'}, \;\;\;\l\in\bC_+,\\
\G_{1b}v_\tau(\l)=(-I+M_{4}(\l)(\tau(\l)+M_{4}(\l))^{-1})M_3(\l)P_{\cH_0'},
\;\;\; \l\in\bC_+,\\
\G_{0b}v_\tau(\l)=-(\tau(\l)+M_{4}(\l))^{-1}S_-(\l),\;\;\;\l\in\bC_-,\\
\G_{1b}v_\tau(\l)=(-I+M_{4}(\l)(\tau(\l)+M_{4}(\l))^{-1})S_-(\l),
\;\;\;\l\in\bC_-.
\end{gather*}
 Hence $\{\G_{0b} v_\tau(\l)h_0,
\G_{1b}v_\tau(\l)h_0\}\in\tau(\l), \; h_0\in H_0, \; \l\in\CR,$ which yields
\eqref{4.57}. Finally, by using Theorem \ref{th4.2a} one proves uniqueness of
$v_\tau (\cd,\l)$ in the same way as in Theorem \ref{th4.2}.
\end{proof}
\subsection{Case 3}
Application of Corollary \ref{cor2.12a} to the boundary triplet \eqref{3.57a}
gives the following theorem.
\begin{theorem}\label{th4.5}
Let in Case 3 $\Pi_-=\bta$ be a decomposing boundary triplet \eqref{3.40},
\eqref{3.42} for $\Tma$, let $T$ be a symmetric relation \eqref{3.56} and let
$\pair\in\RM$ be a collection of holomorphic operator pairs  \eqref{2.2}. Then
for every $f\in\lI$ the boundary value problem
\begin{gather}
J y'-B(t)y=\l \D(t)y+\D(t)f(t), \quad t\in\cI,\label{4.59}\\
\G_{1a}y=0, \quad C_0(\l)\wt\G_{0b}y+C_1(\l)\G_{1b}y=0, \quad \l\in\bC_+,
 \label{4.60}\\
\G_{1a}y=0, \quad \wh \G_a y= 0,\quad D_0(\l)\wt\G_{0b}y +D_1(\l)\G_{1b}y=0,
\quad \l\in\bC_-. \label{4.61}
\end{gather}
 has a unique solution $y(t,\l)=y_f(t,\l) $ and the equality \eqref{4.0.4}
defines a generalized resolvent $R(\l)=:R_\tau(\l)$ of $T$. If in addition $\wh
H=\{0\} $, then for each generalized resolvent $R(\l)$ of $T$ there exists a
unique $\tau\in\RM$ such that $R(\l)=R_\tau(\l)$.
\end{theorem}
\begin{remark}\label{rem4.5a}
Let in Theorem \ref{th4.5} $\tau_0=\{\tau_+,\tau_-\}\in \RM$ be defined by
\eqref{2.2} with
\begin {equation}\label{4.61a}
C_0(\l) \equiv P_{\cH_b}(\in [\wt\cH_b,\cH_b]), \quad C_1(\l)\equiv 0, \;\;\;
  D_0(\l)\equiv I_{\wt\cH_b},\quad D_1(\l)\equiv 0
\end{equation}
and let $A_0$ be the symmetric extension \eqref{3.58}. Then
$R_0(\l):=R_{\tau_0}(\l)$ is if the form \eqref{4.33.0}.
\end{remark}
\begin{proposition}\label{pr4.6}
Let in Case 3   $\Pi_-=\bta$ be a decomposing boundary triplet \eqref{3.40},
\eqref{3.42} for $\Tma$, let $\g_\pm(\cd)$ be the $\g$-fields  and let
\begin{gather}
M_+(\l)=\begin{pmatrix} M_1(\l )& M_{2+}(\l) \cr N_{1+}(\l) &  N_{2+}(\l) \cr
M_{3+}(\l) & M_{4+}(\l)\end{pmatrix}: H\oplus\cH_b\to H\oplus\wh
H\oplus\wt\cH_b, \;\;\;\l\in\bC_+\label{4.61b}\\
M_-(\l)=\begin{pmatrix} M_1(\l )& N_{1-}(\l) & M_{2-}(\l) \cr M_{3-}(\l) &
N_{2-}(\l) & M_{4-}(\l)
\end{pmatrix}: H\oplus \wh H\oplus \wt\cH_b\to H\oplus \cH_b,
\;\;\;\l\in\bC_-\label{4.61c}
\end{gather}
be the block representations of the corresponding Weyl functions. Assume also
that $m_0(\cd):\CR\to [H_0]$ is the operator function defined by
\begin{gather}
m_0(\l)=\begin{pmatrix}M_1(\l) & 0 \cr N_{1+}(\l) & \tfrac i 2
I_{\wh\cH}\end{pmatrix}:H\oplus\wh H\to H\oplus\wh H, \quad \l\in\bC_+\label{4.62}\\
m_0(\l)=\begin{pmatrix}M_1(\l) &  N_{1-}(\l) \cr  0   & -\tfrac i 2 I_{\wh H
}\end{pmatrix}:H\oplus\wh H\to H\oplus\wh H, \quad
\l\in\bC_-\label{4.63}.\end{gather}
 Then: 1) For every $\l\in\CR$ there exists an operator solution
$v_0(\cd,\l)\in\lo{H_0}$ of Eq. \eqref{3.2} such that \eqref{4.3} and
\eqref{4.3a} hold and
\begin{gather}
i\wh\G_a v_0(\l)=P_{\wh H}, \quad \l\in\bC_-,\label{4.63a}\\
 \wt\G_{0b} v_0(\l)=i P_{\wh\cH_b}M_{3+}(\l)P_H(\in [H_0,\wt\cH_b]),
\;\;\;\;
\G_{1b}v_0(\l)= -P_{\cH_b}M_{3+}(\l)P_H, \quad\l\in\bC_+,\label{4.64}\\
\wt\G_{0b} v_0(\l)=0, \quad \G_{1b}v_0(\l)=(-M_{3-}(\l): -N_{2-}(\l)),\quad
\l\in\bC_-.\label{4.65}
\end{gather}

2) For every $\l\in\bC_+ \; (\l\in\bC_-)$ there exists a solution
$u_+(\cd,\l)\in\lo{\cH_b}$ (resp. $u_-(\cd,\l)\in\lo{\wt\cH_b}$) of Eq.
\eqref{3.2} such that
\begin{gather}
\G_{1a} u_\pm(\l)=0, \quad \l\in\bC_\pm, \label{4.66}\\
(\G_{0a}+\wh\G_a)u_+(\l)=M_{2+}(\l)+N_{2+}(\l), \quad \l\in\bC_+,\label{4.66.1}\\
\G_{0a}u_-(\l)=M_{2-}(\l), \qquad \wh\G_a u_-(\l)=0, \quad\l\in\bC_-.\label{4.66.2}\\
\wt \G_{0b}u_+(\l)=I_{\cH_b}+iP_{\wh\cH_b}M_{4+}(\l), \qquad
\G_{1b}u_+(\l)=-P_{\cH_b}M_{4+}(\l), \quad\l\in\bC_+,\label{4.66a}\\
\wt \G_{0b}u_-(\l)=I_{\wt\cH_b}, \qquad \G_{1b} u_-(\l)=-M_{4-}(\l), \quad\;
\l\in\bC_-. \label{4.67}
\end{gather}

3) The following equalities hold
\begin{gather}\label{4.67a}
\g_+(\l)\up \cH_b=\pi u_+(\l), \;\; \;\l\in\bC_+; \qquad \g_-(\l)\up \wt\cH_b
=\pi u_-(\l), \;\; \;\l\in\bC_-.
\end{gather}
\end{proposition}
\begin{proof}
As in Proposition \ref{pr4.1} one proves the existence of isomorphisms
$Z_+(\l):\cH_1\to \cN_\l\; (\l\in\bC_+)$ and   $Z_-(\l):\cH_0\to \cN_\l
\;(\l\in\bC_-)$ satisfying  \eqref{4.40} and \eqref{4.41} or, equivalently, the
equalities
\begin{gather}
\begin{pmatrix}-\G_{1a} \cr \G_{0b}\end{pmatrix} Z_+(\l)=
\begin{pmatrix}
I_{H} & 0 \cr 0 & I_{\cH_b}\end{pmatrix},\; \;\;\begin{pmatrix}\G_{0a} \cr
\wh\G_a \cr  -\G_{1b}-i\wh \G_b\end{pmatrix}
 Z_+(\l)=\begin{pmatrix} M_1(\l )&
M_{2+}(\l)\cr N_{1+}(\l) & N_{2+}(\l) \cr M_{3+}(\l) & M_{4+}(\l)
\end{pmatrix}, \;\; \l\in\bC_+,\label{4.68}\\
\begin{pmatrix}-\G_{1a}\cr i \wh\G_a \cr \wt\G_{0b}\end{pmatrix} Z_-(\l)=
\begin{pmatrix} I_{H} & 0 & 0 \cr
0 & I_{\wh H} & 0 \cr 0 & 0 & I_{\wt\cH_b}\end{pmatrix}, \;\;\;\; \l\in\bC_-
\label{4.69}\\
\begin{pmatrix}\G_{0a}\cr -\G_{1b}\end{pmatrix} Z_-(\l)=\begin{pmatrix}
M_1(\l )& N_{1-}(\l) & M_{2-}(\l) \cr M_{3-}(\l) & N_{2-}(\l) &
M_{4-}(\l)\end{pmatrix}, \;\;\;\; \l\in\bC_-. \label{4.69a}
\end{gather}
It follows from \eqref{4.68}--\eqref{4.69a} that
\begin{gather}
\G_{1a}Z_+(\l)=(-I_H\; : \; 0 ), \;\;\;\l\in\bC_+; \qquad
\G_{1a}Z_-(\l)=(-I_H\;
: \; 0 \; : \; 0 ), \;\;\;\l\in\bC_-,\label{4.70}\\
(\G_{0a}+\wh\G_a)Z_+(\l)=\begin{pmatrix}\G_{0a} \cr \wh\G_a\end{pmatrix}
Z_+(\l)=\begin{pmatrix} M_1(\l)& M_{2+}(\l)  \cr  N_{1+}(\l) & N_{2+}(\l
)\end{pmatrix}:H \oplus \cH_b\to H\oplus \wh H,\label{4.71}\\
(\G_{0a}+\wh\G_a)Z_-(\l)=\begin{pmatrix}\G_{0a} \cr \wh\G_a\end{pmatrix}
Z_-(\l)=\qquad\qquad\qquad\qquad\qquad\qquad\qquad\qquad\qquad
\label{4.71a}\\
\qquad\qquad\qquad\qquad\qquad\qquad\qquad\begin{pmatrix} M_1(\l)& N_{1-}(\l) &
M_{2-}(\l) \cr 0& -iI_{\wh H}
& 0\end{pmatrix}:H \oplus\wh H\oplus \wt\cH_b\to H\oplus \wh H,\nonumber\\
\G_{0b}Z_+(\l)=(0\,:\, I_{\cH_b}), \qquad \wh\G_b Z_+(\l)=(i P_{\wh\cH_b}
M_{3+}(\l):i P_{\wh\cH_b} M_{4+}(\l)), \;\;\;\l\in\bC_+,\label{4.72}\\
\G_{1b}Z_+(\l)=(-P_{\cH_b}M_{3+}(\l)\,:\, -P_{\cH_b}M_{4+}(\l)), \quad
\l\in\bC_+,\label{4.73}\\
 \wt\G_{0b}Z_-(\l)=(0\,:\, 0\,:\, I_{\wt\cH_b}),\quad
\G_{1b}Z_-(\l)=-(M_{3-}(\l)\,:\, N_{2-}(\l)\,:\, M_{4-}(\l)),
\;\;\l\in\bC_-,\label{4.74}
\end{gather}
and combining of \eqref{3.39b} with \eqref{4.72} gives
\begin {equation}\label{4.75}
\wt\G_{0b}Z_+(\l)=(i P_{\wh\cH_b} M_{3+}(\l):I_{\cH_b}+i P_{\wh\cH_b}
M_{4+}(\l)),\quad \l\in\bC_+.
\end{equation}
Assume that
\begin {gather}
Z_+(\l)=(r(\l)\,: \,u_+(\l)):H\oplus\cH_b\to\cN_\l,\;\;\;\l\in\bC_+,
\label{4.76}\\
 Z_-(\l)=(v_0(\l):u_-(\l)):H_0\oplus \wt\cH_b\to
\cN_\l,\;\;\;\l\in\bC_-,\label{4.77}
\end{gather}
are the block representations of $Z_\pm (\l)$ (see  \eqref{3.39c} ) and let
\begin {equation}\label{4.78}
v_0(\l):= (r(\l)\;:\; 0):H\oplus\wh H\to \cN_\l,\;\;\;\l\in\bC_+.
\end{equation}
Then the equalities \eqref{4.76}--\eqref{4.78} define
 the linear mappings $v_0(\l):H_0\to\cN_\l \; (\l\in\CR), \; u_+(\l):\cH_b\to
\cN_\l\; (\l\in\bC_+)$ and $u_-(\l):\wt\cH_b\to \cN_\l\; (\l\in\bC_-)$, which
according to Lemma \ref{lem3.2} generate the operator solutions $v_0(\cd,\l)\in
\lo{H_0}, \; u_+(\cd,\l)\in\lo{\cH_b}$ and $ u_-(\cd,\l)\in\lo{\wt\cH_b}$ of
Eq. \eqref{3.2}. Combining now \eqref{4.76}--\eqref{4.78}  with \eqref{4.70}--
\eqref{4.71a} and taking the equalities \eqref{4.62} and \eqref{4.63} into
account we arrive  at the relations \eqref{4.3}, \eqref{4.3a}, \eqref{4.63a}
and \eqref{4.66}--\eqref{4.66.2}. Moreover, the block representations
\eqref{4.76}--\eqref{4.78} and the equalities \eqref{4.73}--\eqref{4.75} lead
to \eqref{4.64}, \eqref{4.65} and \eqref{4.66a}, \eqref{4.67}.

Finally, \eqref{4.67a} is implied by \eqref{4.11} and  \eqref{4.76},
\eqref{4.77}.
\end{proof}
\begin{theorem}\label{th4.7}
Let the assumptions of Proposition \ref{pr4.6} be satisfied, let $\pair\in \RM$
be a collection of operator pairs \eqref{2.2} and let
\begin {equation}\label{4.83}
S_-(\l)=(M_{3-}(\l)\,:\, N_{2-}(\l)):H\oplus\wh H\to\cH_b, \;\;\;\l\in\bC_-.
\end{equation}

 Then: 1)the statement 1) of Theorem \ref{th4.2} holds with the boundary
condition \eqref{4.27} and the following boundary conditions in place of
\eqref{4.27a}--\eqref{4.28a}:
\begin {gather}
i \wh\G_a v_\tau(\l)=P_{\wh H},\quad \l\in\bC_-,\label {4.79}\\
C_0(\l)\wt\G_{0b}v_\tau(\l)+C_1(\l)\G_{1b}v_\tau(\l)=0,
\;\;\l\in\bC_+,\label {4.80} \\
D_0(\l)\wt\G_{0b}v_\tau(\l)+D_1(\l)\G_{1b}v_\tau(\l)=0, \;\;\l\in\bC_-.\label
{4.80a}
\end{gather}

2) the solution $v_\tau(\cd,\l)$ is of the form
\begin {gather}
v_\tau(t,\l)=v_0(t,\l)-u_+(t,\l)(\tau_-^*(\ov\l)+M_{4+}(\l))^{-1}M_{3+}(\l)P_H,
\quad
\l\in\bC_+ \label{4.81}\\
v_\tau(t,\l)=v_0(t,\l)-u_-(t,\l)(\tau_-(\l)+M_{4-}(\l))^{-1}S_-(\l), \quad
\l\in\bC_- \label{4.82}.
\end{gather}
\end{theorem}
\begin{proof}
 Let us show that the  solution $v_\tau(\cd,\l)\in\lo{H_0}$ of Eq. \eqref{3.2}
defined by \eqref{4.81} and \eqref{4.82} satisfies the boundary conditions
\eqref{4.27} and \eqref{4.79}--\eqref{4.80a}.

Combining \eqref{4.81} and \eqref{4.82} with \eqref{4.3} and \eqref{4.66} we
obtain the equality  \eqref{4.27}. Moreover, \eqref{4.81} and \eqref{4.82}
together with \eqref{4.63a} and the second equality in \eqref{4.66.2} give
\eqref{4.79}.

Next assume that $T_+(\l)=(\tau_-^*(\ov\l)+M_{4+}(\l))^{-1},\;\l\in\bC_+,$ and
$T_-(\l)=(\tau_-(\l)+M_{4-}(\l))^{-1},\;\;\l\in\bC_-$. Then by using
\eqref{2.8b} one obtains
\begin {gather}
\tau_+(\l)=\{\{(-T_+(\l)+iP_{\wh\cH_b}-iP_{\wh\cH_b}M_{4+}(\l)T_+(\l))h ,\qquad
\qquad\qquad \qquad
\qquad\qquad\label{4.84}\\
\qquad\qquad \qquad\qquad\qquad \qquad\qquad \qquad\qquad(-P_{\cH_b} +
P_{\cH_b}M_{4+}(\l)T_+(\l))h\}:h\in\wt\cH_b\},\nonumber\\
\tau_-(\l)=\{\{T_-(\l)h, (I-M_{4-}(\l)T_-(\l))h\}:h\in\cH_b\}\label{4.85}
\end{gather}
(c.f. proofs of the equalities \eqref{4.31.1} and \eqref{4.31.2} in Theorem
\ref{th4.2}). Moreover, combining  \eqref{4.81} and \eqref{4.82} with with the
equalities \eqref{4.64}, \eqref{4.65} and \eqref{4.66a},  \eqref{4.67} one gets
\begin {gather*}
\wt\G_{0b}v_\tau(\l)=(iP_{\wh\cH_b}-T_+(\l)-iP_{\wh\cH_b}M_{4+}(\l)T_+(\l))
M_{3+}(\l)P_H,\;\;\;\l\in\bC_+,\\
\G_{1b}v_\tau(\l)=(-P_{\cH_b} + P_{\cH_b}M_{4+}(\l)T_+(\l))M_{3+}(\l)P_H,
\;\;\;\l\in\bC_+,\\
\wt\G_{0b}v_\tau(\l)=-T_-(\l)S_-(\l), \quad \G_{1b}v_\tau(\l)=
-(I-M_{4-}(\l)T_-(\l))S_-(\l), \quad \l\in \bC_-,
\end{gather*}
which in view of \eqref{4.84} and \eqref{4.85} yields the inclusions
\eqref{4.31.5}. This implies that $v_\tau(\cd,\l)$ satisfies \eqref{4.80} and
\eqref{4.80a}. Finally, by using the same arguments as in Theorem \ref{th4.2}
one proves uniqueness of the solution $v_\tau(\cd,\l)$.
\end{proof}
\section{$m$-functions}\label{sect5}
Assume that $\Pi_\a=\bta$ is a decomposing boundary triplet for $\Tma$ defined
by \eqref{3.42} and one of the equalities \eqref{3.35}, \eqref{3.38} or
\eqref{3.40}.
\begin{definition}\label{def5.1}
A boundary parameter $\tau $ (at the endpoint $b$) is:

--- \, a collection $\pair\in\wt R_\a(\wt\cH_b,\cH_b)$ of operator pairs
\eqref{2.2} with $\a=+1$ in \emph{Case 1} and  $\a=-1$ in \emph{Case 3};

--- \, an operator pair $\tau\in\wt R(\cH_b)$ defined by \eqref{2.19} in
 \emph{Case 2}.

If  $n_+=n_-$ and $\Pi=\bt$ is a decomposing boundary triplet \eqref{3.35.2},
\eqref{3.42} for $\Tma$, then a boundary parameter is an operator pair
$\tau=\wt R(\cH_b)$ of the form \eqref{2.19}.
\end{definition}
Let $\tau$ be a boundary parameter and let $v_\tau (\cd,\l)\in\lo{H_0}$ be the
corresponding operator solution of Eq. \eqref{3.2} defined in Theorems
\ref{th4.2}, \ref{th4.4} and \ref{th4.7}.
\begin{definition}\label{def5.2}
The operator function $m_\tau(\cd):\CR\to [H_0]$ defined by
\begin {equation}\label{5.0}
m_\tau(\l)=(\G_{0a}+\wh \G_a)v_\tau (\l)+\tfrac i 2 P_{\wh H}, \quad\l\in\CR,
\end{equation}
will be called the $m$-function corresponding to the boundary parameter $\tau$
or, equivalently, to the boundary value problem \eqref{4.0.1}--\eqref{4.0.3}
(in \emph{Case 1}), \eqref{4.32.1}--\eqref{4.32.3} (in \emph{Case 2}) or
\eqref{4.59}--\eqref{4.61} (in \emph{Case 3}).

If $n_+=n_-$, then  $m_\tau(\cd)$ corresponds to the boundary value problem
\eqref{4.0.1}, \eqref{4.0.5}. In this case the $m$-function $m_\tau(\cd)$ will
be called canonical if $\tau\in \wt R^0(\cH_b)$.
\end{definition}
It follows from  \eqref{4.27} that $m_\tau(\cd)$ satisfies the equality
\begin {equation}\label{5.1}
v_{\tau,a}(\l)\left(=\begin{pmatrix} \G_{0a}+\wh \G_a\cr \G_{1a} \end{pmatrix}
v_\tau(\l)\right )=\begin{pmatrix} m_\tau(\l)-\tfrac i 2 P_{\wh H}\cr -P_H
\end{pmatrix}:H_0\to H_0\oplus H, \quad \l\in\CR.
\end{equation}
In the following proposition we show that the $m$-function  $m_\tau(\cd)$ can
be defined in a somewhat different way.
\begin{proposition}\label{pr5.3}
Let $\Pi_\a=\bta$ be a decomposing boundary triplet for $\Tma$, let $\tau$ be a
boundary parameter at the endpoint $b$ and let $\f(\cd,\l)(\in [H_0,\bH])$ and
$\psi(\cd,\l)(\in [H_0,\bH]), \;\l\in\bC,$ be the operator solutions of Eq.
\eqref{3.2} with the initial data
\begin {equation}\label{5.2}
\f_a(\l)=\begin{pmatrix} I_{H_0}\cr 0\end{pmatrix}:H_0\to H_0\oplus H, \quad
\psi_a(\l)=\begin{pmatrix} -\tfrac i 2 P_{\wh H}\cr -P_H\end{pmatrix}:H_0\to
H_0\oplus H, \quad \l\in\bC.
\end{equation}
Then there exists a unique operator function $m(\cd):\CR\to [H_0]$ such that
for every $\l\in\CR$ the operator solution $v(\cd,\l)$ of Eq. \eqref{3.2} given
by
\begin {equation}\label{5.3}
v(t,\l):=\f (t,\l)m(\l)+\psi (t,\l)
\end{equation}
belongs to $\lo{H_0}$ and satisfies the following boundary conditions:
\eqref{4.27a}--\eqref{4.28a} in Case 1; \eqref{4.56} and \eqref{4.57} in
\emph\Ca{2}; \eqref{4.79}--\eqref{4.80a} in \emph\Ca{3}.  Moreover, the
equalities $v(t,\l)= v_\tau(t,\l)$ and $m(\l)=m_\tau(\l)$ are valid.
\end{proposition}
\begin{proof}
Let $m_\tau(\cd)$ be the  $m$-function in the sense of Definition \ref{def5.2}
and let $v(\cd,\l), \;\l\in\CR,$ be the solution of Eq. \eqref{3.2} given by
\eqref{5.3} with $m(\l)=m_\tau(\l)$. Then in view of \eqref{5.2} and
\eqref{5.1} one has $v_a(\l)=v_{\tau,a}(\l)$ and, consequently, $v(t,\l)=
v_\tau(t,\l)$. Therefore by Theorems \ref{th4.2}, \ref{th4.4} and \ref{th4.7}
$v(\cd,\l)$ belongs to $\lo{H_0}$ and satisfies the  required boundary
conditions. Hence there exists an operator function $m(\l)(=m_\tau(\l))$ with
the desired properties.

Assume now that the solution $v(\cd,\l)$ of Eq. \eqref{3.2} given by
\eqref{5.3} with some $m(\l)$ belongs to $\lo{H_0}$ and satisfies the specified
boundary conditions. Then in view of \eqref{5.2} $\G_{1a}v(\l)=-P_H$ and
according to Theorems \ref{th4.2}, \ref{th4.4} and \ref{th4.7} $v(t,\l)=
v_\tau(t,\l)$. Therefore $m(\l)=m_\tau(\l)$, which proves uniqueness of
$m(\l)$.
\end{proof}
Description of all $m$-functions immediately in terms of the boundary parameter
$\tau$ is contained in the following three theorems.
\begin{theorem}\label{th5.4}
Let in Case 1 the assumptions of Proposition \ref{pr4.1} be satisfied and  let
$\tau_0=\{\tau_+,\tau_- \}$ be a boundary parameter \eqref{2.2}, \eqref{4.0.6}.
Then $m_0(\l)=m_{\tau_0}(\l)$ and for every boundary parameter $\pair$
 defined by \eqref{2.2} the
corresponding $m$-function $m_\tau(\cd)$ is of the form
\begin {equation}\label{5.5}
m_\tau(\l)=m_0(\l)+M_{2+}(\l)(C_0(\l)-C_1(\l)M_{4+}(\l))^{-1}C_1(\l)M_{3+}(\l),
\quad\l\in\bC_+.
\end{equation}
\end{theorem}
\begin{proof}
It follows from \eqref{4.3} and \eqref{4.4}--\eqref{4.6} that $v_0(t,\l)=v_
{\tau_0}(t,\l)$ and \eqref{4.3a} yields $m_0(\l)=m_{\tau_0}(\l)$. Next,
applying the operator $\G_{0a}+\wh\G_a$ to the equalities \eqref{4.29} and
\eqref{4.30} with taking \eqref{4.3a} and  \eqref{4.7a} into account one
obtains
\begin {gather}
m_\tau(\l)=m_0(\l)-M_{2+}(\l)(\tau_+(\l)+M_{4+}(\l))^{-1}M_{3+}(\l),
\quad\l\in\bC_+, \label{5.6}\\
m_\tau(\l)=m_0(\l)-M_{2-}(\l)(\tau_+^*(\ov\l)+M_{4-}(\l))^{-1}M_{3-}(\l),
\quad\l\in\bC_-. \label{5.7}
\end{gather}
It follows from the equality $M_+^*(\ov\l)=M_-(\l)$ that $m_0^*(\ov\l)=m_0(\l),
\;M_{2+}^*(\ov\l)=M_{3-}(\l), \;  M_{3+}^*(\ov\l)=M_{2-}(\l) $ and
$M_{4+}^*(\ov\l)=M_{4-}(\l)$. This and \eqref{5.6},  \eqref{5.7} yield
\begin {equation}\label{5.7b}
m_\tau^*(\ov\l)=m_\tau(\l), \quad \l\in\CR.
\end{equation}
Moreover, according to \cite[Lemma 2.1]{MalMog02}
$0\in\rho(C_0(\l)-C_1(\l)M_{4+}(\l))$ and
\begin {equation*}
-(\tau_+(\l)+M_{4+}(\l))^{-1}=(C_0(\l)-C_1(\l)M_{4+}(\l))^{-1}C_1(\l),
\quad\l\in\bC_+,
\end{equation*}
which together with \eqref{5.6} yields \eqref{5.5}.
\end{proof}
The following corollary is immediate from Theorem \ref{th5.4} and the
equalities \eqref{3.35.0}.
\begin{corollary}\label{cor5.4a}
Let $n_+=n_-$, let $\Pi=\bt$ be a decomposing boundary triplet \eqref{3.35.2},
\eqref{3.42} for $\Tma$, let $\tau_0=\{(I_{\cH_b},0);\cH_b\}\in\wt R^0(\cH_b)$
and let
\begin {equation}\label{5.8}
M(\l)=\begin{pmatrix} m_0(\l) & M_2(\l) \cr M_3(\l) & M_4(\l)
\end{pmatrix}:H_0\oplus\cH_b\to
H_0\oplus\cH_b, \quad \l\in\CR,
\end{equation}
be the block representation of the  Weyl function of $\Pi$. Then
$m_0(\l)=m_{\tau_0}(\l)$ and for every boundary parameter $\tau$ defined by
\eqref{2.19}  the corresponding $m$-function $m_\tau(\cd)$ is
\begin {equation}\label{5.8a}
m_\tau(\l)=m_0(\l)+M_2(\l)(C_0(\l)-C_1(\l)M_4(\l))^{-1}C_1(\l)M_3(\l),
\quad\l\in\CR.
\end{equation}
\end{corollary}
\begin{theorem}\label{th5.5}
Let in Case 2 the assumptions of Proposition \ref{pr4.3} be satisfied, let
$\tau_0=\{(I_{\cH_b},0);\cH_b\}\in\wt R^0(\cH_b)$ and let $S(\l)$ be the
operator function \eqref{4.58a}.

Then: 1) $m_0(\l)=m_{\tau_0}(\l)$ and for every boundary parameter $\tau$
defined by \eqref{2.19} the corresponding $m$-function $m_\tau(\cd)$ is
\begin {equation}\label{5.9}
m_\tau(\l)=m_0(\l)+M_2(\l)(C_0(\l)-C_1(\l)M_{4}(\l))^{-1}C_1(\l)S_-(\l),
\quad\l\in\bC_-.
\end{equation}

2)  $m_\tau(\l) $ admits the triangular block representation
\begin {gather}
m_\tau(\l)=\begin{pmatrix} m_{1,\tau}(\l) & 0 \cr  m_{+,\tau}(\l) & \tfrac i 2
I_{\cH_2'}\end{pmatrix}:\underbrace{H_0'\oplus\cH_2'}_{H_0}\to
\underbrace{H_0'\oplus\cH_2'}_{H_0}, \quad \l\in\bC_+\label{5.10}\\
m_\tau(\l)=\begin{pmatrix} m_{1,\tau}(\l)  &   m_{-,\tau}(\l)  \cr 0 & -\tfrac
i 2 I_{\cH_2'}\end{pmatrix}:\underbrace{H_0'\oplus\cH_2'}_{H_0}\to
\underbrace{H_0'\oplus\cH_2'}_{H_0}, \quad \l\in\bC_-,\label{5.11}
\end{gather}
where
\begin {gather}
m_{1,\tau}(\l)=M_1(\l)+M_2(\l)(C_0(\l)-C_1(\l)M_{4}(\l))^{-1}C_1(\l)M_3(\l),
\quad \l\in\CR,\label{5.12}\\
 m_{-,\tau}(\l) = N_{1-}(\l)+M_2(\l)(C_0(\l)-C_1(\l)M_{4}(\l))^{-1}C_1(\l)N_{2-}
(\l),\quad \l\in\bC_-,\label{5.13}\\
m_{+,\tau}(\l)=m_{-,\tau}^*(\ov\l),\quad \l\in\bC_+.\label{5.14}
\end{gather}
\end{theorem}
\begin{proof}
1) The equality $m_0(\l)=m_{\tau_0}(\l)$ is implied by \eqref{4.3},
\eqref{4.35a} and \eqref{4.36} in the same way as in Theorem \ref{th5.4}. Next,
by \eqref{4.58}, \eqref{4.58.1} and \eqref{4.3a}, \eqref{4.39} one has
\begin {gather}
m_\tau(\l)=m_0(\l)-(M_2(\l)+N_{2+}(\l))(\tau(\l)+M_{4}(\l))^{-1}M_3(\l)P_{H_0'},
\quad\l\in\bC_+,\label{5.14a}\\
m_\tau(\l)=m_0(\l)-M_2(\l)(\tau(\l)+M_{4}(\l))^{-1}M_3(\l)S_-(\l),
\quad\l\in\bC_-.\label{5.14b}
\end{gather}
It follows from \eqref{4.33.1},  \eqref{4.33.2} and the equality
$M_+^*(\ov\l)=M_-(\l)$ that for all $\l\in\bC_-$
\begin {gather*}
(M_3(\ov\l)P_{H_0'})^*=M_2(\l), \qquad M_4^*(\ov\l)=M_4(\l),\\
(M_2(\ov\l)+N_{2+}(\ov\l))^*= (M_3(\l)\,:\,N_{2-}(\l))=S_-(\l).
\end{gather*}
This and \eqref{5.14a}, \eqref{5.14b} yield the equality \eqref{5.7b}.
Moreover, applying to \eqref{5.14b} the same arguments as in the proof of
Theorem \ref{th5.4} one obtains \eqref{5.9}.

2) It follows from \eqref{5.9} and \eqref{4.58a} that
\begin {equation*}
m_\tau(\l)=\begin{pmatrix}M_1(\l) &  N_{1-}(\l) \cr  0   & -\tfrac i 2
I_{\cH_2'}\end{pmatrix} +\begin{pmatrix}M_2(\l)\cr 0
\end{pmatrix}(C_0(\l)- C_1(\l) M_{4}(\l))^{-1}C_1(\l) (M_3(\l):N_{2-}(\l))
\end{equation*}
for all $\l\in\bC_-$. This proves \eqref{5.11}--\eqref{5.13}, which in view of
\eqref{5.7b} implies  \eqref{5.10} and \eqref{5.14}.
\end{proof}
\begin{theorem}\label{th5.6}
Let in Case 3 the conditions of Proposition \ref{pr4.6} be fulfilled, let
$\tau_0=\{\tau_{+},\tau_{-} \}$ be a boundary parameter \eqref{2.2},
\eqref{4.61a} and let $S_-(\l)$ be the operator function \eqref{4.83}. Then:

1) $m_0(\l)= m_{\tau_0}(\l)$ and for every boundary parameter $\pair$ defined
by \eqref{2.2} the
 $m$-function $m_\tau(\cd)$ is of the form
\begin {equation}\label{5.15}
m_\tau(\l)=m_0(\l)+M_{2-}(\l)(D_0(\l)-D_1(\l)M_{4-}(\l))^{-1}D_1(\l)S_-(\l),
\quad\l\in\bC_-.
\end{equation}

2) The $m$-function $m_\tau(\cd) $ has the triangular block representation
\begin {gather}
m_\tau(\l)=\begin{pmatrix} m_{1,\tau}(\l) & 0 \cr  m_{+,\tau}(\l) & \tfrac i 2
I_{\wh \cH}\end{pmatrix}:\underbrace{H\oplus \wh H}_{H_0}\to
\underbrace{H\oplus \wh H}_{H_0}, \quad \l\in\bC_+\label{5.16}\\
m_\tau(\l)=\begin{pmatrix} m_{1,\tau}(\l)  &   m_{-,\tau}(\l)  \cr 0 & -\tfrac
i 2 I_{\wh H}\end{pmatrix}: \underbrace{H\oplus\wh H}_{H_0}\to \underbrace
{H\oplus\wh H}_{H_0}, \quad \l\in\bC_-,\label{5.17}
\end{gather}
where
\begin {gather}
m_{1,\tau}(\l)=M_1(\l)+M_{2-}(\l)(D_0(\l)-D_1(\l)M_{4-}(\l))^{-1}D_1(\l)M_{3-}(\l),
\quad \l\in\bC_-,\label{5.18}\\
 m_{-,\tau}(\l) = N_{1-}(\l)+M_{2-}(\l)(D_0(\l)-D_1(\l)M_{4-}(\l))^{-1}D_1(\l)
N_{2-}(\l),\quad \l\in\bC_-,\label{5.19}\\
 m_{1,\tau}(\l)=m_{1,\tau}^*(\ov\l), \quad
m_{+,\tau}(\l)=m_{-,\tau}^*(\ov\l),\quad \l\in\bC_+. \label{5.20}
\end{gather}
\end{theorem}
\begin{proof}
We give only the sketch of the proof, because it is similar to that of Theorems
\ref{th5.4} and \ref{th5.5}. The equality $m_0(\l)= m_{\tau_0}(\l)$ follows
from \eqref{4.3} and \eqref{4.63a}-\eqref{4.65}. Next, by using \eqref{4.3a},
\eqref{4.66.1}, \eqref{4.66.2} and \eqref{4.81},\eqref{4.82} one proves the
equalities
\begin {gather*}
m_\tau(\l)=m_0(\l)-(M_{2+}(\l)+N_{2+}(\l))(\tau_-^*(\ov\l)+M_{4+}(\l))^{-1}
M_{3+}(\l)P_H,\quad\l\in\bC_+,\\
m_\tau(\l)=m_0(\l)-M_{2-}(\l)(\tau_-(\l)+M_{4-}(\l))^{-1}S_-(\l),
\quad\l\in\bC_-,
\end{gather*}
which imply \eqref{5.7b} and \eqref{5.15}. Moreover, in view of \eqref{4.83}
the equality  \eqref{5.15} can be written as
\begin {equation*}
m_\tau(\l)=\begin{pmatrix}M_1(\l) &  N_{1-}(\l) \cr  0   & -\tfrac i 2 I_{\wh
H}\end{pmatrix} +\begin{pmatrix}M_{2-}(\l)\cr 0
\end{pmatrix}(D_0(\l)- D_1(\l) M_{4_-}(\l))^{-1}D_1(\l)
(M_{3_-}(\l):N_{2-}(\l)).
\end{equation*}
This and  \eqref{5.7b} yield \eqref{5.16}--\eqref{5.20}.
\end{proof}
\begin{proposition}\label{pr5.6}
The $m$-function $m_\tau(\cd)$ is a Nevanlinna operator function such that the
relation
\begin {equation}\label{5.21}
(\im \,\l)^{-1}\cd \im\, m_\tau(\l)\geq \int_\cI v_\tau^*(t,\l)\D(t)
v_\tau(t,\l)\, dt
\end{equation}
holds for all $\l\in\bC_+$ in Case 1 and $\l\in\bC_-$ in Cases 2 and 3.

If in addition $n_+=n_-$, then \eqref{5.21} holds for all $\l\in\CR$.
\end{proposition}
\begin{proof}
We prove the proposition only for \emph{Case 1} (in \emph{Cases 2} and \emph{3}
the proof is similar).

Let $\Pi_+=\bta$ be a decomposing boundary triplet \eqref{3.42}, \eqref{3.35}
for $\Tma$ and let $\pair\in\RP$  be a boundary parameter defined by
\eqref{2.2}. Let us show that the corresponding $m$-function $m_\tau(\cd)$
satisfies \eqref{5.21}.

Assume that $\l\in\bC_+, \; h_0\in H_0$ and let $y:=v_\tau(\l)h_0$, so that
$y=y(t)=v_\tau(t,\l)h_0,\; t\in\cI$. Applying the Lagrange's identity
\eqref{3.6} to $\{y,\l y\}\in\tma$ and taking the equalities \eqref{1.3} and
\eqref{3.19} into account one obtains
\begin {equation}\label{5.22}
\im\, \l \cd (y,y)_\D=\tfrac 1 2 (||\wh \G_b y||^2-||\wh \G_a y||^2 )+ \im\,
(\G_{1a}y,\G_{0a}y )-\im\, (\G_{1b}y,\G_{0b}y ).
\end{equation}
It follows from \eqref{4.27a} that $P_{\wh H} \wh \G_b y=\wh\G_a y+i P_{\wh H}
h_0$ and, therefore,
\begin {equation}\label{5.23}
||P_{\wh H} \wh \G_b y||^2= ||\wh\G_a y||^2+||P_{\wh H} h_0||^2 + 2\im (\wh\G_a
y,P_{\wh H} h_0 ).
\end{equation}
Moreover, in view of \eqref{3.34a} one has
\begin {equation}\label{5.24}
P_{\cH_2'}\wh\G_b y=P_{\cH_2'}\wt\G_{0b} y, \qquad
(\G_{1b}y,\G_{0b}y)=(\G_{1b}y,\wt\G_{0b}y).
\end{equation}
Now by using first the decomposition \eqref{3.34.0} and then  the equality
\eqref{5.23} one gets
\begin{gather}
||\wh \G_b y||^2-||\wh \G_a y||^2 =||P_{\wh H} \wh \G_b
y||^2+||P_{\cH_2'}\wh\G_{b} y||^2- ||\wh\G_a y||^2=\label {5.25}\\
||P_{\wh H} h_0||^2+ 2\im (\wh\G_a y,P_{\wh H} h_0 )+||P_{\cH_2'}
\wt\G_{0b}y||^2. \nonumber
\end{gather}
Next, according to \eqref{5.1}
\begin {gather}
\G_{0a}y=P_H m_\tau(\l)h_0, \qquad \G_{1a}y=-P_H h_0, \label{5.26}\\
\wh\G_a y=P_{\wh H}m_\tau(\l)h_0 - \tfrac i 2 P_{\wh H}h_0\label{5.27}
\end{gather}
and substitution of \eqref{5.27} to \eqref{5.25} yields
\begin {equation}\label{5.28}
||\wh \G_b y||^2-||\wh \G_a y||^2 =2\im\, (P_{\wh H}m_\tau(\l)h_0, P_{\wh
H}h_0)+||P_{\cH_2'} \wt\G_{0b}y||^2.
\end{equation}
Moreover, by \eqref{5.26} one has
\begin {equation}\label{5.29}
\im\, (\G_{1a}y,\G_{0a}y )=\im\, (P_H m_\tau(\l)h_0, P_H h_0).
\end{equation}
Substituting now \eqref{5.28} and \eqref{5.29} to \eqref{5.22} and taking the
second equality in \eqref{5.24} into account we obtain
\begin {equation}\label{5.30}
\im\, \l \cd (y,y)_\D=\im\, (m_\tau(\l)h_0, h_0)- (\im\,(\G_{1b}y,
\wt\G_{0b}y)-\tfrac 1 2  ||P_{\cH_2'} \wt\G_{0b}y||^2).
\end{equation}
It follows from  \eqref{4.28} that $\{\wt\G_{0b}y, \G_{1b}y\}\in\tau_+(\l)$.
Therefore according to \cite[Proposition 4.3]{Mog06.1}
\begin {equation}\label{5.31}
\im\,(\G_{1b}y, \wt\G_{0b}y)-\tfrac 1 2  ||P_{\cH_2'} \wt\G_{0b}y||^2\geq 0.
\end{equation}
Moreover, in view of \eqref{3.0} one has
\begin {equation}\label{5.32}
(y,y)_\D=\int_\cI (\D(t)v_\tau(t,\l)h_0, v_\tau(t,\l)h_0)\,dt=((\int_\cI
v_\tau^*(t,\l)\D(t) v_\tau(t,\l)\, dt)h_0, h_0).
\end{equation}
Combining now  \eqref{5.31} and \eqref{5.32} with \eqref{5.30} we arrive at the
relation \eqref{5.21}.

It follows from \eqref{5.5} that the operator function $m_\tau(\cd)$ is
holomorphic in $\bC_+$. Moreover, the relation \eqref{5.21} shows that $\im\,
m_\tau(\l)\geq 0$ for all $\l\in\bC_+$. This and \eqref{5.7b} imply that
$m_\tau(\cd)$ is a Nevanlinna function.
\end{proof}
In the following proposition we show that a canonical $m$-function
$m_\tau(\cd)$ is the Weyl function of some symmetric extension of $\Tmi$.
\begin{proposition}\label{pr5.7}
Assume that $n_+=n_-$ and $\Pi=\bt$ is a decomposing boundary triplet
\eqref{3.35.2}, \eqref{3.42} for $\Tma$. Moreover, let $\tau\in\wt R^0(\cH_b)$
be a boundary parameter \eqref{2.22} and let $v_\tau(\cd,\l)\in\lo{H_0}$ be the
operator solution of Eq. \eqref{3.2} defined in Theorem \ref{th4.2}.

 Then: 1) The equalities
\begin{gather*}
\wt T=\{\{\wt y,\wt f\}\in\Tma : y(a)=0, \;\wh \G_a y=\wh \G_b y, \;
C_0\G_{0b}y+ C_1 \G_{1b}y =0 \},\\
\wt T^*=\{\{\wt y,\wt f\}\in\Tma : C_0\G_{0b}y + C_1 \G_{1b}y =0 \}
\end{gather*}
define a symmetric extension $\wt T$ of $\Tmi$ and its adjoint $\wt T^*$;

2) The collection $\Pi^n=\{H_0, \G_0^n, \G_1^n\}$ with the operators
\begin {equation}\label{5.37}
\G_0^n \{\wt y,\wt f\}=-\G_{1a}y+i (\wh \G_a -\wh \G_b) y, \quad \G_1^n \{\wt
y,\wt f\}=\G_{0a}y+ \tfrac 1 2 (\wh \G_a +\wh \G_b) y, \quad \{\wt y,\wt
f\}\in\wt T^*,
\end{equation}
is a boundary triplet for $\wt T^*$. Moreover, the $\g$-field $\g^n(\cd)$ and
Weyl function $M^n(\cd)$ of $\Pi^n$ are
\begin {equation}\label{5.38}
\g^n(\l)=\pi v_\tau(\l), \qquad M^n(\l)=m_\tau(\l), \qquad \l\in\CR;
\end{equation}

3) the following identity holds
\begin {equation}\label{5.39}
m_\tau(\mu)-m_\tau^*(\l)= (\mu-\ov\l)\int_\cI v_\tau^*(t,\l)\D(t)
v_\tau(t,\mu)\, dt, \quad \mu,\l\in\CR.
\end{equation}
This implies that for the canonical $m$-function $m_\tau(\cd)$ the inequality
\eqref{5.21} turns into the eq1uality, which holds for all $\l\in\CR$.
\end{proposition}
\begin{proof}
Clearly, we may assume that $\tau $ is given in the normalized form
\eqref{2.23}, in which case the operators
\begin{gather}
\ov\G_0\{\wt y,\wt f\}=\{-\G_{1a}y+i (\wh \G_a -\wh \G_b) y,\;C_0\G_{0b}y+ C_1
\G_{1b}y  \}(\in H_0\oplus\cH_b), \label{5.40}\\
\ov\G_1\{\wt y,\wt f\}=\{\G_{0a}y+ \tfrac 1 2 (\wh \G_a +\wh \G_b) y,\;
C_1\G_{0b}y-C_0\G_{1b}y  \}(\in H_0\oplus\cH_b),  \quad \{\wt y,\wt
f\}\in\Tma,\label{5.40a}
\end{gather}
form a decomposing boundary triplet $\ov\Pi=\{\cH, \ov\G_0, \ov\G_1\}$ for
$\Tma$. Let $\ov\g(\l)$ be the $\g$-field and
\begin {equation}\label{5.41}
\ov M(\l)=\begin{pmatrix} \ov m_0(\l) & \ov M_2(\l) \cr \ov M_3(\l) & \ov
M_4(\l)
\end{pmatrix}:H_0\oplus\cH_b\to
H_0\oplus\cH_b, \quad \l\in\CR,
\end{equation}
be the Weyl function of the triplet $\ov\Pi$. Assume also that $\ov
v_0(\cd,\l)\in \lo{H_0}$ is the operator solution of Eq. \eqref{3.2} defined in
Proposition \ref{pr4.1} (for the triplet $\ov\Pi$). Then $\ov v_0(t,\l)=
v_\tau(t,\l) $ and \eqref{4.10a} yields $\ov\g(\l)\up H_0=\pi v_\tau (\l)$.
Moreover, in view of \eqref{4.3a} one has $\ov m_0(\l)=m_\tau(\l), \;
\l\in\CR$. Applying now Proposition \ref{pr2.10a} to the triplet $\ov\Pi$ (with
$\dot\cH_0=\dot\cH_1=H_0$) we obtain the statements 1) and 2). Finally,
\eqref{5.39} follows from the identity \eqref{2.49} for the triplet $\Pi^n$ and
Lemma \ref{lem3.2}, 2) applied to the solution $v_\tau(\cd,\l)$.
\end{proof}
\begin{remark}\label{rem5.8}
Let $\Pi=\bt$ be a decomposing boundary triplet for $\Tma$, let $\tau\in\wt
R^0(\cH_b)$ be a boundary parameter given in the normalized form  \eqref{2.22},
\eqref{2.23} and let $\ov\Pi=\{\cH, \ov\G_0, \ov\G_1\}$ be a decomposing
boundary triplet \eqref{5.40}, \eqref{5.40a} for $\Tma$. It is easy to see that
$\Pi$ and $\ov\Pi$ are connected by
\begin {equation*}
\begin{pmatrix}\ov\G_0 \cr\ov\G_1 \end{pmatrix}=\begin{pmatrix} X_1 & X_2
\cr X_3 & X_4 \end{pmatrix}
\begin{pmatrix} \G_0 \cr \G_1 \end{pmatrix},
\end{equation*}
where $X_j\in [H_0\oplus\cH_b] $ are  defined as follows:
\begin {equation*}
X_1=\begin{pmatrix} I & 0 \cr 0 &C_0 \end{pmatrix}, \;\;\; X_2=\begin{pmatrix}
0 & 0 \cr 0 & -C_1 \end{pmatrix} , \;\;\; X_3=\begin{pmatrix} 0 & 0 \cr 0 &
C_1\end{pmatrix} , \;\;\; X_4=\begin{pmatrix} I & 0 \cr 0 &C_0 \end{pmatrix}.
\end{equation*}
Therefore according to \cite{DM95} the Weyl functions $M(\cd)$ and $\ov M(\cd)$
of the triplets $\Pi$ and $\ov\Pi$ respectively are connected via
\begin {equation} \label{5.45}
\ov M(\l)=(X_3+X_4 M(\l))(X_1+X_2 M(\l))^{-1}.
\end{equation}
By using the block representation \eqref{5.8} of $M(\l)$ one obtains $(X_1+X_2
M(\l))^{-1}=$

$=\begin{pmatrix}I & 0 \cr -C_1M_3 & C_0-C_1M_4
\end{pmatrix}^{-1}=\begin{pmatrix}I & 0 \cr  (C_0-C_1M_4)^{-1}C_1M_3
& (C_0-C_1M_4)^{-1}\end{pmatrix}$

\noindent and \eqref{5.45}, \eqref{5.41}  imply that $\ov m_0(\l)$ coincides
with the right hand side of \eqref{5.8a}. This and the equality $m_\tau(\l)=\ov
m_0(\l)$ obtained in the proof of Proposition \ref{pr5.7} yield \eqref{5.8a}.
Thus, for canonical $m$-functions $m_\tau(\cd)$ formula \eqref{5.8a} is a
simple consequence of the relation \eqref{5.45} for Weyl functions.

Note that considerations in this remark are inspired by \cite[Remark 86]{DM95},
where the Krein formula for resolvents was proved in a similar way.
\end{remark}
\section{Particular cases}\label{sect6}
\subsection{Symmetric systems with minimal deficiency indices}
It follows from \eqref{3.17c} that the minimally possible deficiency indices of
$\Tmi$ are
\begin {equation} \label{6.1}
n_+=\nu_+, \qquad n_-=\nu_-
\end{equation}
and \eqref{6.1} holds if and only if $\nu_{b+}=\nu_{b-}=0$ or, equivalently,
$[y,z]_b=0$ for all $y,z\in\dom\tma$. This implies that the system \eqref{3.1}
with minimal deficiency indices of $\Tmi$ is in \emph{Case 2} and by
\eqref{3.20} $\cH_b=\wh\cH_b=\{0\}$. Therefore in \eqref{3.36}
\begin {equation} \label{6.2}
H_0'=H, \qquad \cH_2'=\wh H
\end{equation}
and the decomposing boundary triplet for $\Tma$ takes the form
$\Pi_-=\{H_0\oplus H,\G_0,\G_1\}$, where $H_0$ is given by \eqref{3.17a} and
\begin {equation} \label{6.3}
\G_0\{\wt y,\wt f\}=\{-\G_{1a}y, i \wh\G_a y\}\,(\in H\oplus \wh H), \quad
\G_1\{\wt y,\wt f\}=\G_{0a}y\,(\in H), \quad \{\wt y,\wt f\}\in \Tma.
\end{equation}

In the case of minimal deficiency indices the symmetric extension $T$ defined
by \eqref{3.50} coincides with
\begin {equation} \label{6.4}
A_0(=\ker\G_0)=\{\{\wt y, \wt f\}\in\Tma: \, \G_{1a}y=\wh\G_a y=0\}
\end{equation}
(c.f. \eqref{3.53}). The unique generalized resolvent $R_0(\l)$ of $A_0$ is of
the form \eqref{4.33.0} and according to Theorem \ref{th4.2a} it is given by
the boundary value problem
\begin{gather}
J y'-B(t)y=\l\D (t)y +\D(t) f(t), \quad t\in\cI,\label{6.5}\\
\G_{1a}y =0, \;\;\;\l\in\bC_+; \qquad \G_{1a}y =0, \; \wh\G_a y=0,
\;\;\;\l\in\bC_-.\label{6.6}
\end{gather}

In view of Theorem \ref{th4.4} for each $\l\in\CR$ there exists a unique
operator solution $v(\cd,\l)\in\lo{H_0}$ of Eq. \eqref{3.2} such that
\begin {equation*}
\G_{1a}v(\l)=-P_H, \;\;\;\l\in\bC_+;\qquad \G_{1a}v(\l)=-P_H, \;\;\; i\wh \G_a
v(\l)= P_{\wh  H}, \;\;\;\l\in\bC_-,
\end{equation*}
and according to Definition \ref{def5.2} the $m$-function of the boundary value
problem \eqref{6.5}, \eqref{6.6} is given by
\begin {equation} \label{6.7}
m(\l)=(\G_{0a}+\wh\G_a)v(\l)+\tfrac i 2 P_{\wh H}, \quad \l\in\CR.
\end{equation}
In view of Proposition \ref{pr5.3} the $m$-function $m(\cd)$ can be also
defined by the relations
\begin {equation} \label{6.8}
v(t,\l):=\f (t,\l)m(\l)+\psi (t,\l)\in\lo{H_0}, \;\;\l\in\CR; \;\;\;\;\;
i\wh\G_a v(\l)=P_{\wh H}, \;\; \l\in\bC_-,
\end{equation}
where $\f(\cd,\l)$ and $\psi(\cd,\l)$  are the solutions of Eq. \eqref{3.2}
with the initial data \eqref{5.2}.

Next, Proposition \ref{pr4.3}, 2) and Theorem \ref{th5.5} yield the following
proposition.
\begin{proposition}\label{pr6.1}
Assume that $\Tmi$ has minimal deficiency indices \eqref{6.1},
$\Pi_-=\{H_0\oplus H,\G_0,\G_1\}$ is a decomposing boundary triplet \eqref{6.3}
for $\Tma$, $\g_-(\cd) $ is the $\g$-field and
\begin{gather}
M_+(\l)=(M(\l)\,:\, N_+(\l))^\top :H\to H\oplus \wh H, \quad\l\in\bC_+,
\label{6.9}\\
M_-(\l)=(M(\l)\,:\, N_-(\l)) :H\oplus \wh H\to H , \quad\l\in\bC_-,
\label{6.10}
\end{gather}
are the corresponding Weyl functions. Then $\g_-(\l)=\pi v(\l), \;\l\in\bC_-,$
and the following equalities hold
\begin{gather}
m(\l)=\begin{pmatrix} M(\l) & 0 \cr  N_+(\l) & \tfrac i 2 I_{\wh
\cH}\end{pmatrix}:\underbrace{H\oplus \wh H}_{H_0}\to
\underbrace{H\oplus \wh H}_{H_0}, \quad \l\in\bC_+\label{6.11}\\
m(\l)=\begin{pmatrix} M(\l)  &   N_-(\l)  \cr 0 & -\tfrac i 2 I_{\wh
H}\end{pmatrix}: \underbrace{H\oplus\wh H}_{H_0}\to \underbrace {H\oplus\wh
H}_{H_0}, \quad \l\in\bC_-.\label{6.12}
\end{gather}
Formulas \eqref{6.11} and \eqref{6.12} imply that the $m$-function $m(\cd)$
coincides with the function $\cM(\cd)$ corresponding to the decomposing
boundary triplet $\Pi_-$ (see \eqref{2.43} and \eqref{2.44}).
\end{proposition}
Combining the latter statement of Proposition \ref{pr6.1} with \eqref{2.44a}
and taking the equality $\g_-(\l)=\pi v(\l)$ and Lemma \ref{lem3.2}, 2) into
account we obtain the following corollary.
\begin{corollary}\label{cor6.2}
$m(\cd)$ is a Nevanlinna operator function satisfying the identity
\begin{gather}\label{6.13}
m(\mu)-m^*(\l)=(\mu-\ov\l)\int_\cI v^*(t,\l)\D(t)v(t,\mu)\,dt, \quad
\mu,\l\in\bC_-.
\end{gather}
\end{corollary}
\begin{remark}\label{rem6.3}
It follows from \eqref{6.8} and \eqref{6.11} that the Weyl function \eqref{6.9}
of the decomposing boundary triplet $\Pi_-$ for $\Tma$ is defined by the
relation
\begin {equation*}
\f(t,\l) M_+(\l) +\chi(t,\l) \in \lo{H}, \quad \l\in\bC_+,
\end{equation*}
where $\chi(t,\l)(\in [H,\bH])$  is a solution of Eq. \eqref{3.2} with the
initial data \eqref{1.38}. This and \eqref{1.37} imply that $M_+(\l)$ coincides
with the Titchmarsh - Weyl coefficient $M_{TW}(\l)$ introduced in
\cite{HinSch06} for symmetric systems \eqref{3.1} with minimal deficiency
indices $n_\pm$ under the additional assumption that the operator $X_a$ is of
the special  form \eqref{3.23}. Observe also that the square matrix $m(\l)$
defined by \eqref{6.11} appears in \cite[p.34]{HinSch06}, where it forms the
upper left block of the Nevanlinna matrix $\wh M(\l)$ (here $\wh M(\l)$ is the
Titchmarsh - Weyl  coefficient of the ''doubled'' system with equal deficiency
indices).
\end{remark}
\subsection{Symmetric systems with minimal equal deficiency indices}
It follows from \eqref{3.17c} and \eqref{3.17d} that the minimally possible
equal deficiency indices of $\Tmi$ are
\begin {equation}\label{6.15}
n_+= n_-=\nu_-
\end{equation}
and the equalities \eqref{6.15} hold if and only if $\nu_{b-}=0$ and
$\nu_{b+}=\wh\nu$. Therefore by \eqref{3.20} and Proposition \ref{pr3.3} the
(ordinary) decomposing boundary triplet for $\Tma$ in the case \eqref{6.15}
takes the form $\Pi=\{H_0, \G_0, \G_1\}$ with
\begin {equation}\label{6.16}
\G_0\{\wt y,\wt f\}=(-\G_{1a}+ i (\wh\G_a - \wh\G_b))y, \quad \G_1\{\wt y,\wt
f\}=(\G_{0a}+ \tfrac 1 2  (\wh\G_a + \wh\G_b))y, \quad \{\wt y,\wt f\}\in \Tma,
\end{equation}
where $\wh\G_b:\dom\tma\to \wh H$ is a surjective linear mapping such that
\begin {equation*}
[y,z]_b=i(\wh\G_b y, \wh\G_b z), \quad y,z\in\dom\tma.
\end{equation*}
Moreover, the extension \eqref{3.48} coincides with the self-adjoint extension
\begin {equation}\label{6.17}
A_0(=\ker \G_0)=\{\{\wt y,\wt f\}\in\Tma: \G_{1a}y=0, \; \wh\G_a y = \wh\G_b y
\}
\end{equation}
and the canonical resolvent $R_0(\l)=(A_0-\l)^{-1}, \; \l\in\CR,$ is defined by
the boundary value problem (c.f. \eqref{4.0.5})
\begin{gather}
J y'-B(t)y=\l\D (t)y +\D(t) f(t), \quad t\in\cI,\label{6.18}\\
\G_{1a}y =0,  \qquad \wh\G_a y=\wh\G_b y, \;\;\;\l\in\CR.\label{6.19}
\end{gather}
It follows from Theorem \ref{th4.2} that in the case \eqref{6.15} for each
$\l\in\CR$ there exists a unique operator solution $v(\cd,\l)\in\lo{H_0}$ of
Eq. \eqref{3.2} such that
\begin {equation*}
\G_1 v(\l)=-P_H, \qquad i(\wh\G_a y-\wh\G_b)v(\l)=P_{\wh H}, \quad \l\in\CR,
\end{equation*}
and according to Definition \ref{def5.2} and Proposition \ref{pr5.3} the
(canonical) $m$-function $m(\cd)$ of the boundary value problem \eqref{6.18},
\eqref{6.19} is defined by \eqref{6.7} or, equivalently, by the relations
\begin {equation*}
v(t,\l):=\f (t,\l)m(\l)+\psi (t,\l)\in\lo{H_0}, \qquad  i(\wh\G_a-\wh\G_b)
v(\l)=P_{\wh H}, \;\;\l\in\CR.
\end{equation*}
Finally, \eqref{4.10a} and Corollary \ref{cor5.4a} yield the following
proposition.
\begin{proposition}\label{pr6.4}
Let $\Tmi$ has minimal equal deficiency indices \eqref{6.15}, let $\Pi=\{H_0,
\G_0, \G_1 \}$ be the decomposing boundary triplet \eqref{6.16} for $\Tma$, let
$\g(\cd)$ and $M(\cd)$ be the corresponding $\g$-field and Weyl function
respectively and let $m(\cd)$ be the $m$-function of the boundary value problem
\eqref{6.18}, \eqref{6.19}. Then
\begin {equation*}
\g(\l)=\pi v(\l), \qquad M(\l)=m(\l), \qquad \l\in\CR,
\end{equation*}
and the identity \eqref{5.39} holds with $m_\tau(\l)=m(\l)$ and
$v_\tau(t,\l)=v(t,\l)$.
\end{proposition}
\subsection{Hamiltonian systems}
Recall that the system \eqref{3.1} ia called Hamiltonian if $\nu_+=\nu_-=:\nu$,
in which case the following assertions hold:

1) $\wh H=\{0\}$, so that $\bH=H\oplus H$ (with $\dim H=\nu$) and the signature
operator \eqref{1.3} takes the form \eqref{1.5a};

2) the deficiency indices of $\Tmi$ are $n_\pm=\nu+\nu_{b\pm}$ (c.f.
\eqref{3.17c});

3) the block representation \eqref{3.22} of the mapping $\G_a$ takes the form
\begin {equation*}
\G_a= \begin{pmatrix}\G_{0a}\cr \G_{1a} \end{pmatrix}:\AC\to H\oplus H.
\end{equation*}

In this subsection we let
\begin {equation*}
\a=\sign (\nu_{b+}-\nu_{b-})=\sign (n_+ - n_-).
\end{equation*}
Clearly, the Hamiltonian system \eqref{3.1} is in \emph{Case 1} when $\a\in
\{0,1\}$ and in \emph{Case 3} when $\a=-1$. Moreover, in view of \eqref{3.17d}
 $n_+=n_-$ if and only if
\begin {equation}\label{6.22}
\nu_{b+}=\nu_{b-}=:\nu_b.
\end{equation}

Assume that $\cH_b$ and $\wh\cH_b$ are Hilbert spaces and $\G_b$ is a
surjective linear map \eqref{3.18} such that \eqref{3.19} holds. Let
$\wt\cH_b=\cH_b\oplus\wh\cH_b$ and let $\wt\G_{0b}:\dom\tma\to \wt\cH_b$ be the
mapping \eqref{3.39b}.

The following proposition is implied by Proposition \ref{pr3.3}.
\begin{proposition}\label{pr6.14}
Let in the case of the Hamiltonian system \eqref{3.1} $\cH_0:=H\oplus\wt \cH_b,
\; \cH_1= H\oplus \cH_b$ and let $\G_j:\Tma\to\cH_j, \; j\in\{0,1\},$  be the
linear maps given for $\{\wt y,\wt f\}\in \Tma$ by
\begin {equation}\label{6.23}
\G_0\{\wt y,\wt f\}=\{-\G_{1a}y, \wt\G_{0b}y\}(\in H\oplus\wt\cH_b),\quad
\G_1\{\wt y,\wt f\}=\{\G_{0a}y, -\G_{1b}y\}(\in H\oplus\cH_b) .
\end{equation}
Then the collection $\Pi_\a=\bta$ is a decomposing boundary triplet for $\Tma$.

If in addition $n_+=n_-$, then $\Pi_\a$ turns into the ordinary boundary
triplet $\Pi=\bt$, where $\cH=H\oplus\cH_b$ and the operators $\G_0$ and $\G_1$
are given by \eqref{6.23}  with $\wt\cH_b=\cH_b$ and $\wt\G_{0b}=\G_{0b}$.
\end{proposition}
It follows from Propositions \ref{pr3.6} and \ref{pr3.8} that in the case of
the Hamiltonian system  the equality
\begin {equation}\label{6.25}
T=\{\{\wt y,\wt f\}\in\Tma: \G_{1a}y=0, \; \wt\G_{0b}y=\G_{1b}y =0\}
\end{equation}
defines a symmetric extension $T$ of $\Tmi$ with the deficiency indices
$n_\pm(T)=\nu_{b\pm}$. If in addition $n_+=n_-$, then in \eqref{6.25}
$\wt\G_{0b}=\G_{0b}$ and $T$ has equal deficiency indices
$n_+(T)=n_-(T)=\nu_b$.

Theorems \ref{th4.0.1} and \ref{th4.5} yield the following theorem.
\begin{theorem}\label{th6.5}
Let in the case of the Hamiltonian system \eqref{3.1} $\Pi_\a=\bta$ be a
decomposing boundary triplet \eqref{6.23} for $\Tma$. If $\pair\in \wt
R_\a(\wt\cH_b,\cH_b)$ is a collection \eqref{2.2}, then for every $f\in\lI$ the
boundary value  problem
\begin{gather}
J y'-B(t)y=\l \D(t)y+\D(t)f(t), \quad t\in\cI,\label{6.26}\\
\G_{1a}y=0, \quad \l\in\CR \label {6.27}\\
C_0(\l)\wt\G_{0b}y+C_1(\l)\G_{1b}y=0, \;\; \l\in\bC_+; \quad
D_0(\l)\wt\G_{0b}y+D_1(\l)\G_{1b}y=0, \;\; \l\in\bC_-, \label{6.27a}
\end{gather}
 has a unique solution $y(t,\l)=y_f(t,\l) $ and the equality \eqref{4.0.4}
defines a generalized resolvent $R(\l)=:R_\tau(\l)$ of $T$ (see \eqref{6.25}).
Conversely, for each generalized resolvent $R(\l)$ of $T$ there exists a unique
$\tau\in\wt R_\a(\wt\cH_b,\cH_b) $ such that $R(\l)=R_\tau(\l)$.

If in addition $n_+=n_-$ and $\Pi=\bt$ is  an ordinary decomposing boundary
triplet  for $\Tma$, then the statements of the  theorem  hold with the
Nevanlinna operator pairs $\tau\in\wt R(\cH_b)$ in the form \eqref{2.19} and
the  boundary conditions
\begin {equation}\label {6.28}
\G_{1a}y=0, \qquad  C_0(\l)\G_{0b}y +C_1(\l)\G_{1b}y=0, \quad \l\in\CR,
\end{equation}
instead  of \eqref{6.27} and \eqref{6.27a}. In this case $R_\tau(\l)$ is a
canonical resolvent of $T$ if and only if $\tau\in\wt R^0(\cH_b)$.
\end{theorem}

For Hamiltonian systems the operator solution $v_\tau(\cd,\l)$ takes on values
in $[H,\bH]$ in place of $[H_0, \bH]$ for general systems. More precisely, the
following theorem is implied by Theorems \ref{th4.2} and \ref{th4.7}.
\begin{theorem}\label{th6.6}
Let the assumptions of Theorem \ref{th6.5} be satisfied. Then for each
collection $\pair\in \wt R_\a(\wt\cH_b,\cH_b)$ defined by \eqref{2.2} there
exists a unique solution $v_\tau(\cd,\l)\in\lo{H}$ of Eq. \eqref{3.2} such that
\begin{gather}
\G_{1a}v_\tau(\l)=-I_H, \quad \l\in\CR, \label{6.29}\\
C_0(\l)\wt\G_{0b}v_\tau(\l)+C_1(\l)\G_{1b}v_\tau(\l)=0, \quad  \l\in\bC_+,
\label{6.30}\\
D_0(\l)\wt\G_{0b}v_\tau(\l)+D_1(\l)\G_{1b}v_\tau(\l)=0, \quad  \l\in\bC_-.
\label{6.31}
\end{gather}
If $n_+=n_-$, then $\tau\in\wt R(\cH_b )$ is given by \eqref{2.19} and the
conditions \eqref{6.29}--\eqref{6.31} take the form
\begin {equation}\label {6.32}
\G_{1a}v_\tau(\l)=-I_H, \quad C_0(\l) \G_{0b} v_\tau(\l)+C_1(\l)\G_{1b}
v_\tau(\l)=0, \quad  \l\in\CR.
\end{equation}
\end{theorem}

Next, a boundary parameter in the sense of Definition \ref{def5.1} is a
collection $\pair\in \wt R_\a(\wt\cH_b,\cH_b)$ of operator pairs \eqref{2.2}.
Moreover, in the case $n_+=n_-$ the boundary parameter is an operator pair
$\tau\in\wt R(\cH_b )$ given by \eqref{2.19}.

For Hamiltonian systems Definition \ref{def5.2} and  Proposition \ref{pr5.3}
take the following form.
\begin{definition}\label{def6.7}
The operator function $m_\tau(\cd):\CR\to [H]$ defined by
\begin {equation*}
m_\tau(\l)=\G_{0a} v_\tau(\l), \quad \l\in\CR,
\end{equation*}
is called the $m$-function corresponding to the boundary parameter $\tau$ or,
equivalently, to the boundary value problem \eqref{6.26}--\eqref{6.27a}.

If $n_+=n_-$, then $m_\tau(\cd)$ corresponds to the problem \eqref{6.26},
\eqref{6.28}. In this case the $m$-function $m_\tau(\cd)$ is called canonical
if $\tau\in\wt R^0(\cH_b)$.
\end{definition}
\begin{proposition}\label{pr6.7a}
The $m$-function $m_\tau(\cd)$ is a unique $[H]$-valued function such that
\begin {equation}\label{6.33}
v_\tau(t,\l):=\f (t,\l)m_\tau(\l)+\psi (t,\l)\in\lo{H}, \quad \l\in\CR,
\end{equation}
and the boundary conditions \eqref{6.30} and \eqref{6.31} are satisfied. If, in
addition, $n_+=n_-$ and $\tau\in\wt R^0(\cH_b)$ is given by \eqref{2.22}, then
the conditions  \eqref{6.30} and \eqref{6.31} take the form
\begin {equation}\label{6.34}
C_0\G_{0b}v_\tau(\l) + C_1\G_{1b}v_\tau(\l)=0, \quad \l\in\CR.
\end{equation}
In formula \eqref{6.33} $\f(\cd,\l)$ and $\psi(\cd,\l)$ are the operator
solutions of Eq. \eqref{3.2} with values in $[H,H\oplus H]$ and such that
\begin {equation*}
\f_a(\l)=\begin{pmatrix} I_{H}\cr 0\end{pmatrix}:H\to H\oplus H, \quad
\psi_a(\l)=\begin{pmatrix} 0\cr -I_H\end{pmatrix}:H\to H\oplus H, \quad
\l\in\bC.
\end{equation*}
\end{proposition}

The following theorem is implied by  Theorems  \ref{th5.4}, \ref{th5.6} and
Corollary \ref{cor5.4a}.
\begin{theorem}\label{th6.8}
Let  the conditions of Theorem \ref{th6.5} be satisfied and  let
$\tau_0=\{\tau_+,\tau_- \}$ be a boundary parameter \eqref{2.2} defined by
\eqref{4.0.6} in the case $n_+>n_-$ and by \eqref{4.61a} in the case $n_+<n_-$.
Assume also that the Weyl functions $M_\pm(\cd) $ have the block
representations:

--- in the case $n_+ > n_-$
\begin {equation*}
M_+(\l)=\begin{pmatrix} m_0(\l) & M_{2+}(\l) \cr M_{3+}(\l) &
M_{4+}(\l)\end{pmatrix}:H\oplus \wt\cH_b\to H\oplus \cH_b, \quad \l\in\bC_+;
\end{equation*}
--- in the case $n_+ < n_-$
\begin {equation*}
M_-(\l)=\begin{pmatrix} m_0(\l) & M_{2-}(\l) \cr M_{3-}(\l) &
M_{4-}(\l)\end{pmatrix}:H\oplus \wt\cH_b\to H\oplus \cH_b, \quad \l\in\bC_-.
\end{equation*}
Moreover, let $\pair$ be a boundary parameter \eqref{2.2}. Then:  1)
$m_0(\l)=m_{\tau_0}(\l)$; 2) in the case $n_+>n_-$ the equality \eqref{5.5}
holds; 3) in the case $n_+<n_-$ the equality
\begin {equation*}
m_\tau(\l)=m_0(\l)+M_{2-}(\l)(D_0(\l)-D_1(\l)M_{4-}(\l))^{-1}D_1(\l)M_{3-}(\l),
\quad\l\in\bC_-,
\end{equation*}
is valid. Moreover, if $n_+=n_-$ and the Weyl function $M(\l)$ has the block
representation \eqref{5.8} with $H$ instead of $H_0$, then for every boundary
parameter $\tau$ in the form \eqref{2.19} the equality \eqref{5.8a} is valid.
\end{theorem}

For the Hamiltonian system the minimally possible deficiency indices of $\Tmi$
are
\begin {equation}\label{6.35}
n_+=n_-=\nu(=\dim H),
\end{equation}
in which case the boundary triplet \eqref{6.23} for $\Tma$ takes the form
$\Pi=\{H,\G_0,\G_1\}$ with
\begin {equation}\label{6.36}
\G_0\{\wt y,\wt f\}=-\G_{1a}y,\quad \G_1\{\wt y,\wt f\}=\G_{0a}y, \quad \{\wt
y,\wt f\}\in\Tma,
\end{equation}
and the extension \eqref{6.25} turns into the self-adjoint extension
\begin {equation*}
A_0(=\ker \G_0)= \{\{\wt y,\wt f\}\in\Tma:\,\G_{1a}y=0 \}.
\end{equation*}
Moreover, the resolvent $R(\l)=(A_0-\l)^{-1}, \; \l\in\CR,$ of this extension
is defined by the boundary value problem  \eqref{6.26}, \eqref{6.27}.

It follows from Theorem \ref{th6.6} that in the case  \eqref{6.35} for each
$\l\in\CR$ there exists a unique operator solution $v(\cd,\l)\in\lo{H}$ of Eq.
\eqref{3.2} such that $\G_{1a}v(\l)=-I_H$. The $m$-function of the problem
\eqref{6.26}, \eqref{6.27} is defined by $m(\l)=\G_{0a}v(\l), \; \l\in\CR,$ or,
equivalently, by the relation
\begin {equation*}
v(t,\l):=\f (t,\l)m(\l)+\psi (t,\l)\in\lo{H}, \;\;\l\in\CR.
\end{equation*}
 Note in conclusion, that according to Proposition \ref{pr6.4}
$m(\cd)$ is the Weyl function of the boundary triplet \eqref{6.36}.
\begin{remark}\label{rem6.9}
Assume that $n_+=n_-$ and the operators $\G_{0b}$ and $\G_{1b}$ in \eqref{6.23}
are defined by  \eqref{3.33}. Moreover, let  $\tau\in\wt R^0(\bC^{\nu_b})$ be a
self-adjoint boundary parameter \eqref{2.22}. Then by using the matrix
representation of $C_0$ and $C_1$ one can express the boundary condition
\eqref{6.34} as
\begin {equation*}
[v_\tau(\cd,\l)h, \chi_j ]_b=0, \quad h\in H, \quad j=1\div \nu_b,
\end{equation*}
where $\chi_j\in\dom\tma$ are linear combinations of $\psi_k $ and $\t_k$. This
and Proposition \ref{pr6.7a} imply that the canonical $m$-function
$m_\tau(\cd)$ is the Titchmarsh - Weyl coefficient of the Hamiltonian system in
the sense of \cite{HinSch93}.
\end{remark}
\section{Applications to differential operators}
In this section we apply the above results to operators generated by a
differential expression $l[y]$ of  an odd order $r=2n+1$ defined on an interval
$\cI=[a,b\rangle\;(-\infty< a<b\leq \infty)$ with the regular endpoint $a$.
Such an expression is of the form \cite{Wei}
\begin {equation}\label{6.39}
l[y]=\tfrac 1 w \left\{ \sum_{k=0}^n  (-1)^k \left( \tfrac {i}{2}
[(q_{n-k}y^{(k)})^{(k+1)}+(q_{n-k} y^{(k+1)})^{(k)}] +
(p_{n-k}y^{(k)})^{(k)}\right) \right\},
\end{equation}
where $p_j(\cd), \; q_j(\cd)$ and $w(\cd)$ are real valued functions on $\cI$
such that : 1) $p_j(\cd)$ and $q_j(\cd)$ are smooth enough and $q_0(t)>0,\;
t\in\cI $; \; 2) $w(\cd)\in L_1([a,\b])$ for each $\b\in\cI$ and $w(t)>0$ a.e.
on $\cI$. Denote by $y^{[k]}(\cd), \; k=0\div 2n+1$ the quasi-derivatives of a
complex-valued function $y(\cd)\in Ac(\cI)$ and let $\dom l$ be the set of
functions $y(\cd)$ for which $l[y]:=y^{[2n+1]}$ makes sense
\cite{Wei,KogRof75,Rof69}. For each $y\in \dom l$ we let
\begin{gather*}
y^{(1)}(t):=\{y^{[k-1]}(t)\}_{k=1}^n (\in \bC^n), \qquad y^{(2)}(t):=
\{y^{[2n-k+1]}(t)\}_{k=1}^n (\in \bC^n),\\
\bold y(t)=\{y^{(1)}(t),\,\wh y(t),\,y^{(2)}(t)\}\, (\in \bC^{2n+1}),
\end{gather*}
where $\wh y(t)=q_0^{\frac 1 2}(t)y^{(n)}(t)$. Moreover, for each $m$-component
operator solution
\begin {equation}\label{6.40}
Y(t,\l)= (y_1(t,\l)\,:\; y_2(t,\l)\,:\;\dots\; y_m(t,\l) ): \bC^m\to \bC
\end{equation}
of the differential equation
\begin {equation}\label{6.41}
l[y]=\l y \quad (\l\in\bC)
\end{equation}
we put
\begin{gather*}
Y^{(j)}(t,\l)=(y_1^{(j)}(t,\l)\,:\; y_2^{(j)}(t,\l)\,:\;\dots\;:\,
y_m^{(j)}(t,\l)
): \bC^m\to \bC^n, \quad j\in \{1,2\},\\
\wh Y(t,\l)=(\wh y_1(t,\l)\,:\; \wh y_2(t,\l)\,:\;\dots\; :\, \wh y_m(t,\l) ):
\bC^m\to \bC,\\
\bold Y (t,\l)=(Y^{(1)}(t,\l)\,:\,\wh Y(t,\l)\,:\,Y^{(2)}(t,\l))^\top: \bC^m\to
\bC^{2n+1}, \quad t\in\cI.
\end{gather*}

Next assume that $L_r^2(\cI)$ is the Hilbert space of all complex-valued Borel
functions on $\cI$ such that $\int_\cI r(t)|f(t)|^2 \, dt<\infty$. It is known
\cite{Wei} that the expression \eqref{6.39} generates in $L_r^2(\cI)$ the
maximal operator $L$ and the minimal operator $L_0$. Moreover, $L_0$ is a
closed densely defined symmetric operator and $L_0^*=L$.

By using the results of \cite{KogRof75} one can easily prove the following
assertion.
\begin{assertion}\label{as6.10}
Let $l[y]$ be the expression \eqref{6.39} and let
\begin {equation*}
J_0=\begin{pmatrix} 0 & 0& -I_n \cr 0 & i I_1 & 0 \cr I_n & 0 & 0
\end{pmatrix}\in [\bC^n\oplus\bC\oplus\bC^n], \quad \D_0(t)=\begin{pmatrix}
w(t) & 0 \cr 0 & 0 \end{pmatrix} \in [\bC\oplus \bC^{2n}].
\end{equation*}
Then there exists a continuous  operator function $B_0(t)=B_0^*(t)\in
[\bC^{2n+1}]$ (defined in terms of $p_j$ and $q_j$) such that:

1) a complex-valued function $ y(\cd,\l)$  (operator function $Y(\cd,\l)$ of
the form \eqref{6.40}) is a solution of Eq. \eqref{6.41} if and only if $\bold
y (\cd,\l) $ (resp. $\bold Y(\cd,\l)$) is a solution of the symmetric system
\begin {equation}\label{6.43}
J_0y'-B_0(t)y=\l\D_0(t)y, \quad t\in\cI;
\end{equation}

2) the equality
\begin {equation}\label{6.44}
V\{y,f\}=\{\bold y, \dot f\}, \quad \{y,f\}\in \rm {gr}\, L,
\end{equation}
with $\dot f(t)=\{ f(t), \, 0,\; \dots ,\; 0\}(\in \bC^{2n+1})$ defines a
unitary operator  $V:\rm {gr} L\to \tma$, where $\tma$ is the maximal relation
in $\cL_{\D_0}^2 (\cI)$ for the system
\begin {equation}\label{6.45}
J_0y'-B_0(t)y=\D_0(t)f(t), \quad t\in\cI.
\end{equation}
Moreover, $V \,\rm{gr} L_0=\tmi$, where $\tmi$ is the minimal relation for the
system \eqref{6.45}.
\end{assertion}
Assertion \ref{as6.10} enables us to identify all the objects related to the
expression \eqref{6.39} with similar objects for the system \eqref{6.45}. In
particular, we assume that:\, 1) $\nu_{b+}$ and $\nu_{b-}$ are indices of
inertia of the bilinear form \eqref{3.7} for the system \eqref{6.45}; \, 2) the
linear map $\G_b$ in \eqref{3.18} is defined on $\dom L$, so that $\G_{0b}y,\;
\G_{1b}y $ and $\wh\G_b y$ are the singular boundary values of a function
$y\in\dom L$ and its quasi-derivatives  (c.f. Remark \ref{rem3.2a}). Moreover,
let $X_a\in [\bC^{2n+1}]$ be the operator such that $X_a^*J_0X_a=J_0$ and let
\begin {equation*}
\G_a = (\G_{0a}\,:\, \wh \G_a \,:\, \G_{1a})^\top : \dom l\to \bC^n\oplus
\bC\oplus\bC^n
\end{equation*}
be the linear map given by $\G_a y=X_a \bold y(a), \; y\in\dom l$.

Clearly, for the system \eqref{6.45} one has $\nu_- -\nu_+=1$. Hence this
system is either in \emph{Case 1} or in \emph{Case 3} and the reasonings in
Subsection \ref{sub2.3} take the following form:

1) \emph{Case 1}: $\;\nu_{b+}-\nu_{b-}\geq 1$ or, equivalently, $n_-(L_0)\leq
n_+(L_0)$. In this case we put $d=\nu_{b+}-\nu_{b-}-1, \; \cH_b=\bC^{\nu_{b-}}$
and $\wh \cH_b=\bC^d\oplus\bC$ (c.f. \eqref{3.34.0}), so that the operator
$\wh\G_b$ can be represented as
\begin {equation*}
\wh\G_b=(\wh\G_b' \,:\, \wh\G_b'')^\top:\dom L\to\bC^d\oplus\bC.
\end{equation*}
This implies that $\wt\cH_b=\bC^{\nu_{b-}}\oplus\bC^d$ and by \eqref{3.34a}
$\wt\G_{0b}=(\G_{0b}:\wh\G_b' )^\top$.

2)\emph{ Case 3}: $\;\nu_{b+}-\nu_{b-}\leq 0$ or, equivalently, $n_+(L_0)<
n_-(L_0)$. We put $d'=\nu_{b-}-\nu_{b+},\;\cH_b=\bC^{\nu_{b+}} $ and $\wh
\cH_b=\bC^{d'}$ (c.f. \eqref{3.39a}). Then $\wt\cH_b=
\bC^{\nu_{b+}}\oplus\bC^{d'}(=\bC^{\nu_{b-}})$ and in view of \eqref{3.39b} one
has
$$
\wt\G_{0b}=(\G_{0b}:\wh\G_b )^\top:\dom L\to\bC^{\nu_{b+}}\oplus\bC^{d'}.
$$

Now by using Assertion \ref{as6.10} one can easily reformulate all the previous
results for symmetric systems \eqref{3.1} in terms of the expression
\eqref{6.39}. For example Theorems \ref{th4.2}, \ref{th4.4} and \ref{th4.7}
take the following form.
\begin{theorem}\label{th6.11}
Let $\pair\in\wt R_\a (\wt\cH_b,\cH_b)$ be a collection of operator pairs
\eqref{2.2} with $\a=+1$ in the case $n_-(L_0)\leq n_+(L_0)$ and $\a=-1$ in the
case $n_+(L_0)< n_-(L_0)$. Then for each $\l\in\CR$ there exists a unique
$(n+1)$-component operator solution
\begin {equation}\label{6.47}
v_\tau(t,\l)= (v_1(t,\l)\,:\; v_2(t,\l)\,:\;\dots\; v_n(t,\l)\,:\,v_{n+1}(t,\l)
): \bC^n\oplus\bC\to \bC
\end{equation}
of Eq. \eqref{6.41} such that $v_j(\cd,\l)\in L_r^2(\cI), \; j=1\div (n+1),$
and the following boundary conditions are satisfied:

1) in the case $n_-(L_0)\leq n_+(L_0)$
\begin{gather}
\G_{1a}v_\tau(\l)=(-I_{\bC^n}\,:\, 0):\bC^n\oplus\bC\to \bC^n, \quad
\l\in\CR,\label{6.48}\\
i(\wh \G_a-\wh \G_b'')v_\tau(\l)=(0_{\bC^n}\,:\, I_\bC):\bC^n\oplus\bC\to \bC,
\quad \l\in\CR,\label{6.49}\\
C_0(\l)\wt\G_{0b}v_\tau(\l)+C_1(\l)\G_{1b}v_\tau(\l)=0,
\;\;\l\in\bC_+,\label {6.50} \\
D_0(\l)\wt\G_{0b}v_\tau(\l)+D_1(\l)\G_{1b}v_\tau(\l)=0, \;\;\l\in\bC_-.\label
{6.51}
\end{gather}

2) in the case $n_+(L_0)< n_-(L_0)$ -- the conditions \eqref{6.48},
\eqref{6.50}, \eqref{6.51} and
\begin {equation}\label{6.52}
i\wh \G_a v_\tau(\l)=(0_{\bC^n}\,:\, I_\bC):\bC^n\oplus\bC\to \bC, \quad
\l\in\bC_-.
\end{equation}

Here the linear map $v_\tau(\l):\bC^n\oplus \bC\to \gN_\l(L_0)$ is given by
\begin {equation*}
(v_\tau(\l)h)(t)=v_\tau(t,\l)h=\sum_{j=1}^{n+1}v_j(t,\l)h_j, \quad
h=\{h_1,\,h_2, \, \dots , \, h_n, \, h_{n+1} \}\in\bC^n\oplus \bC,
\end{equation*}
so that $\G_a v_\tau(\l)=X_a \bold v_\tau(a,\l)$.
\end{theorem}

Next, the $m$-function $m_\tau(\cd)$ of the expression $l[y]$ corresponding to
the boundary parameter $\tau\in\wt R_\a (\wt\cH_b,\cH_b)$ (with the same $\a$
as in Theorem \ref{th6.11}) is defined as the $m$-function of the system
\eqref{6.45}. In view of Proposition \ref{pr5.3} this means  that
$m_\tau(\cd):\CR\to [\bC^{n+1}] $ is a unique operator function such that for
every $\l\in\CR$ the $(n+1)$-component operator solution \eqref{6.47} of Eq.
\eqref{6.41} given by
\begin {equation*}
v_\tau(t,\l):=Y_1(t,\l)m_\tau(\l)+Y_2(t,\l)
\end{equation*}
possesses the following properties: 1)$v_j(\cd,\l)\in L_r^2(\cI)$ for all
$j=1\div (n+1)$, 2) $v_\tau(\cd,\l)$ satisfies the boundary conditions
\eqref{6.49}-\eqref{6.51} in the case $n_-(L_0)\leq n_+(L_0)$ and
\eqref{6.50}-\eqref{6.52} in the case $n_+(L_0)< n_-(L_0)$. Here $Y_1(t,\l)$
and $Y_2(t,\l)$ are the $(n+1)$-component operator solutions of Eq.
\eqref{6.41} with the initial data $X_a \bold Y_1(a,\l)= \begin{pmatrix} I_n &
0\cr 0 & I_1 \cr 0 & 0
\end{pmatrix}\in [\bC_n\oplus\bC, \bC_n\oplus\bC\oplus\bC^n]$ and $ X_a \bold Y_2(a,\l)
= \begin{pmatrix} 0 & 0\cr 0 & -\tfrac i 2 I_1 \cr -I_n & 0
\end{pmatrix}\in [\bC_n\oplus\bC, \bC_n\oplus\bC\oplus\bC^n].$

According to results in Section \ref{sect5} $m_\tau(\cd)$ is a Nevanlinna
operator function, which in the case $n_+(L_0)< n_-(L_0)$ has the triangular
form \eqref{5.16} and \eqref{5.17} with $H=\bC^n$ and $\wh H=\bC$.

The reformulations of other results in Sections \ref{sect5} and \ref{sect6} to
the case of the $m$-functions of the differential expression \eqref{6.39} are
left to the reader.


\begin{thebibliography}{DHS}
\bibitem{Atk}
F.V. Atkinson, \textit{Discrete and continuous boundary problems}, Academic
Press, New York, 1963.

\bibitem{BHSW10}
J. Behrndt, S. Hassi, H. de Snoo, and R. Wiestma, \textit{Square-integrable
solutions and Weyl functions for singular canonical systems}, Math. Nachr.
\textbf{284} (2011), no. 11-12, 1334--1383.


\bibitem{Cal39}
J. Calkin, Abstract symmetric boundary conditions, Trans. Amer. Math. Soc.
\textbf{45}, 369-442 (1939).

\bibitem{DM00}
V.A. Derkach, S. Hassi, M.M. Malamud, and H.S.V. de Snoo,\textit { Generalized
resolvents of symmetric operators and admissibility}, Methods of Functional
Analysis and Topology \textbf{6} (2000), no.~3 , 24--55.



\bibitem{DM91}
V.A.~Derkach and M.M.~Malamud, \textit {Generalized resolvents and the boundary
value problems for Hermitian operators with gaps}, J. Funct. Anal. \textbf{95}
(1991),1--95.


\bibitem {DM95}
V.A. ~Derkach and M.M. ~Malamud, \textit {The extension theory of Hermitian
operators and the moment problem}, J. Math. Sciences, 73, (1995), no 2,
141-242.

\bibitem{DLS88}
A. Dijksma, H. Langer, and H.S.V. de Snoo, \textit{Hamiltonian systems with
eigenvalue depending boundary conditions}, Oper. Theory Adv. Appl. \textbf{35}
(1988), 37--83.

\bibitem{DLS93}
A. Dijksma, H. Langer, and H.S.V. de Snoo, \textit{Eigenvalues and pole
functions of Hamiltonian systems with eigenvalue depending boundary
conditions}, Math. Nachr.  \textbf{161} (1993), 107--153.


\bibitem{DunSch}
N. Dunford and J.T. Schwartz, \textit{Linear operators. Part2. Spectral
theory}, Interscience Publishers, New York-London, 1963.

\bibitem{Eve72}
W.N. Everitt, \textit{Integrable-square, analytic solutions of odd
order,fomally symmetric, ordinary differential equations}, Proc. London Math.
Soc. (3)\textbf{25} (1972), 156--182.

\bibitem{EveKum76.1}
W.N. Everitt and V. Krishna Kumar, \textit{On the Titchmarsh-Weyl theory of
ordinary symmetric differential expressions 1 : The general theory }, Nieuw
Arch. voor Wiskunde (3) \textbf{24} (1976), 1--48.

\bibitem{EveKum76.2}
W.N. Everitt and V. Krishna Kumar, \textit{On the Titchmarsh-Weyl theory of
ordinary symmetric differential expressions 2 : The odd-order case}, Nieuw
Arch. voor Wiskunde (3) \textbf{24} (1976), 109--145.


\bibitem{Ful77}
Ch. T. Fulton, \textit{Parametrizations of Titchmarsh's $m(\l)$-functions in
the limit circle case}, Trans. Amer. Math. Soc. \textbf{229} (1977), 51--63.

\bibitem{GK}
I. Gohberg and M.G. Krein, \textit{Theory and applications of Volterra
operators in Hilbert space}, Transl. Math. Monographs, 24, Amer. Math. Soc.,
Providence, R.I., 1970

\bibitem{Gor66}
M.L.~Gorbachuk,\, \textit{On spectral functios of a differential equation of
the second order with operator-valued coefficients}, Ukrain. Mat. Zh.
\textbf{18} (1966), no.~2 , 3--21.

\bibitem{GorGor}
V.I.~Gorbachuk and M.L.~Gorbachuk, \textit{Boundary problems for
differential-operator equations}, Kluver Acad. Publ., Dordrecht-Boston-London,
1991. (Russian edition: Naukova Dumka, Kiev, 1984).

\bibitem{HSW00}
S. Hassi, H.S.V. de Snoo, and H. Winkler, \textit{Boundary-value problems for
two-dimensional canonical systems}, Integral Equations Operator Theory
\textbf{36} (2000), 445--479.

\bibitem{HinSha82}
D.B. Hinton and J.K. Shaw, \textit{Parameterization of the $M(\l)$ function for
a Hamiltonian system of limit circle type}, Proc. Roy. Soc. Edinburgh Sect. A
\textbf{93} (1982/83), no. 3-4,  349-360.

\bibitem{HinSch93}
D.B. Hinton and A. Schneider, \textit{On the Titchmarsh-Weyl coefficients for
singular S-Hermitian systems I}, Math. Nachr. \textbf{163} (1993), 323--342.

\bibitem{HinSch06}
D.B. Hinton and A. Schneider, \textit{Titchmarsh-Weyl coefficients for odd
order linear Hamiltonian systems}, J. Spectral Mathematics
\textbf{1}(2006),1-36.

\bibitem{Hol85}
A.M.~Khol'kin,\, \textit{Description of selfadjoint extensions of differential
operators of an arbitrary order on the infinite interval in the absolutely
indefinite case},  Teor. Funkcii Funkcional. Anal. Prilozhen. \textbf{44}
(1985), 112--122.

\bibitem{Kac50}
I.S. Kats, \textit{On Hilbert spaces generated by monotone Hermitian
matrix-functions}, Khar'kov. Gos. Univ. Uchen. Zap. \textbf{34} (1950),
95-113=Zap.Mat.Otdel.Fiz.-Mat. Fak. i Khar'kov. Mat. Obshch. (4)\textbf{22}
(1950), 95-113.

\bibitem{Kac03}
I.S. Kats, Linear relations generated by the canonical differential equation of
phase dimension 2, and eigenfunction expansion, St. Petersburg Math. J.
\textbf{14}, 429--452 (2003).

\bibitem{KogRof75}
 V.I.~Kogan and F.S.~Rofe-Beketov,\, \textit{On square-integrable solutions of
 symmetric systems of  differential equations of arbitrary order}, Proc. Roy.
Soc. Edinburgh Sect. A \textbf{74} (1974/75), 5--40.

\bibitem{Kov83}
I.V. Kovalishina, \textit {Analytic theory of a class of interpolation
problems}, Izv. Akad. Nauk SSSR Ser. Mat.,  \textbf{47} (1983), no. 3,
 455–-497


\bibitem{Kra89}
A.M. Krall, \textit{$M(\l)$-theory for singular Hamiltonian systems with one
singular endpoint}, SIAM J. Math. Anal. \textbf{20} (1989), 664--700.

\bibitem{Kum82}
V. Krishna Kumar, \textit{On the Titchmarsh-Weyl theory of ordinary symmetric
odd-order differential expressions and a direct convergence theorem},
Quaestiones Math. \textbf{5} (1982),  165-185.



\bibitem{LT82}
H. Langer and B. Textorius, \textit{$L$-resolvent matrices of symmetric linear
relations with equal defect numbers; appliccations to canonical differential
relations}, Integral Equations Operator Theory  \textbf{5} (1982), 208--243.


\bibitem{LesMal03}
M. Lesch and M.M. Malamud, \textit{On the deficiency indices and
self-adjointness of symmetric Hamiltonian systems}, J. Differential Equations
\textbf{189} (2003), 556--615.

\bibitem {Mal92}
 M. M. ~Malamud, \textit{On the formula of generalized resolvents of a
nondensely defined Hermitian operator}, Ukr. Math. Zh. \textbf{44}(1992), no.
12, 1658-1688.

\bibitem {MalMog02}
M. M. Malamud and V. I. Mogilevskii, ''Krein type formula for canonical
resolvents of dual pairs of linear relations'', Methods of Funct. Anal. and
Topology \textbf{8}, No.4 (2002), 72-100.

\bibitem{Mog06.1}
 V.I.Mogilevskii,\, \textit{Nevanlinna type families of linear relations and
the dilation theorem}, Methods  Funct. Anal.  Topology \textbf{12} (2006),
no.~1, 38--56.

\bibitem{Mog06.2}
 V.I.Mogilevskii,\, \textit{Boundary triplets and Krein type resolvent formula
 for symmetric operators
with unequal defect numbers}, Methods  Funct. Anal.  Topology \textbf{12}
(2006), no.~3, 258--280.

\bibitem{Mog07}
 V.I.Mogilevskii,\, \textit{Description of spectral functions of differential
 operators  with arbitrary deficiency indices}, Math. Notes \textbf{81}
(2007), no.~4, 553--559.

\bibitem {Mog09.1}
V.I.Mogilevskii,\, \textit{Boundary triplets and Titchmarsh - Weyl   functions
of differential operators with arbitrary  deficiency indices },Methods  Funct.
Anal.  Topology \textbf{15} (2009), no.~3, 280--300.

\bibitem {Mog09.2}
V.I.Mogilevskii,\,\textit{Fundamental solutions of boundary-value problems and
resolvents of differential operators }, Ukr. Math. Bul. \textbf{6} (2009), no.
4, 487--525.

\bibitem{Mog11mz}
V.I.Mogilevskii, \textit{Description of generalized resolvents and
characteristic matrices if differential operators in terms of the boundary
parameter}, Math. Notes \textbf{90} (2011), no. 4, 548--570.

\bibitem{Mog11}
V.I.Mogilevskii,


\bibitem{Orc}
B.C. Orcutt, \textit{Canonical differential equations}, Dissertation,
University of Virginia, 1969.

\bibitem{Rof69}
F.S.~Rofe-Beketov,\, \textit{Self-adjoint extensions of differential operators
in the space of vector-valued functions}, Teor. Funkcii Funkcional. Anal.
Prilozhen. \textbf{8} (1969), 3--23.

\bibitem{Wei}
J. Weidmann, \textit{Spectral theory of ordinary differential operators},
Lecture notes in mathematics, \textbf{1258}, Springer-Verlag,  Berlin, 1987.




\end{thebibliography}
\end{document}